\documentclass[a4paper, 11pt]{article} 

\usepackage[T1]{fontenc}

\usepackage{lmodern}

\usepackage{amssymb, mathtools}
\usepackage{graphicx}

\usepackage{pgfplots}
\pgfplotsset{compat=1.11}
\usepackage{mathrsfs}
\usepackage{mathtools}
\usepackage{dsfont}
\usepackage{amsthm}
\usepackage{enumerate}
\usepackage{color}
\usepackage{verbatim}
\usepackage[margin=2.3cm, includefoot, footskip=30pt]{geometry}
\usepackage{array}
\newcolumntype{e}{>{\displaystyle}r @{\,} >{\displaystyle}c @{\,} >{\displaystyle}l}

\usepackage{todonotes}

\usepackage{amsopn, amscd}

\usepackage{microtype}

\usepackage{babel}

\usepackage{xcolor}
\usepackage[colorlinks]{hyperref}
\hypersetup{final}
\hypersetup{
    colorlinks,
    linkcolor={blue!40!black},
    citecolor={blue!40!black},
    urlcolor={blue!80!black}
}
\usepackage{nameref}
\usepackage[all]{hypcap}

\theoremstyle{plain}
\newtheorem{Theorem}{Theorem}[section]

\newtheorem{Lemma}[Theorem]{Lemma}
\newtheorem{Proposition}[Theorem]{Proposition}

\theoremstyle{definition}

\newtheorem{Remark}[Theorem]{Remark}

\newcommand{\E}{\mathbb{E}}
\newcommand{\R}{\mathbb{R}}
\newcommand{\N}{\mathbb{N}}
\newcommand{\Q}{\mathbb{Q}}
\newcommand{\Z}{\mathbb{Z}}

\renewcommand{\P}{\mathbb{P}}
\newcommand{\range}[2]{[\![#1,#2]\!]}

\numberwithin{equation}{section}

\begin{document}

\title{The Directed Spanning Forest: coalescence versus dimension}
\author{Tom Garcia-Sanchez\footnote{IMT Nord Europe, Lille - France}}
\maketitle

\begin{abstract}
For $p\in[1,\infty]$, the $\ell^p$ \emph{directed spanning forest} (DSF) of dimension $d\geq 2$ is an oriented random geometric graph whose vertex set is given by a \emph{homogeneous Poisson point process} $\mathcal N$ on $\R^d$ and whose edges consist of all pairs $(x, y)\in\mathcal N^2$ such that $y$ is the closest point to $x$ in $\mathcal N$ for the $\ell^p$ distance among points with a strictly larger $e_d$ coordinate. First introduced by Baccelli and Bordenave in 2007 in the case $p=d=2$, this graph has a natural \emph{forest} structure: it is a collection of \emph{unrooted directed trees}. In this work, we study the number of disjoint trees in the $\ell^p$ DSF for arbitrary dimensions $d\geq2$ and various values of $p\in [1,\infty]$. We prove that for $p\in\{1, 2,\infty\}$, the graph is almost surely a tree when $d=3$, and consists of infinitely many disjoint trees when $d\geq 4$. Additionally, we show that for all $p\in[1,\infty]$, the DSF in dimension $2$ is almost surely a tree and, under appropriate diffusive scaling, converges weakly to the Brownian web, generalizing the result of Coupier, Saha, Sarkar, and Tran (2021) for $p=2$, and thus also covering the recent advance of Pal and Saha (2025) for $p=\infty$. Although these results were expected from a heuristic point of view, and the main strategies and tools were largely understood, analyzing the DSF beyond dimension 2 presented a significant challenge. Notably, in the absence of \emph{planarity}, which plays a crucial role in existing arguments, delicate and innovative techniques were required to manage the complex geometric dependencies of the model. We develop substantially new ideas to handle arbitrary dimension $d\geq2$ and various values of $p\in [1,\infty]$ within a unified framework. In particular, we introduce a novel stochastic domination argument that allows us to compare the fully dependent model with a simplified version in which the geometric constraints are ignored. This relation had long been sought, but until now had never been established.
\end{abstract}


\paragraph{Acknowledgments.}
The author warmly thanks his PhD advisors, David Coupier and Viet Chi Tran, for their continuous guidance and support throughout this work. This research was partly supported by the ANR project \emph{GrHyDy} (ANR-20-CE40-0002), the CEFIPRA project \emph{Directed random networks and their scaling limits} (No.\ 6901), the CNRS RT \emph{MAIAGES} (Action 2179), and by a doctoral fellowship from ENS Paris.

\section{Introduction}

In this paper, we study the $\ell^p$ \emph{Directed Spanning Forest} (DSF) in $\R^d$, for arbitrary fixed dimension $d\geq 2$ and exponent $p\in[1, \infty]$. The model is constructed as follows. We consider a \emph{homogeneous Poisson point process} $\mathcal N$ of unit intensity in $\R^d$. For each $x\in\R^d$, we define $\Psi(x)$ as the point of $\mathcal N$ that realizes
    \[\|\Psi(x)-x\|=\min \{\|y-x\|:y\in\mathcal N,~y\cdot e_d>x\cdot e_d\},\]
where $\|\cdot\|$ denotes the $\ell^p$ norm, $e_d$ denotes the $d$-th canonical basis vector, and the dot represents the standard scalar product. It is a.s.\ well defined as there are almost surely no two distinct points of $\mathcal N$ at the same distance from $x$. In words, $\Psi(x)$ is the closest point to $x$ in $\mathcal N$ having strictly larger $e_d$ coordinate, for the $\ell^p$ distance. The $\ell^p$ DSF of dimension $d$ is then defined as the \emph{random geometric graph} whose vertex set is $\mathcal N$ and set of directed edges is given by
    \[\{(x,\Psi(x)):x\in\mathcal N\}.\]
By construction, all vertices have \emph{out-degree} one and for all edge $(u, v)$ we have $(v-u)\cdot e_d>0$, preventing the presence of \emph{cycles} in the graph. This confers on the DSF a \emph{forest} structure, justifying the nomenclature. In particular, the graph consists of a collection of \emph{unrooted directed trees}. A natural question is then: \emph{how many such trees are there?} The main goal of this work is to answer this question. To this end, it is convenient to consider \emph{trajectories} of the DSF, i.e.\ infinite sequences $(v_n)_{n\geq 0}$, with $v_0\in\mathcal N$ and $v_{n+1}=\Psi(v_n)$ for all $n\geq 0$. The DSF consists of a single tree if and only if all trajectories \emph{coalesce}, i.e.\ eventually coincide. Similarly, there are at least $k\geq 1$ disjoint trees in the DSF if and only if we can find $k$ trajectories that do not coalesce. A realization of the DSF in a finite rectangle with some trajectories is presented in Figure~\ref{fig_dsf}.

\begin{figure}[h]
    \centering
    \includegraphics[width=0.33\linewidth]{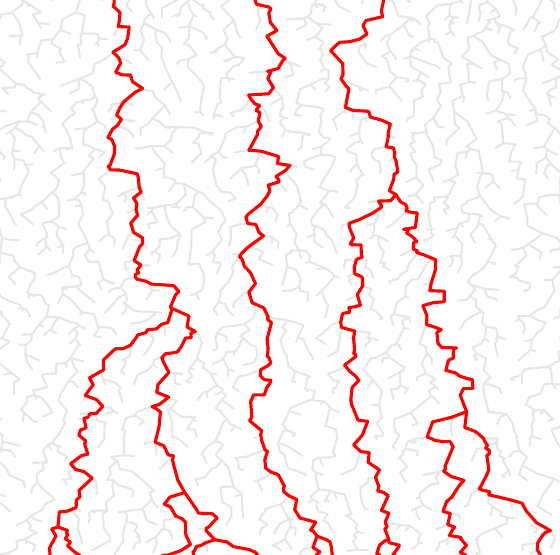}
    \caption{Simulation of a realization of the DSF in a finite box with $p=d=2$. We represent the edges of the DSF in gray and some trajectories in red. Note that $e_d$ is pointing to the top in this picture, so all the edges are directed upward.}
    \label{fig_dsf}
\end{figure}

We now give an overview of the origin of the model and existing results. The DSF was introduced in 2007 by Baccelli and Bordenave \cite{BacceliBordenave} in the \emph{planar Euclidean} case $p=d=2$, as an auxiliary model in their analysis of the \emph{Radial Spanning Tree} (RST), another geometric random graph originally introduced to model telecommunication networks. However, the DSF quickly proved to be of independent interest, due to both its natural geometric construction and the challenging conjectures it raises, some of which were already mentioned in the original paper.

In the original setting $p=d=2$, the first conjecture posed was that the DSF is almost surely a tree. It was confirmed in 2013 by Coupier and Tran \cite{Dsf2dtree} using an efficient percolation technique that relies on the planar structure, namely the so-called Burton and Keane argument (see \cite{BurtonKeane}). In the same setting, they also showed that there is almost surely no \emph{bi-infinite} path, i.e.\ infinite sequence $(v_n)_{n\in\Z}$ with $v_0\in\mathcal N$ and $\Psi(v_n)=v_{n+1}$ for all $n\in\Z$.

In 2021, Coupier, Saha, Sarkar, and Tran \cite{Dsf2d} proved a strong conjecture of Baccelli and Bordenave concerning the planar Euclidean DSF: under appropriate \emph{diffusive scaling}, it \emph{converges weakly} to the \emph{Brownian web}. Notably, their proof involves sharp estimates for the coalescence time of two trajectories. Intuitively, the Brownian web can be thought of as a collection of coalescing $1$-dimensional \emph{Brownian motions} starting at every point of space and time $\R\times\R$. We refer to \cite{schertzer2017brownian} for an overview of the Brownian web.

More recently, Pal and Saha \cite{DsfInfinity} extended the coalescing time estimates of \cite{Dsf2d} to the $\ell^\infty$ DSF in dimension $2$, and proved in particular that it is almost surely a tree. 

It is also worth mentioning that the DSF has been studied in hyperbolic space by Flammant \cite{flammant2024directed}, who showed that the corresponding graph is almost surely a tree in any dimension by exploiting the distinctive properties of hyperbolic geometry.

In fact, the DSF belongs to a broader family of random geometric directed trees and forests, sometimes referred to as \emph{drainage networks}. This family includes, in particular, an $\ell^1$ discrete analog of the DSF \cite{DiscreteDsf}, other models on $\Z^d$ such as \cite{gangopadhyay2004random} and \cite{athreya2008random}, as well as different constructions on a Poisson point process in Euclidean space, such as \cite{Ferrari2004}. We will discuss these related works in more detail below.\\

To summarize, while the behavior of the Euclidean DSF is well understood in the planar setting, little is known in higher dimensions. The main result of this paper is the following theorem, which addresses the fundamental question of how many disjoint trees the DSF contains in arbitrary dimension $d$, for different values of the exponent $p$.

\begin{Theorem}
    \label{thm_main}
    If $d=2$ and $p\in[1, \infty]$ or $d=3$ and $p\in\{1, 2, \infty\}$, the DSF is almost surely a tree. Furthermore, if $d\geq 4$ and $p\in\{1, 2, \infty\}$, the DSF consists almost surely of an infinite collection of disjoint trees.
\end{Theorem}

This result was expected and aligns with a natural heuristic that will guide our analysis throughout the paper. The trajectories of the DSF roughly behave like centered random walks in $\mathbb{R}^{d-1}$, where the $e_d$ direction plays the role of time. Moreover, two distant trajectories evolve approximately independently, so their difference behaves like a centered random walk in $\mathbb{R}^{d-1}$ when far from the origin. Since coalescence amounts to asking whether this difference walk eventually hits $0$, the question essentially reduces to the recurrence or transience of such walks. Recalling that under suitable assumptions, centered random walks are recurrent in dimensions one and two but transient in higher dimensions (see, e.g., \cite[Chapter~4]{Durett}), one may conjecture that DSF trajectories coalesce if and only $d-1\leq2$, that is, if and only if $d\leq3$.

In fact, such a dimensional dichotomy has been established for several related models. It is important to note, however, that these other constructions typically incorporate simplifying features by design, such as discrete settings (see \cite{gangopadhyay2004random, athreya2008random, DiscreteDsf}) or construction rules that allow one to directly exhibit a suitable Markovian structure (see \cite{Ferrari2004}). Although the overall strategy for proving such results is now well understood, implementing it without model-specific conveniences remains challenging. In particular, while the DSF is among the simplest and most natural constructions, significant technical obstacles have long prevented its analysis in arbitrary dimension.

The relatively recent work \cite{Dsf2d} marked a significant advance in the study of the DSF in dimension 2 and provided a framework that appeared promising for tackling the problem beyond the planar case. In this work, we generalize this framework to higher dimensions. This problem had remained open because, once planarity is lost, existing difficulties are exacerbated, crucial arguments break down, and entirely new issues arise. Notably, a key argument in \cite{Dsf2d}, which was crucial for ensuring suitable progress along the $e_d$ direction during the exploration of trajectories, does not extend to higher dimensions. Note that in discrete settings, such progress is typically guaranteed to be at least one unit, but in general it is a genuine challenge. This major obstacle led us to introduce a new and rather elegant stochastic domination argument (see Theorem \ref{thm_sto_dom_poisson}), allowing us to compare the fully dependent model to a simplified one in which geometric constraints are ignored. Beyond its central role in our proof, this argument constitutes a significant contribution, as such a stochastic domination result had long been sought but had never previously been established. Nevertheless, a key step in its proof relies on a specific geometric result, which surprisingly fails to hold for all pairs of exponent and dimension, notably because non-Euclidean balls are not isotropic (see Remark \ref{remark_fail}). This limitation explains why Theorem \ref{thm_main} does not cover the full range of exponents $p$ when $d\geq 3$, even though it is natural to expect that the dimensional dichotomy should hold in full generality. To the best of our knowledge, there is no direct way to extend our ideas to overcome this restriction. In particular, there is no apparent monotonicity in $p$ that could be exploited to fill the gap.\\

As a secondary result, we establish the following theorem concerning the dimension $d=2$. This extends the main result of \cite{Dsf2d}, which treated the case $p=2$. In fact, the case $p<\infty$ can be obtained through relatively straightforward adaptations of \cite{Dsf2d}, and the case $p=\infty$ was already nearly settled in \cite{DsfInfinity}. Still, our framework provides a unified and simplified treatment of all values of $p$. We defer the formal definition of convergence under appropriate diffusive scaling toward the Brownian web to a dedicated section.

\begin{Theorem}
    \label{thm_main2}
    For all $p\in[1, \infty]$, the $\ell^p$ DSF of dimension $2$ converges weakly, under appropriate diffusive scaling, to the Brownian web.
\end{Theorem}

Let us emphasize that studying the DSF presents significant challenges due to the complex geometrical dependencies that arise in its construction. For any $h\in\R$ we define the \emph{open half-space} of points whose $e_d$ coordinates is strictly greater than $h$ with
    \[\mathbb H^+(h)\coloneqq\{x\in\R^d:x\cdot e_d>h\}.\]
For any $x\in\R^d$ and $r\in\R_+$, we denote by $B(x, r)$ the \emph{$\ell^p$ open ball} of radius $r$ centered at $x$, and we define the \emph{open half-ball}
    \[B^+(x, r)\coloneqq B(x, r)\cap\mathbb H^+(x\cdot e_d).\]
In the DSF, for any $x\in\mathcal N$, the presence of the edge $(x,\Psi(x))$ imposes that $B^+(x,\|\Psi(x)-x\|)$ contains no points of $\mathcal N$. Since this half-ball typically overlaps $\mathbb H^+(\Psi(x)\cdot e_d)$, the edge $(\Psi(x), \Psi^2(x))$ is strongly constrained by the edge $(x,\Psi(x))$, where $\Psi^2\coloneqq\Psi\circ\Psi$. This heavy dependence phenomenon prevents us from directly exhibiting any Markov property while exploring iteratively a single trajectory in the DSF. Similarly, different trajectories in the DSF are significantly correlated, as the presence of one trajectory imposes strong geometric constraints for $\mathcal N$. An illustration of these complex dependencies is provided in Figure \ref{fig_dependencies}.

\begin{figure}[h]
    \centering
    \includegraphics[height=0.15\linewidth]{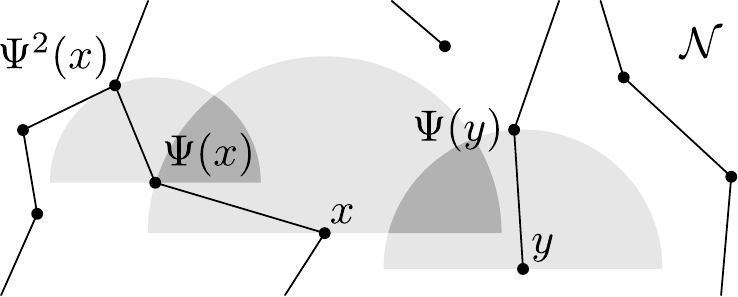}
    \caption{Visualization of the geometrical constraints imposed during the construction of the DSF for $p=d=2$. In gray the half-balls that must be empty of points.}
    \label{fig_dependencies}
\end{figure}

To tackle the complex correlations of the model, Coupier, Saha, Sarkar, and Tran \cite{Dsf2d} proposed a method to decompose the trajectories of the DSF into quasi-independent blocks of random size. This method is at the core of their proof that the planar Euclidean DSF is almost surely a tree, and converges, under appropriate diffusive scaling, to the Brownian web. In fact, their arguments extend with minimal modification to the case $d=2$ and $p\in[1,\infty)$. A delicate adaptation to the singular case $p=\infty$ was recently carried out in \cite{DsfInfinity}. %However, substantial mathematical barriers prevent a straightforward extension to higher dimensions, where planarity is lost, existing difficulties are exacerbated, and entirely new challenges arise.

\subsection*{Structure of the paper and strategies of the proofs}

In the first part of this work, we build on the ideas of \cite{Dsf2d}, introducing a unified framework that extends the analysis beyond the planar Euclidean setting. This involves new techniques and arguments to address the challenges posed by higher dimensions and general values of $p$, including the singular case $p=\infty$. Notably, we establish a novel stochastic domination result of independent interest. In the second part, we combine this framework with random walk techniques to prove Theorem~\ref{thm_main}, and show how Theorem~\ref{thm_main2} follows from the results of \cite{Dsf2d} within our setting.\\

The main notation is introduced in section \ref{section_exploration}, where we define the \emph{joint exploration process} of trajectories in the DSF. Following the approach of \cite{Dsf2d}, we use \emph{history sets} to record the geometrical constraints imposed during the exploration and to provide a useful characterization of the process. 

In Section \ref{section_good_steps}, we focus on controlling the size of the history sets, demonstrating that their size will return infinitely often below a specific threshold. To achieve this, we develop a novel approach, which includes proving a new stochastic domination argument. This is where the condition $p\in\{1, 2, \infty\}$ of Theorem \ref{thm_main} will arise.

In Section \ref{section_renewal}, we use the control over the size of the history set to show that certain \emph{renewal events} occur at random times, as in \cite{Dsf2d}. This allows us to properly formalize the heuristic that trajectories of the DSF behave like random walks. A key ingredient of this section will be to control the fluctuations of the trajectories between two consecutive \emph{renewal times}. While inspired by the ideas in \cite{Dsf2d}, the proofs we develop are substantially different and rely on new techniques tailored to our more general setting.

In Section \ref{section_coalescence}, we focus on establishing the coalescence part of Theorem \ref{thm_main}, i.e., that when either $d=2$ and $p\in[1, \infty]$, or $d=3$ and $p\in\{1, 2, \infty\}$, the DSF is almost surely a tree. This will be achieved by combining the renewal decomposition with the use of \emph{Lyapunov functions}.

In Section \ref{section_non_coalescence}, we address the high-dimensional regime of Theorem \ref{thm_main}, showing that when $p\in\{1, 2, \infty\}$ and $d\geq 4$, the DSF consists almost surely of infinitely many disjoint trees. The proof relies on the renewal decomposition along with a general result by Kesten \cite{KestenErickson} on the rate of escape to infinity of random walks.

Finally, in Section \ref{section_BW}, we prove Theorem \ref{thm_main2} by defining and establishing the convergence toward the Brownian web. This step follows directly from the results in \cite{Dsf2d}, which we apply without modification.

\section{The exploration process}

\label{section_exploration}

In this section, we define the central process that will allow us to explore trajectories in the DSF iteratively. Recall that $d\geq2$ and $p\in[1, \infty]$ are arbitrary fixed. From now on, we fix $k\geq1$ as the number of trajectories we will consider. We also fix deterministic \emph{starting points} $(\mathbf u_i)_{i\in\range{1}{k}}\in(\R^d)^k$ for our trajectories having the same $e_d$ coordinate. The \emph{exploration process}
    \[((g_n(\mathbf u_i))_{i\in\range{1}{k}})_{n\geq 0}\]
is constructed as follows. We set $g_0(\mathbf u_i)\coloneqq\mathbf u_i$ for each $i\in\range{1}{k}$. Now let $n\geq 0$ and assume that the process is constructed until step $n$. We define
    \[m_n\coloneqq\min\{g_n(\mathbf u_i)\cdot e_d~:~i\in\range{1}{k}\},\]
and we fix $x_n\in\{g_n(\mathbf u_i)~:~i\in\range{1}{k}\}$ such that $x_n\cdot e_d=m_n$, which will represent the moving vertex of the step. In case of multiple possibilities, we choose one deterministically. Then, for each $i\in\range{1}{k}$ we set
    \[g_{n+1}(\mathbf u_i)\coloneqq\begin{cases}
        \Psi(x_n) & \text{if $g_n(\mathbf u_i)=x_n$}\\
        g_n(\mathbf u_i) & \text{otherwise.}
    \end{cases}\]
In words, at step $n$, we explore one step forward along the trajectory whose current vertex $x_n$ has the lowest $e_d$ coordinate. As explained in the introduction, this exploration process is highly non-Markovian, as the next step depends in a complex way on the entire previously explored region. Each edge $(x,\Psi(x))$ revealed during the exploration imposes that the corresponding half-ball $B^+(x,\|\Psi(x)-x\|)$ must be empty of Poisson points, and such constraints can influence arbitrarily distant edges.\\

To record the geometrical constraints imposed by the observation of the first steps of the exploration process, we define the \emph{history sets} $(H_n)_{n\geq 0}$ as follows. For each $n\geq 0$, we set
    \begin{equation}
        \label{eq_history_set}
        H_n\coloneqq \mathbb H^+(m_n)\cap\bigcup_{i=0}^{n-1}B^+(x_i,\|\Psi(x_i)-x_i\|)
    \end{equation}
In words, for all $n\geq0$, $H_n$ corresponds to the intersection of the half-space $\mathbb H^+(m_n)$ with the union of all the half-balls that must be empty of Poisson points due to the edges revealed during the $n$ first steps of exploration. By construction, for all $n\geq0$, 
    \[\mathcal N\cap H_n=\emptyset,\]
and
    \[\mathcal N\cap\partial H_n=\{g_n(\mathbf u_i):i\in\range{1}{k}\}\cap\mathbb H^+(m_n),\]
where $\partial H_n$ denotes the boundary of $H_n$. An illustration of the construction of the exploration process together with history sets is provided in Figure \ref{fig_exploration_process}.

\begin{figure}[h]
    \centering
    \includegraphics[height=0.33\linewidth]{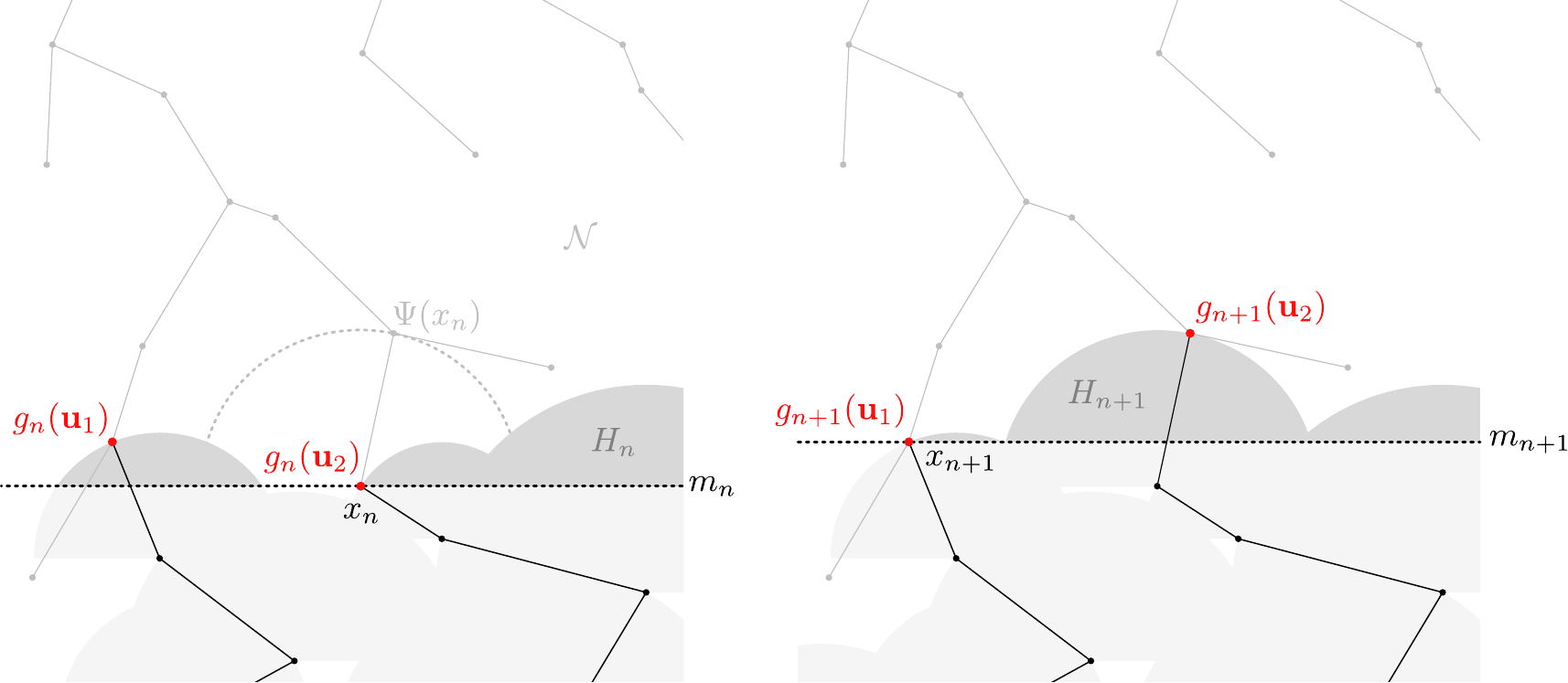}
    \caption{Illustration of the construction of the exploration process with history sets.}
    \label{fig_exploration_process}
\end{figure}

We now define a natural filtration associated with the exploration process. For all $n\geq 0$, let
    \[\mathcal F_n\coloneqq\sigma\big(\mathcal N\cap\mathbb H^-(m_n),~\mathcal N\cap\overline{H}_n\big),\]
where $\overline{H}_n$ denotes the closure of $H_n$ and $\mathbb H^-(h)\coloneqq\R^d\setminus\mathbb H^+(h)$ for any $h\in\R$. By construction, the \emph{exploration process with history} $((g_n(\mathbf u_i))_{i\in\range{1}{k}}, H_n)_{n\geq 0}$ is adapted to the filtration $(\mathcal F_n)_{n\geq0}$. It is worth noting that, for technical reasons to be addressed later, this filtration contains more information than what is strictly used by the exploration process.

We can now state the key proposition of this section, which justifies the importance of the history sets. It formalizes the intuitive idea that after $n\geq 0$ steps of exploration, no information about the Process point $\mathcal N$ has been revealed outside the region $\mathbb H^-(m_n)\cup\overline{H}_n$. implying that the process behaves as if $\mathcal N$ were independently \emph{resampled} outside the region $\mathbb H^-(m_n)\cup H_n$ (the closure can be omitted as $\partial H_n$ has a.s.\ measure $0$). This crucial property was previously used implicitly in \cite{BacceliBordenave}, and explicitly in \cite{Dsf2d}. We omit its proof, which is straightforward and technical.

\begin{Proposition}
    \label{prop_resampling}
    Let $\mathcal N'$ denote an independent copy of $\mathcal N$. Then for any $n\geq 0$, conditionally on $\mathcal F_n$, the random variable
        \[\mathcal N\setminus\big(\mathbb H^-(m_n)\cup \overline{H}_n\big)\]
    is distributed as $\mathcal N'\setminus(\mathbb H^-(m_n)\cup H_n)$.
\end{Proposition}

As a consequence of Proposition \ref{prop_resampling}, the exploration process with history $((g_n(\mathbf u_i))_{i\in\range{1}{k}}, H_n)_{n\geq 0}$ is a Markov chain. Its distribution can be characterized as follows. Fix $n\geq 0$, and let $y_{n+1}$ denote the closest point to $x_n$ in $\mathcal N\setminus\overline{H}_n$ having strictly larger $e_d$-coordinate. Conditional on $\mathcal F_n$, $y_{n+1}$ is distributed as the closest point to $x_n$ in $\mathcal N'\setminus H_n$ having strictly large $e_d$-coordinate. Then, $\Psi(x_n)$ is obtained as the closest point to $x_n$ in $\{y_{n+1}\}\cup\{g_n(\mathbf u_i):i\in\range{1}{k}\}\setminus\{x_n\}$.

This convenient characterization of the conditional law of the process, with \emph{resampling} outside the explored region, will play a central role in the analysis of the exploration with history.

\section{Good steps}

\label{section_good_steps}

As the history sets encode the geometrical constraints imposed during the exploration process, it is crucial to control their size to understand how far the interactions can extend into the future. To this end, for all $n\geq0$, we define
    \[M_n\coloneqq\max(\{m_n\}\cup\{x\cdot e_d~:~x\in H_n\}),\]
and we set
    \[L_n\coloneqq M_n-m_n.\]
By construction, $L_n$ represents the height of the history set $H_n$ with respect to the $e_d$-coordinate while $M_n$ and $m_n$ represent the $e_d$ level of its top and bottom, respectively. It is worth noticing that both processes $(m_n)_{n\geq0}$ and $(M_n)_{n\geq 0}$ are non-decreasing, as during exploration, both the top and bottom of the history sets move upward at each step. This section aims to prove the following theorem, which provides control over $(L_n)_{n\geq0}$. It gives in particular that there exists a threshold $\kappa>0$ such that $L_n\leq\kappa$ for infinitely many $n\geq 0$. In this work, we use $\alpha$ and $\beta$ to denote constants whose values may vary from one statement to another.

\begin{Proposition}
    \label{thm_good_steps}
    Assume $d=2$ or $p\in\{1, 2, \infty\}$. For all $\kappa>0$ and $R>0$, we define a sequence $(\tau_n)_{n\geq 0}=(\tau_n(\kappa, R))_{n\geq 0}$ of $(\mathcal F_n)_{n\geq 0}$ stopping times by induction with $\tau_0\coloneqq0$ and
        \[\tau_{n+1}\coloneqq\inf\{j>\tau_n:L_j\leq\kappa,~ m_j-m_{\tau_n}\geq\kappa+R\}\]
    for each $n\geq0$. Then there exist $\kappa=\kappa(k, d, p)>0$, $\alpha=\alpha(k, d, p)>0$ and $\beta=\beta(k, d, p)>0$ such that for all $R>0$, the variables $(\tau_n)_{n\geq 0}$ are almost surely finite, and all $n,\ell\geq0$,
        \[\P(\tau_{n+1}-\tau_n>\ell~\vert~\mathcal F_{\tau_n})\leq \alpha e^{R-\beta \ell}.\]
    The steps associated to the indices $(\tau_n)_{n\geq 0}$ are then called \emph{good steps}.
\end{Proposition}

Note that the condition $d=2$ or $p\in\{1, 2, \infty\}$ serves as a concise shorthand we will frequently use to refer to the two case $d=2$ with $p\in[1, \infty]$, and $d\geq3$ with $p\in\{1, 2,\infty\}$.

The strategy of the proof of Theorem \ref{thm_good_steps} is inspired by \cite{Dsf2d}, but we develop an original approach that overcomes the key mathematical obstacles that previously prevented generalization both to the singular $\ell^\infty$ norm and, more importantly, beyond the planar setting. Our framework handles both challenges in a unified way. A central contribution is a novel stochastic domination that allows us to compare any arbitrary step of the exploration process with the idealized case where the history set is empty, that is, where no constraints from the past are present. We begin by establishing an upper bound on the growth of $(M_n)_{n\geq0}$ in Subsection \ref{subsection_upper_control}. In Subsection \ref{subsection_stochastic_domination}, we introduce the novel stochastic domination result, which will allow us to derive a lower bound on the growth of $(m_n)_{n\geq0}$ in Subsection \ref{subsection_lower_control}. Finally, in Subsection \ref{subsection_proof_good_steps}, we combine these ingredients to prove Theorem \ref{thm_good_steps}.

\subsection{Upper control for the increase of $(M_n)_{n\geq0}$}

\label{subsection_upper_control}

In this subsection, we focus on establishing an upper bound on the growth of the sequence $(M_n)_{n\geq0}$. We begin by proving the following proposition, which forms the core of the argument. In the second part, we use this result to derive some exponential moment bounds.

\begin{Proposition}
    \label{prop_upper_control}
    For arbitrary dimension $d\geq 2$ and parameter $p\in [1, \infty]$, there exist a function $f=f_{d, p}:(\R_+)^2\to[0, 1]$ such that 
        \begin{itemize}
            \item[(a)] For each $n\geq 0$ and $\ell\in\R_+$,
                \[\P\big(M_{n+1}-M_n\geq \ell~\vert~\mathcal F_n\big)\leq f(\ell, L_n).\]
            \item[(b)] There exist $c=c(d, p)>0$ such that for all $(\ell,L)\in(\R_+)^2$,
                \[f(\ell, L)\leq \exp(-c\ell^d).\]
            \item[(c)] For each $\ell\in(0,\infty)$, $L\mapsto f(\ell, L)$ is non-increasing and  $\lim_{L\to\infty}f(\ell, L)=0$.
        \end{itemize}
\end{Proposition}

In words, Proposition \ref{prop_upper_control} states that for any $n\geq0$, conditionally on $\mathcal F_n$, the increment $M_{n+1}-M_n$ has a distribution with a rapidly decaying tail. Moreover, each tail probability is arbitrarily small when $L_n$ is sufficiently large. The key idea to prove Proposition \ref{prop_upper_control} comes from the following lemma, which provides an upper bound on the conditional probability of interest in terms of the measure of some region.

\begin{Lemma}
    \label{lemma_proba_area}
    For all $n\geq0$ and $\ell\in\R_+$, we have
        \[\P\big(M_{n+1}-M_n\geq\ell~\vert~\mathcal F_n\big)\leq\exp{(-|B^+(x_n,L_n+\ell)\setminus H_n|)},\]
    where we recall that $x_n$ is the vertex that will move at the next step, and $|\cdot|$ denotes the Lebesgue measure.
\end{Lemma}

\begin{proof}
    First, observe that by construction of the exploration process, 
        \[M_{n+1}=\min\{M_n,~m_n+\|\psi(x_n)-x_n\|\}.\]
    Therefore, if $M_{n+1}-M_n\geq \ell$, it must be that $\|\Psi(x_n)-x_n\|\geq M_n+\ell-m_n=L_n+\ell$, i.e.\
        \[\mathcal N\cap B^+(x_n,L_n+\ell)=\emptyset,\]
    and in particular, $\mathcal N\cap B^+(x_n,~L_n+\ell)\setminus\overline{H}_n=\emptyset$. This shows that
        \[\P\big(M_{n+1}-M_n\geq \ell~\vert~\mathcal F_n\big)\leq\P(\mathcal N\cap B^+(x_n,~L_n+\ell)\setminus \overline{H}_n=\emptyset~|~\mathcal F_n).\]
    Finally, since 
        \[\mathcal N\cap B^+(x_n,~L_n+\ell)\setminus\overline{H}_n\subset\mathcal N\setminus\big(\mathbb H^-(m_n)\cup\overline{H}_n\big),\]
    applying Proposition \ref{prop_resampling}, we have that, conditionally on $\mathcal F_n$, the number of points of $\mathcal N$ in $B^+(x_n,L_n+\ell)\setminus \overline{H}_n$ is distributed as a Poisson random variable of parameter $|B^+(x_n,~L_n+\ell)\setminus H_n|$, implying the result.
\end{proof}

Now it remains to derive a suitable lower bound on the measure of the set $B^+(x_n, L_n+\ell)\setminus H_n$ for all $n\geq0$ and $\ell\in\R_+$ to prove Proposition \ref{prop_upper_control}. This reduces the problem to a deterministic geometric one, where we must leverage the inherent properties of the exploration process. We proceed by treating the cases $p=\infty$ and $p<\infty$ separately. We begin with the case $p=\infty$, which is addressed by the following lemma.

\begin{Lemma}
    \label{lemma_area_linfty}
    If $p=\infty$, then for all $n\geq 0$ and $\ell\in\R_+$, we have 
        \[|B^+(x_n,L_n+\ell)\setminus H_n|\geq \ell(L_n+\ell)^{d-1}.\]
\end{Lemma}

\begin{proof}
    We start by setting
        \[R_n(\ell)\coloneqq B^+(x_n, L_n+\ell)\cap\mathbb H^+(M_n).\]
    By construction, since $H_n\subset\mathbb H^-(M_n)$, we have
        \[R_n(\ell)\subset B^+(x_n,L_n+\ell)\setminus H_n.\]
    We refer to Figure \ref{fig_cone} for a graphic illustration. Finally, observing that we can write
        \[R_n(\ell)=x_n+(-L_n+\ell, L_n+\ell)^{d-1}\times(L_n, L_n+\ell),\]
    we deduce $|R_n(\ell)|=\ell(L_n+\ell)^{d-1}$, implying the result.
\end{proof}

\begin{Remark}
    It could be possible to adapt the argument used in the proof of Lemma \ref{lemma_area_linfty} to derive a suitable similar result for all $p>1$, using that the top of the $\ell^p$ balls are locally flat. However, this argument breaks down for $p=1$ due to the singularity of the $\ell^1$ ball.
\end{Remark}

We now turn to the case $p<\infty$. The following lemma is the key to our argument. Informally, it says that any ball contained in $\mathbb H^-(1)$, whose center lies in $\mathbb H^-(0)$ and which avoids the origin, must be disjoint from a ball centered at $e_d$ with universal radius. This geometric property, which relies on the rounded shape of the $\ell^p$ balls when $p<\infty$, will allow us, after a suitable translation and scaling, to construct a nontrivial region that systematically avoids the history set. This argument breaks down for $p=\infty$, due to the flatness of the $\ell^\infty$ ball.

\begin{Lemma}
    \label{lemma_empty_ball}
    Assume $p<\infty$. Let $c\in\mathbb H^-(0)$ and $r>0$ such that $0\notin B(c,r)$ and $B(c, r)\subset\mathbb H^-(1)$, then
        \[B(e_d, \alpha_p)\cap B(c, r)=\emptyset,\quad\text{where}\quad \alpha_p\coloneqq2^{\frac{1}{p}}-1\in(0,1).\]
    In particular, 
        \[B\Big(z_p, \frac{\alpha_p}{2}\Big)\subset B^+(0,1)\setminus B(c, r),\quad\text{where}\quad z_p\coloneqq\left(1-\frac{\alpha_p}{2}\right)e_d.\]
\end{Lemma}

\begin{proof}
    First, note that since $c\in\mathbb H^-(0)$, we can write $c=c_0-\beta e_d$ with $c_0\cdot e_d=0$ and $\beta\geq 0$. Since $B(c, r)\subset\mathbb H^-(1)$ and $c_0+e_d\in\mathbb \partial \mathbb H^-(1)$, we must have
        \begin{equation}
            \label{eq_rbeta}
            r\leq \|e_d+c_0-c\|=\|e_d+\beta e_d\|=\beta+1
        \end{equation}
    Moreover, since $0\notin B(c,r)$, we must have $r\leq\|c\|$ and thus
        \begin{equation}
            \label{eq_deltarp}
            \begin{split}
                \|e_d-c\|^p-r^p&\geq\|e_d-c\|^p-\|c\|^p\\
                &=\|c_0\|^p+(\beta+1)^p-\|c_0\|^p-\beta^p\\
                &=(\beta+1)^p-\beta^p.
            \end{split}
        \end{equation}
    Now, we can write
        \begin{equation}
            \label{eq_deltar}
            \begin{split}
                \|e_d-c\|-r&=(r^p+\|e_d-c\|^p-r^p)^{\frac{1}{p}}-(r^p)^{\frac{1}{p}}\\
                &\geq \big((\beta+1)^p+\|e_d-c\|^p-r^p\big)^{\frac{1}{p}}-(\beta+1)\\
                &\geq\big(2(\beta+1)^p-\beta^p\big)^{\frac{1}{p}}-(\beta^p)^{\frac{1}{p}}-1\\
                &\geq\big(2(\beta+1)^p\big)^{\frac{1}{p}}-(2\beta^p)^{\frac{1}{p}}-1\\
                &=2^{\frac{1}{p}}-1\\
                &=\alpha_p,
            \end{split}
        \end{equation}
    where the third line is obtained from \eqref{eq_deltarp} and the first and last inequality comes from concavity of $z\mapsto z^{1/p}$ with \eqref{eq_rbeta} and $\beta^p\geq 0$ as for all $0\leq a\leq b$ and $c\geq 0$ we have $b^{1/p}-a^{1/p}\geq (b+c)^{1/p}-(a+c)^{1/p}$. Finally let $x\in B(e_d,\alpha_p)$. From triangle inequality and \eqref{eq_deltar}, we can write
        \begin{equation*}
            \|x-c\|\geq\|e_d-c\|-\|x-e_d\|\geq r+\alpha_p-\alpha_p=r,
        \end{equation*}
    implying $x\notin B(c,r)$. This shows that $B(e_d,\alpha_p)\cap B(c, r)=\emptyset$. Let $x\in B(z_p,\frac{\alpha_p}{2})$. We have
        \[(z_p-x)\cdot e_d\leq\|x-z_p\|<\frac{\alpha_p}{2},\]
    hence $x\cdot e_d\geq z_p\cdot e_d-\frac{\alpha_p}{2}=1-\alpha_p\geq 0$, which shows $x\in\mathbb H^+(0)$. Moreover,
        \[\|x\|\leq\|x-z_p\|+\|z_p\|<\frac{\alpha_p}{2}+\Big(1-\frac{\alpha_p}{2}\Big)=1,\]
    hence $x\in B^+(0, 1)$. The proof is complete.
\end{proof}

\begin{Remark}
    The radius $\alpha_p$ given by the Lemma \ref{lemma_empty_ball} is optimal. One can check that in the case $c=e_1$ and $r=1$, we have
        \[e_1+\frac{e_d-e_1}{\|e_d-e_1\|}\in \overline B\left(e_d, \alpha_p\right)\cap\overline B(c, r).\]
\end{Remark}

We can now derive a suitable lower bound on the measure of $B^+(x_n,L_n+\ell)\setminus H_n$ for all $n\geq 0$ and $l\in\R_+$. The argument is essentially an adaptation of \cite[Lemma 3.2]{Dsf2d}, carefully generalized to all $d\geq 2$ and $p\in[1,\infty)$.

\begin{Lemma}
    \label{lemma_area_lp}
    For all $p<\infty$, there exists $c=c(d, p)>0$ such that for all $n\geq0$ and $\ell\in\R_+$, we have
        \[|B^+(x_n,L_n+\ell)\setminus H_n|\geq c\max\{\ell, L_n\}^d.\]
\end{Lemma}

\begin{proof}
    First, observing that since $H_n$ consists of the intersection of $\mathbb H^+(m_n)$ with an union of balls avoiding $x_n$, included in $\mathbb H^-(M_n)$ with centers in $\mathbb H^-(m_n)$, performing a translation by $-x_n$ and a scaling by ${L_n}^{-1}$, we can apply Lemma \ref{lemma_empty_ball} to get that
        \begin{equation}
            \label{eq_bottom_ball}
            B\left(c_n, \frac{\alpha_p}{2}L_n\right)\subset B^+(x_n, L_n+\ell)\setminus H_n,\quad\text{where}\quad c_n\coloneqq x_n+L_n z_p.
        \end{equation}
    We refer to Figure \ref{fig_cone} for a graphic illustration. Taking the Lebesgue measure in \eqref{eq_bottom_ball}, we get
        \begin{equation}
            \label{eq_area1}
            |B^+(x_n, L_n+\ell)\setminus H_n|\geq\left(\frac{\alpha_p}{2}L_n\right)^d|B(0,1)|.
        \end{equation}
    Additionally, we have that
        \[B^+(x_n+L_n e_d,\ell)\subset B^+(x_n, L_n+\ell)\setminus H_n.\]
    Indeed $B^+(x_n+L_n e_d,\ell)\subset\mathbb H^+(M_n)$ whereas $H_n\subset\mathbb H^-(M_n)$, and $B^+(x_n+L_n e_d,\ell)\subset B^+(x_n, L_n+\ell)$ using triangle inequality. Taking the Lebesgue measure, we get that
        \[|B^+(x_n, L_n+\ell)\setminus H_n|\geq \ell^d|B^+(0, 1)|.\]
    Together with \eqref{eq_area1}, it gives the result.
\end{proof}

Putting everything together, we prove Proposition \ref{prop_upper_control}.

\begin{figure}[h]
    \centering
    \includegraphics[width=\linewidth]{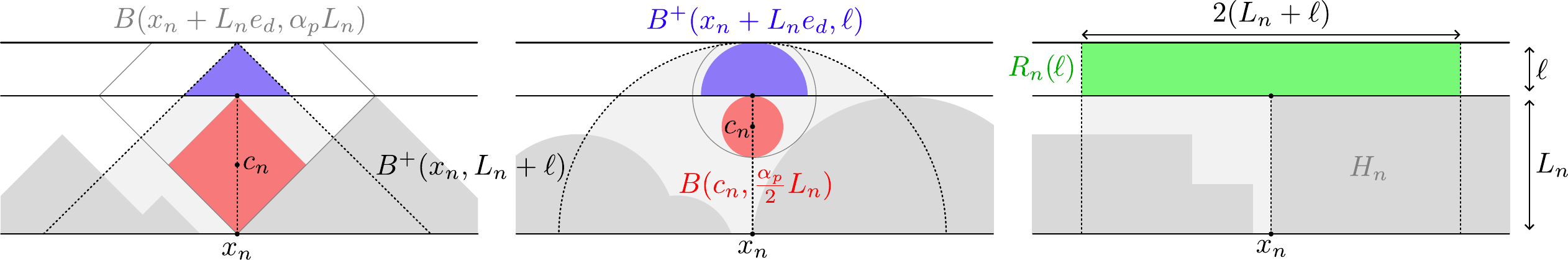}
    \caption{Illustration of the reasoning behind the proof of Lemma \ref{lemma_area_linfty} and \ref{lemma_area_lp}. From left to right, the cases $p=1$, $p=2$, and $p=\infty$.}
    \label{fig_cone}
\end{figure}

\begin{proof}[Proof of Proposition \ref{prop_upper_control}]
    We set
        \[f=f_{d, p}:(\ell, L)\mapsto\begin{cases}
            e^{-\ell(L+\ell)^{d-1}}&\text{if}\quad p=\infty\\
            e^{-c\max\{\ell, L\}^d}&\text{otherwise}
        \end{cases}\]
    where $c$ is given by Lemma \ref{lemma_area_lp}. Then, combining Lemma \ref{lemma_proba_area} with Lemmas \ref{lemma_area_linfty} and \ref{lemma_area_lp}, we get that
        \[\P(M_{n+1}-M_n\geq\ell~|~\mathcal F_n)\leq f(\ell, L_n)\]
    for each $\ell\in\R_+$ and $n\geq 0$. Additionally, for all $(\ell, L)\in(\R_+)^2$ we have
        \[f(\ell, L)\leq e^{-\min\{1,c\}\ell^d}.\]
    Finally, for each $\ell\in(0,\infty)$, the function $L\mapsto f(\ell, L)$ is non-increasing and $\lim_{L\to\infty}f(\ell, L)=0$. The proof is complete.
\end{proof}

To conclude this subsection, we derive two lemmas from Proposition \ref{prop_upper_control} that establish exponential moment bounds, which will be useful in later arguments.

\begin{Lemma}
    \label{lemma_GtL}
    For every $n\geq 0$ and $i\geq n$ we denote
        \[Y(n, i)\coloneqq \big(M_{i+1}-M_i-(L_i-L_n)_-\big)_+.\]
    Fix $t>0$. There exist a non-increasing function $G_t:\R_+\to\R$ with $\lim_{L\to\infty}G_t(L)=1$ such that for all $n\geq 0$ and $i\geq n$, we have
        \[\E\left[e^{tY(n, i)}~\middle\vert~\mathcal F_i\right]\leq G_t(L_n).\]
\end{Lemma}

\begin{proof}
    We set
        \[G_t:L\mapsto 1+\int_0^\infty te^{t\ell}f(\ell, L)\mathrm d\ell,\]
    where $f$ is given by Proposition \ref{prop_upper_control}. Using point (b) of Proposition \ref{prop_upper_control}, we get $G_t(0)<\infty$. The integrand in the definition of $G_t(L)$ decreases to $0$ as $L\rightarrow\infty$ for all $\ell>0$ from point (c), therefore by dominated convergence, $G_t(L)$ decreases to $1$ as $L\rightarrow\infty$. Now, observe that applying Fubini, we have
    \begin{equation}
        \label{eq_fubini}
        \begin{split}
            \E\left[e^{tY(n, i)}~\middle\vert~\mathcal F_i\right]&=\int_\R te^{t\ell}\P[Y(n, i)\geq \ell~\vert~\mathcal F_i]\mathrm d\ell\\
            &=1+\int_0^\infty te^{t\ell}\P[Y(n, i)\geq \ell~\vert~\mathcal F_i]\mathrm d\ell,
        \end{split}
    \end{equation}
    where the last equality comes from the fact that $Y(n, i)\geq 0$ almost surely and $\int_{\R_-}te^{t\ell}\mathrm d\ell=1$. Now, using Proposition \ref{prop_upper_control} again, for all $\ell>0$ we write
    \begin{equation*}
        \begin{split}
            \P[Y(n, i)\geq \ell~\vert~\mathcal F_i]&=\P(M_{i+1}-M_i\geq \ell+[L_i-L_n]_-~\vert~\mathcal F_i)\\
            &\leq\P(M_{i+1}-M_i\geq \ell~\vert~\mathcal F_i)\\
            &\leq f(\ell, L_n).
        \end{split}
    \end{equation*}
    Injecting in \eqref{eq_fubini}, we get
        \[\E\left[e^{tY(n, i)}~\middle\vert~\mathcal F_j\right]\leq 1+\int_0^\infty te^{t\ell}f(\ell, L_n)\mathrm d\ell=G_t(L_n).\]
    The proof is complete.
\end{proof}

%The following proposition conclude the subsection. It is this result together with a result on the lower control of the increase of $(m_n)_{n\geq 0}$ that will allow us to get sufficient control on $(L_n)_{n\geq 0}$ to prove the good steps decomposition of Proposition \ref{prop_good_steps}.

\begin{Lemma}
    \label{lemma_upper_moment}
    For all $n\geq0$ and $j>n$ we have
        \[\E\left[e^{t(L_j-L_n)_+}~\middle\vert~\mathcal F_n\right]\leq\E\left[e^{t(L_{j-1}-L_n+M_j-M_{j-1})_+}~\middle\vert~\mathcal F_n\right]\leq \big(G_t(L_n)\big)^{j-n}.\]
\end{Lemma}

\begin{proof}
    The first inequality comes from the inequality of the corresponding variables since we have $L_j-L_{j-1}\leq M_j-M_{j-1}$. To prove the second inequality, we proceed by induction on $j$. If $j=n+1$, we have
        \[L_{j-1}-L_n+M_j-M_{j-1}=M_{n+1}-M_n=Y(n, n),\]
    hence the initialization follows from Lemma \ref{lemma_GtL}. Now, suppose the result holds for some $j>n$. We can write
    \begin{equation*}
        \begin{split}
            (L_j-L_n+M_{j+1}-M_j)_+&=(L_j-L_n)_++(M_{j+1}-M_j-[L_{j}-L_n]_-)_+\\
            &=(L_j-L_n)_++Y(n, j),
        \end{split}
    \end{equation*}
    where the first equality comes from the fact that $(a+b)_+=b_++(a-b_-)_+$ for all $(a, b)\in\R_+\times\R$. We deduce then that
    \begin{equation*}
        \begin{split}
            \E\left[e^{t(L_j-L_n+M_{j+1}-M_j)_+}~\middle\vert~\mathcal F_n\right]
            &=\E\left[e^{t(L_j-L_n)_+} e^{tY(n, j)}~\middle\vert~\mathcal F_n\right]\\
            &\leq\E\left[e^{t(L_j-L_n)_+}~\middle\vert~\mathcal F_n\right]G_t(L_n)\\
            &\leq \big(G_t(L_n)\big)^{j-n+1},
        \end{split}
    \end{equation*}
    where the second inequality comes from Proposition \ref{lemma_GtL} since $e^{t(L_j-L_n)_+}$ is $\mathcal F_j$ measurable, and the last inequality comes from the induction hypothesis. The heredity is thus verified, and the proof is complete.
\end{proof}

\subsection{A useful stochastic domination}

\label{subsection_stochastic_domination}

In this subsection, we establish a novel and useful stochastic domination result, which will later be used to obtain a lower bound on the growth of the sequence $(m_n)_{n\geq 0}$. The idea is as follows. For all $n\geq 0$, recall that $y_{n+1}$ denotes the closest point to the moving vertex $x_n$ in $\mathbb H^+(m_n)\cap\mathcal N\setminus\overline H_n$. The key difference between $\Psi(x_n)$ and $y_{n+1}$ is that $y_{n+1}$ ignores the other trajectories, whose current positions may lie on $\partial H_n$. To obtain a lower bound the increments of $(m_n)_{n\geq 0}$, we aim to bound from below, for each $n\geq0$, the quantity
    \[(y_{n+1}-x_n)\cdot e_d.\]
To study this quantity, for any Borel set $A\subset\R^d$ of finite measure, we introduce the random variable $X^A$, which is the closest point to the origin in $\mathbb H^+(0)\cap\mathcal N'\setminus A$, for the $\ell^p$ distance and where $\mathcal N'$ denotes an independent copy of $\mathcal N$. By construction, using a spatial translation along with Proposition \ref{prop_resampling}, we have that for all $n\geq 0$, conditionally on $\mathcal F_n$, the displacement $y_{n+1}-x_n$ is distributed as $X^{H_n-x_n}$. In particular, conditionally on $\mathcal F_n$, $(y_{n+1}-x_n)\cdot e_d$ is distributed as 
    \[X^{H_n-x_n}\cdot e_d.\]
Let us denote by $\mathcal H_0$ the collection of all subsets of $\R^d$ that can be written as the intersection of $\mathbb H^+(0)$ with a finite union of balls that avoid the origin and whose centers lie in $\mathbb H^-(0)$. In words, $\mathcal H_0$ represents the collection of all \emph{re-centered valid history sets}. By construction, we have $H_n-x_n\in\mathcal H_0$ almost surely. It is therefore sufficient to study the random variable $X^H\cdot e_d$ for any fixed $H\in\mathcal H_0$. In the remainder of this subsection, we fix such a deterministic set $H\in\mathcal H_0$. The main theorem of this section asserts that, for all $n\geq 0$, conditionally on the first $n$ steps of the exploration process, $(y_{n+1}-x_n)\cdot e_d$ is stochastically larger than it would be if the history set $H_n$ were empty.

\begin{Theorem}
    \label{thm_sto_dom_poisson}
    If $d=2$ or $p\in\{1, 2, \infty\}$, we have the following stochastic domination:
        \[X^H\cdot e_d\succeq_{\mathrm{sto}}X^\emptyset\cdot e_d,\]
    that is, $\P(X^H\cdot e_d\geq h)\geq\P(X^\emptyset\cdot e_d\geq h)$ for for all $h\in\R_+$. In particular, for all $n\geq 0$, conditional on $\mathcal F_n$, we have $(y_{n+1}-x_n)\cdot e_d\succeq_{\mathrm{sto}}X^\emptyset\cdot e_d$, meaning that
        \[\P\big((y_{n+1}-x_n)\cdot e_d\geq h~\big|~\mathcal F_n\big)\geq\P(X^\emptyset\cdot e_d\geq h)\]
    for each $h\geq 0$.
\end{Theorem}

To our knowledge, such stochastic domination was never stated before. It is worth noting that we will not provide an explicit coupling that realizes the stochastic domination. Indeed, the most natural coupling one might consider, namely, using the same Poisson point process, fails to achieve it. We therefore adopt a computational strategy, comparing the distributions directly. We begin by establishing an intermediate stochastic domination in Section~\ref{section_uniform_sto_dom}, which forms the core of the proof of Theorem~\ref{thm_sto_dom_poisson}, presented in Section~\ref{section_proof_sto_dom}.

The fact that Theorem \ref{thm_sto_dom_poisson} holds for any $p\in[1, \infty]$ when $d=2$ naturally raises the question of whether the result remains valid beyond these specific assumptions. As we will see in the sequel, however, one of the intermediate results crucial to our proof fails to hold in general once these conditions are removed. This failure is far from obvious and suggests that the general case is genuinely delicate. Whether the theorem itself extends remains an open and subtle question. There is, at this stage, no clear reason why it should or should not.

\subsubsection{An intermediate stochastic domination}

\label{section_uniform_sto_dom}

In this subsection, we establish an intermediate stochastic domination that will be crucial in the proof of Theorem \ref{thm_sto_dom_poisson}. For any Borel set $A\subset\R^d$ such that $|B^+(0,1)\setminus A|>0$, we denote by $U^A$ a uniform point of $B^+(0,1)\setminus A$. The aim is to show the following proposition. Recall that $H\in\mathcal H_0$ is a fixed re-centered valid history set. Assume that $|B^+(0,1)\setminus H|>0$. The aim is to show the following proposition.

\begin{Proposition}
    \label{prop_sto_dom}
    We have the stochastic domination $U^H\cdot e_d\succeq_{\mathrm{sto}}U^\emptyset\cdot e_d$, that is for all $h\in[0, 1)$
    \begin{equation}
        \label{eq_dom}
        \P(U^H\cdot e_d\geq h)\geq\P(U^\emptyset\cdot e_d\geq h).
    \end{equation}
\end{Proposition}

Let us start with some notations. For all $h\in[0, 1)$, we define the section at level $h$ by
    \[S_h\coloneqq \{x\in B(0,1):x\cdot e_d=h\}.\]
Let $\mathbf p:\R^d\rightarrow\R^{d-1},~(x_k)_{1\leq k\leq d}\mapsto(x_k)_{1\leq k\leq d-1}$ be the application that forgets the $e_d$ coordinate. For all $h\in[0,1)$, we define the proportion of $S_h$ that is not covered by $H$ by
    \[\alpha_h\coloneqq\frac{|\mathbf p(S_h\setminus H)|}{|\mathbf p(S_h)|}.\]
Note that here $|\cdot|$ denotes the Lebesgue measure on $\R^{d-1}$. The main idea is the following. The random variable $U^\emptyset\cdot e_d$ has density
    \[g_\emptyset:[0,1)\to\R_+,~h\mapsto \frac{1}{|B^+(0,1)|}|\mathbf p(S_h)|,\]
with respect to the Lebesgue measure on $[0,1)$, whereas $U^H\cdot e_d$ has density
    \[g_H:[0,1)\to\R_+,~h\mapsto \frac{1}{|B^+(0,1)\setminus H|}\alpha_h|\mathbf p(S_h)|.\]
Therefore, denoting the normalizing constant $Z\coloneqq|B^+(0,1)\setminus H|/|B^+(0, 1)|$, for all $h\in[0,1)$, the Radon-Nikodym derivative of the law of $U^H\cdot e_d$ with respect to the law of $U^\emptyset\cdot e_d$ satisfies
    \begin{equation}
        \label{eq_radon_nikodym}
        \forall h\in[0,1),\quad\frac{\mathrm{d}\P_{U^H\cdot e_d}}{\mathrm{d}\P_{U^\emptyset\cdot e_d}}(h)=\frac{\alpha_h}{Z}.
    \end{equation}
Intuitively, we can say that if $h\mapsto\alpha_h$ is non-decreasing, then the random variable $U^H\cdot e_d$ tends to place more mass on larger values than $U_\emptyset\cdot e_d$, suggesting the stochastic domination $U^H\cdot e_d\succeq_{\mathrm{sto}}U^\emptyset\cdot e_d$. The following lemma confirms that this intuition is indeed correct.

\begin{Lemma}
    \label{lemma_radon_sto_dom}
    Let $\mu$ and $\nu$ be two Borel probability measure on $\R$ with $\nu$ absolutely continuous with respect to $\mu$ and such that the Radon-Nikodym derivative
        \[D:\R\to\R_+,~h\mapsto\frac{\mathrm{d}\nu}{\mathrm{d}\mu}(h)\]
    is non-decreasing. Then $\nu\succeq_{\mathrm{sto}}\mu$ in the sense that for all $h\in\R$,
        \[\nu([h, \infty])\geq\mu([h,\infty]).\]
\end{Lemma}

\begin{proof}
    Let $h\in\R$. Since $D$ is non-decreasing, we have
        \begin{equation}
            \label{eq_Dh1}
            \nu([h, \infty])=\int_{[h,\infty]} D(t)\mathrm{d}\mu(t)\geq D(h)\int_{[h,\infty]}\mathrm{d}\mu(t)=D(h)\mu([h,\infty]).
        \end{equation}
    Similarly, we have
        \begin{equation}
            \label{eq_Dh2}
            \nu(\R\setminus[h, \infty])\leq D(h)\mu(\R\setminus[h, \infty]).
        \end{equation}
    If $D(h)\geq1$, the sought inequality follows from \eqref{eq_Dh1}, otherwise \eqref{eq_Dh2} gives it. The proof is complete.
\end{proof}

We now dedicate the remainder of this section to proving that the proportion function $h\mapsto\alpha_h$ is indeed non-decreasing. To do this, we compare different sections by mapping them onto a common reference, namely $S_0$, through a combination of translation and scaling. For that purpose, for every $h\in[0,1)$, we define the radius of the section $S_h$ by
    \[\rho(h)\coloneqq\begin{cases}
        1&\text{if}\quad p=\infty\\
        (1-h^p)^{\frac{1}{p}}&\text{otherwise.}
    \end{cases}\]
We then define the map
    \[\Phi:B(0,1)\cap\mathbb H^+(0)\mapsto S_0,~x\mapsto\frac{x-(x\cdot e_d)e_d}{\rho(x\cdot e_d)}.\]
For each $h \in [0,1)$, the restriction of $\Phi$ to $S_h$ is a bijection onto $S_0$, and its inverse is given by 
    \[{\Phi_{\vert S_h}}^{-1}:S_0\ni x_0\mapsto \rho(h)x_0+he_d.\]
An illustration of the function $\Phi$ is presented in Figure \ref{fig_phi}. The following lemma uses the map $\Phi$ to compare the sections via set inclusion. The proof is technical and relies on convexity, triangle inequalities, and orthogonal decomposition. The condition that $d=2$ or $p\in\{1,2,\infty\}$ emerges naturally in the proof. Roughly speaking, this arises from the following specific singularities: when $p=\infty$, all sections have the same radius; when $p=2$, the Pythagorean Theorem applies; when $p=1$, the function $\|\cdot\|^p$ satisfies the triangular inequality; and finally, when $d=2$, the $\mathbb{R}^d$ is only one dimension higher than $e_d\R$.

\begin{Lemma}
    \label{lemma_section}
    Assume $d=2$ or $p\in\{1, 2, \infty\}$. Let $h,h'\in[0,1)$ be such that $h\geq h'$. Then,
        \[\Phi(S_h\cap H)\subset\Phi(S_{h'}\cap H).\]
\end{Lemma}

\begin{figure}[h]
    \centering
    \includegraphics[width=0.75\linewidth]{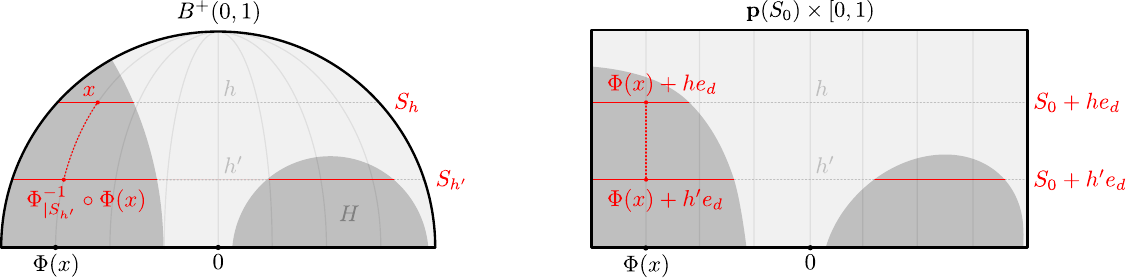}
    \caption{Illustration of how the function $\Phi$ allows to transform $B^+(0,1)$ into $\mathbf p(S_0)\times[0,1)$.}
    \label{fig_phi}
\end{figure}

\begin{proof}
    Let us start with the case $p=\infty$. Let $x_0\in\Phi(S_h\cap H)$. Since $\rho(h)=\rho(h')=1$, we have that $x_0+he_d\in H$ and we need to show that $x_0+h'e_d\in H$. Since $x_0+he_d\in H$, there exist $c\in\mathbb H^-(0)$ and $r>0$ with $\mathbb H^+(0)\cap B(c,r)\subset H$ such that $x_0+he_d\in B(c, r)$. Since $h\geq h'$ and $c\cdot e_d\leq 0$ from $c\in\mathbb H^-(0)$, we have
        \begin{equation*}
            \begin{split}
                \|x_0+h'e_d-c\|&=\max\{\|\mathbf p(x_0-c)\|, h'-c\cdot e_d\}\\
                &\leq\max\{\|\mathbf p(x_0-c)\|, h-c\cdot e_d\}\\
                &=\|x_0+he_d-c\|<r,
            \end{split}
        \end{equation*}
    implying that $x_0+h'e_d\in \mathbb H^+(0)\cap B(c,r)\subset H$. This concludes the proof when $p=\infty$.

    Now, assume $p\neq\infty$. Let $x_0\in\Phi(S_h\cap H)$. Again, we have $x\coloneqq\rho(h) x_0+he_d\in H$ and we need to show that $x'\coloneqq\rho(h')r+h'e_d\in H$. Also, since $x\in H$, there exist $c\in\mathbb H^-(0)$ and $r>0$ such that $x\in \mathbb H^+(0)\cap B(c, r)\subset H$ and it is enough to check that $\|x-c\|^p-\|x'-c\|^p\geq 0$. First, since $c\in\mathbb H^-(0)$, we can write $c=c_0-\beta e_d$ with $c_0\cdot e_d=0$ and $\beta\in\R_+$. Then we have
    \begin{equation*}
        \begin{split}
            &\|x-c\|^p=(h+\beta)^p+\|\rho(h)x_0-c_0\|^p\\
           & \text{and}\quad\|x'-c\|^p=(h'+\beta)^p+\|\rho(h')x_0-c_0\|^p
        \end{split}
    \end{equation*}
    Subtracting, we get
    \begin{equation}
        \label{eq_diff_distances_p}
        \begin{split}
            \|x-c\|^p-\|x'-c\|^p&=(h+\beta)^p-(h'+\beta)^p+\|\rho(h)x_0-c_0\|^p-\|\rho(h')x_0-c_0\|^p\\
            &\geq h^p-{h'}^p+\|\rho(h)x_0-c_0\|^p-\|\rho(h')x_0-c_0\|^p,
        \end{split}
    \end{equation}
    where the inequality comes from convexity of $y\mapsto y^p$ since $h\geq h'$ and $\beta \geq0$. Let us show that we have indeed $\|x-c\|^p-\|x'-c\|^p\geq 0$ when $d=2$ or $p\in\{1, 2\}$ using \eqref{eq_diff_distances_p}. If $p=1$, we can write
    \begin{equation*}
        \begin{split}
            \|\rho(h)x_0-c_0\|^p-\|\rho(h')x_0-c_0\|^p&=\|\rho(h)x_0-c_0\|-\|\rho(h')x_0-c_0\|\\
            &\geq-\|\rho(h)x_0-\rho(h')x_0\|\\
            &=-|\rho(h)-\rho(h')|\|x_0\|\\
            &=-(h-h')\|x_0\|\\
            &\geq -(h-h'),
        \end{split}
    \end{equation*}
    where we used the triangle inequality in the second line, and the last inequality uses $\|x_0\|\leq 1$ from $x_0\in S_0$. Injecting in \eqref{eq_diff_distances_p} we get $\|x-c\|^p-\|x'-c\|^p\geq0$ and the case is close. It remains to check the case $p=2$ and the case $d=2$. We treat them simultaneously. Assume either $p=2$ or $d=2$. We can write
        \[x_0=\alpha c_0+x_0^{\perp}\]
    where $\alpha\in\R$ and $x_0^\perp\cdot c_0=0$. Note that since $x\in B(c, r)$ and $0\notin B(c, r)\subset H$, we have
        \begin{equation*}
            \begin{split}
                \beta^p+\|c_0\|^p=\|c\|^p&\geq r^p\\
                &>\|x-c\|^p\\
                &=(\beta+h)^p+|\alpha\rho(h)-1|^p\|c_0\|^p\\
                &\geq\beta^p+|\alpha\rho(h)-1|^p\|c_0\|^p,
            \end{split}
        \end{equation*}
    leading to $\|c_0\|>|\alpha\rho(h)-1|\|c_0\|$. This implies that $c_0\neq 0$, and $\alpha\geq 0$ as $\rho(h)\geq 0$. Now, observe that we always have
        \begin{equation}
            \label{eq_norm_perp1}
            \|\rho(h)x_0-c_0\|^p=\|\rho(h)\alpha c_0-c_0\|^p+\|\rho(h)x_0^\perp\|^p.
        \end{equation}
    Indeed, if $p=2$ the equality follow from the Pythagorean Theorem, and if $d=2$ then $x_0\in\{z\in\R^d:z\cdot e_d=0\}=c_0\R$ as the two spaces have dimension $d-1=1$, thus $x_0^\perp=0$ and the inequality is verified. Similarly,
        \begin{equation}
            \label{eq_norm_perp2}
            \|\rho(h')x_0-c_0\|^p=\|\rho(h')\alpha c_0-c_0\|^p+\|\rho(h')x_0^\perp\|^p.
        \end{equation}
    Subtracting \eqref{eq_norm_perp1} and \eqref{eq_norm_perp2}, we get
        \begin{equation}
            \label{eq_alpha_diff}
            \|\rho(h)x_0-c_0\|^p-\|\rho(h')x_0-c_0\|^p=(|\alpha\rho(h)-1|^p-|\alpha\rho(h')-1|^p)\|c_0\|^p-(h^p-h'^p)\|x_0^\perp\|^p.
        \end{equation}
    As $\rho(h)\leq\rho(h')$ and $\alpha\geq 0$, the convexity of $z\mapsto|z|^p$ implies that
        \[|\alpha\rho(h')-1|^p-|\alpha\rho(h)-1|^p\leq|\alpha\rho(h')|^p-|\alpha\rho(h)|^p=\alpha^p(h^p-{h'^p}).\]
    Injecting in \eqref{eq_alpha_diff}, we deduce
        \begin{equation*}
            \begin{split}
                \|\rho(h)x_0-c_0\|^p-\|\rho(h')x_0-c_0\|^p&\geq-(h^p-h'^p)(\alpha^p\|c_0\|^p+\|x_0^\perp\|^p)\\
                &=-(h^p-h'^p)\|x_0\|^p\\
                &\geq-(h^p-h'^p),
            \end{split}
        \end{equation*}
    where the second equality is obtained from $x_0=\alpha c_0+x_0^\perp$ using the Pythagorean Theorem when $p=2$ and from the fact that $x_0^\perp=0$ when $d=2$. Injecting in \eqref{eq_diff_distances_p} we get 
        \[\|x-c\|^p-\|x'-c\|^p\geq 0,\]
    which concludes the proof.
\end{proof}

\begin{Remark}
    Lemma \ref{lemma_section} does not hold in complete generality if we remove the assumption that $d=2$ or $p\in\{1,2,\infty\}$. We provide a counterexample in the Appendix, Section \ref{section_counter_example}, with a specific $H\in\mathcal H_0$ and $p=d=3$. However, it remains unclear whether this condition is tight. In particular, we have not found any counterexamples when $d\geq3$ and $p\in(1, 2)$.
\end{Remark}

We now prove the monotonicity of $h\mapsto\alpha_h$.

\begin{Lemma}
    \label{lemma_mono_alpha}
    The application $[0,1)\ni h\mapsto\alpha_h$ is non-decreasing.
\end{Lemma}

\begin{proof}
    First, observe that for all $h\in[0,1)$ we have
        \[\rho(h)\mathbf p\circ\Phi_{\vert S_h}=\mathbf p_{\vert S_h}.\]
    Indeed, if $x\in S_h$ then $\mathbf p(x)=\mathbf p(x-he_d)=\mathbf p(\rho(h)\Phi(x))=\rho(h)\mathbf p(\Phi(x))$. Therefore, we can write
        \[\alpha_h=\frac{|\rho(h)\mathbf p\circ\Phi(S_h\setminus H)|}{|\rho(h) \mathbf p\circ\Phi(S_h)|}=\frac{|\mathbf p\circ\Phi(S_h\setminus H)|}{|\mathbf p(S_0)|},\]
    where we used that $|rA|=r^{d-1}|A|$ for any $r>0$ and Borel set $A\subset\R^{d-1}$ as well as $\Phi(S_h)=S_0$ to get the last equality. Now, since on the right-hand side of the last display the denominator does not depend on $h$, it is enough to check that $|\mathbf p\circ\Phi(S_h\setminus H)|$ is non-decreasing in $h$. Let $h, h'\in[0,1)$ be such that $h\geq h'$. From Lemma \ref{lemma_section} we have $\Phi(S_h\cap H)\subset\Phi(S_{h'}\cap H)$, and thus $\Phi(S_h\setminus H)\supset\Phi(S_{h'}\setminus H)$. Applying $\mathbf p$ we get $\mathbf p\circ\Phi(S_{h'}\cap H)\subset \mathbf p\circ\Phi(S_h\setminus H)$ and finally taking the Lebesgue measure, $|\mathbf p\circ\Phi(S_{h'}\setminus H)|\leq|\mathbf p\circ\Phi(S_h\setminus H)|$, giving $\alpha_{h'}\leq\alpha_h$. The proof is complete.
\end{proof}

Finally, as announced, summing up everything, we get the intermediate stochastic domination.

\begin{proof}[Proof of Proposition \ref{prop_sto_dom}]
    The result is obtained by combining Lemma \ref{lemma_radon_sto_dom} with \eqref{eq_radon_nikodym} and Lemma \ref{lemma_mono_alpha}.
\end{proof}

\begin{Remark}
    \label{remark_fail}
    In Appendix, Section \ref{section_counter_example}, we present a specific example of $H \in \mathcal H_0$ such that, for $d=3$ and $p=4$, numerical computations suggest that the function $h\mapsto\alpha_h$ is non-increasing and non-constant. This may imply that $U^H \cdot e_d$ and $U^\emptyset \cdot e_d$ are not identically distributed, yet the reverse stochastic domination to that of Proposition \ref{prop_sto_dom}, namely $U^H \cdot e_d \preceq_{\mathrm{sto}}U^\emptyset \cdot e_d$, holds. This provides strong evidence that, perhaps surprisingly, Proposition \ref{prop_sto_dom} may fail when the assumptions on $p$ and $d$ are completely removed.
\end{Remark}

\subsubsection{Proof of Theorem \ref{thm_sto_dom_poisson}}

\label{section_proof_sto_dom}

In this subsection, we derive Theorem \ref{thm_sto_dom_poisson} from Proposition \ref{prop_sto_dom}. For any Borel set $A\subset\R^d$ with $|A|<\infty$, denote by $X_2^A$ second-closest point to $0$ in $\mathbb H^+(0)\cap \mathcal N\setminus A$. Note that since $X^A_2$ and $X^\emptyset_2$ are defined using the same Poisson point process, we always have
\begin{equation}
    \label{eq_ineq_r2}
    \|X_2^A\|\geq\|X_2^\emptyset\|.
\end{equation}
The following lemma establishes the key connection between the closest Poisson points and uniform points by conditioning on $\|X_2^A\|$. As the result can be obtained using straightforward density estimations, we omit its proof.

\begin{Lemma}
    \label{lemma_uniformisation}
    Let $A\subset\R^d$ be a Borel set with $|A|<\infty$. Then, conditionally on $\|X_2^A\|$,
        \[\frac{X^A}{\|X_2^A\|}\quad\text{is uniform in}\quad B^+(0,1)\setminus \frac{A}{\|X_2^A\|},\quad \text{i.e.\ distributed as}\quad U^{A/\|X_2^A\|}.\]
\end{Lemma}

Using this connection, we derive Theorem \ref{thm_sto_dom_poisson}.

\begin{proof}[Proof of Theorem \ref{thm_sto_dom_poisson}]
    For any Borel set $A\subset\R^d$ such that $|A|<\infty$ we denote $f^A:h\mapsto\P(U^A\cdot e_d\geq h)$. Recall that from Proposition \ref{prop_sto_dom} we have $f^H\geq f^\emptyset$ for any re-centered valid history set $H\in\mathcal H_0$. Let $h\in[0,1)$, the result is obtained writing
    \begin{equation*}
        \begin{split}
            \P(X^H\geq h)&=\E\left[\P\left(\frac{X^H}{\|X_2^H\|}\geq \frac{h}{\|X_2^H\|}~\middle\vert~ \|X_2^H\|\right)\right]\\
            &=\E\left[f^{H/\|X_2^H\|}\left(\frac{h}{\|X_2^H\|}\right)\right]\\
            &\geq\E\left[f^{\emptyset}\left(\frac{h}{\|X_2^H\|}\right)\right]\\
            &\geq\E\left[f^{\emptyset}\left(\frac{h}{\|X_2^\emptyset\|}\right)\right]\\
            &=\E\left[\P\left(\frac{X^\emptyset}{\|X_2^\emptyset\|}\geq\frac{h}{\|X_2^\emptyset\|}~\middle\vert~\|X_2^\emptyset\|\right)\right]=\P(X^\emptyset\geq h).
        \end{split}
    \end{equation*}
    The second and second-to-last equalities follow from Lemma \ref{lemma_uniformisation} and the definition of $f^\cdot$. The first inequality comes from Proposition \ref{prop_sto_dom} as $H/\|X_2^H\|\in\mathcal H_0$ by scaling of $H\in\mathcal H_0$. The second inequality follows from the fact that $f^\emptyset$ is non-increasing and \eqref{eq_ineq_r2}. The proof is complete.
\end{proof}

\subsection{Lower control for the growth of $(m_n)_{n\geq 0}$}

\label{subsection_lower_control}

From now on, we will work with the assumption $p=2$ or $d\in\{1, 2,\infty\}$.\\

In this subsection, we use the stochastic domination of Theorem \ref{thm_sto_dom_poisson} to provide an exponential moment lower bound on the growth of the bottom of the history sets, $(m_n)_{n\geq0}$. Recall that as a consequence of Proposition \ref{prop_resampling}, conditionally on $\mathcal F_n$, the displacement $y_{n+1}-x_n$ is distributed as $X^{H_n-x_n}$. Furthermore, $\Psi(x_n)$ is the closest point to $x_n$ in
    \[\{y_{n+1}\}\cup\{g_n(\mathbf u_i):i\in\range{1}{k}\}\setminus\{x_n\}.\]
Hence, when $k\geq 2$, due to the presence of $k-1$ other trajectories, $\Psi(x_n)$ is not necessarily equal to $y_{n+1}$. This complicates the analysis conditionally on $\mathcal F_n$: for instance, if the other trajectories are very close to $x_n$, coalescence will very likely occur and the increment $m_{n+1}-m_n$ can be very small. To overcome this issue, we consider the progress over $k$ consecutive steps of the exploration process, rather than just one. The following lemma ensures that, conditionally on $\mathcal F_n$, the increment $m_{n+k}-m_n$ stochastically dominates a fixed law.

\begin{Lemma}
    \label{lemma_progress}
    Let $(X^\emptyset_i)_{1\leq i\leq k}$ denote i.i.d.\ copies of $X^\emptyset$. We set
        \[\Gamma_k\coloneqq\min_{1\leq i\leq k}(X_i^\emptyset\cdot e_d).\]
    Then, for all $n\geq 0$, the law of $m_{n+k}-m_n$ conditional on $\mathcal F_n$ stochastically dominates the law of $\Gamma_k$, that is, for all $h\in\R_+$,
        \[\P\left(m_{n+k}-m_n\geq h~\vert~\mathcal F_n\right)\geq\P(\Gamma_k\geq h).\]
\end{Lemma}

\begin{proof}
    Let $n\geq 0$. To begin, let us show that we have almost surely
        \begin{equation}
            \label{eq_mnk}
            m_{n+k}\geq\min_{n\leq i< n+k}(y_{i+1}\cdot e_d).
        \end{equation}
    The intuition is that making $k$ steps, coalescence can occur on at most $k-1$ of them, hence at least one step will connect to some $y_{j+1}$ with $n\leq j<n+k$. However, for a fixed $j\geq0$, we don't necessarily have $m_{j+1}=\Psi(x_j)\cdot e_d$, which makes the analysis more delicate. To begin, observe that for all $j\geq 0$, we have
        \[\Psi(x_j)\in\{y_{j+1}\}\cup\{g_j(\mathbf u_i)\}_{1\leq i\leq k}\setminus\{x_j\}.\]
    Hence, since only $x_j$ moves during the step, we deduce that
        \[\{g_{j+1}(\mathbf u_i)\}_{1\leq i\leq k}\subset\{y_{j+1}\}\cup\{g_j(\mathbf u_i)\}_{1\leq i\leq k}\setminus\{x_j\},\]
    with $x_j\in\{g_j(\mathbf u_i)\}_{1\leq i\leq k}$. By induction, it follows that
        \begin{equation}
            \label{eq_gn1}
            \{g_{n+k}(\mathbf u_i)\}_{1\leq i\leq k}\subset\{y_{\ell+1}\}_{n\leq \ell<n+k}\cup\{g_n(\mathbf u_i)\}_{1\leq i\leq k}\setminus\{x_\ell\}_{n\leq \ell<n+k},
        \end{equation}
    with $\{x_\ell\}_{n\leq l<n+k}\subset\{y_{\ell+1}\}_{n\leq l<n+k}\cup\{g_n(\mathbf u_i)\}_{1\leq i\leq k}$. We proceed by treating two cases. If we have
        \[\{x_\ell\}_{n\leq \ell<n+k}\cap\{y_{\ell+1}\}_{n\leq \ell<n+k}\neq\emptyset,\]
    then there exist $\ell_1, \ell_2\in\range{n}{n+k-1}$ such that $x_{\ell_1}=y_{\ell_2+1}$ and thus \eqref{eq_mnk} is verified as
        \[m_{n+k}\geq m_{\ell_1}=x_{\ell_1}\cdot e_d=y_{\ell_2+1}\cdot e_d\geq\min_{n\leq i< n+k}(y_{i+1}\cdot e_d).\]
    In the other case, we have $\{x_\ell\}_{n\leq \ell<n+k}\subset\{g_n(\mathbf u_i)\}_{1\leq i\leq k}$. Hence, since the $(x_\ell)_{n\leq \ell<n+k}$ are disjoints, by checking cardinals we get $\{x_\ell\}_{n\leq \ell<n+k}=\{g_n(\mathbf u_i)\}_{1\leq i\leq k}$. From \eqref{eq_gn1} we get then
        \[\{g_{n+k}(\mathbf u_i)\}_{1\leq i\leq k}\subset\{y_{\ell+1}\}_{n\leq \ell<n+k},\]
    and \eqref{eq_mnk} follows. We have proved \eqref{eq_mnk}. Now, observe that for each $n\leq i<n+k$, we have
        \[y_{i+1}\cdot e_d-m_n\geq y_{i+1}\cdot e_d-m_i=(y_{i+1}-x_i)\cdot e_d,\]
    so that together with \eqref{eq_mnk},
        \[m_{n+k}-m_n\geq\min_{n\leq i< n+k}\big((y_{i+1}-x_i)\cdot e_d\big).\]
    Therefore, for any $h\in\R_+$, we can write
        \begin{equation*}
            \begin{split}
                \P\left(m_{n+k}-m_n\geq h~\vert~\mathcal F_n\right)&\geq\P\left(\min_{n\leq i< n+k}\big((y_{i+1}-x_i)\cdot e_d\big)\geq h~\middle|~\mathcal F_n\right)\\
                &=\E\left[\prod_{i=n}^{n+k-1}\mathds 1_{(y_{i+1}-x_i)\cdot e_d\geq h}~\middle\vert~\mathcal F_n\right]
                \geq\P(X^\emptyset\cdot e_d\geq h)^k
                =\P(\Gamma_k\geq h),
            \end{split}
        \end{equation*}
    where for the second inequality, we used that for all $n\leq i<n+k$, the law of $(y_{i+1}-x_i)\cdot e_d$ conditional on $\mathcal F_i$ stochastically dominates the law of $X^\emptyset\cdot e_d$ from Theorem \ref{thm_sto_dom_poisson}. The proof is complete.
\end{proof}

We conclude this subsection by deriving an exponential moment bound from Lemma \ref{lemma_progress}.

\begin{Lemma}
    \label{lemma_lower_control}
    Let $t>0$. There exist $c_t=c_t(k, d, p)<1$ such that for all $n\geq 0$ we can find a $\mathcal F_n$-measurable random index $K_n\in\range{1}{k}$ such that
        \[\E\left[e^{-t(m_{n+K_n}-m_{n+K_n-1})}~\middle\vert~\mathcal F_n\right]\leq c_t.\]
\end{Lemma}

\begin{proof}
    Since $\Gamma_k>0$ almost surely, there exists $h>0$ such that $\P(\Gamma_k\geq h)>0$. Now write
    \begin{equation*}
        \begin{split}
            \P(\Gamma_k\geq h)\leq\P(m_{n+k}-m_n\geq h~\vert~\mathcal F_n)&\leq\P\left(\exists j\in\range{1}{k}~m_{n+j}-m_{n+j-1}\geq\frac{h}{k}~\middle\vert~\mathcal F_n\right)\\
            &\leq k\max_{1\leq j\leq k}\P\left(m_{n+j}-m_{n+j-1}\geq\frac{h}{k}~\middle\vert~\mathcal F_n\right),
        \end{split}
    \end{equation*}
    where the first inequality comes from Lemma \ref{lemma_progress}. Now set $K_n$ to be the smallest index realizing the maximum in the right-hand side of the last display. Since each term of this maximum is $\mathcal F_n$-measurable, $K_n$ is $\mathcal F_n$-measurable. Moreover, by construction,
        \[\P\left(m_{n+K_n}-m_{n+K_n-1}\geq\frac{h}{k}~\middle\vert~\mathcal F_n\right)\geq\frac{1}{k}\P(\Gamma_k\geq h)>0.\]
    This inequality allows us to upper bound $\E[e^{-t(m_{n+K_n}-m_{n+K_n-1})}~\vert~\mathcal F_n]$ by a constant strictly less than $1$ that does not depend on $n$. The proof is complete.
\end{proof}

%\begin{Remark}\todo{for the PhD}
%    \color{lightgray}
%    In fact, to bypass the arguments of stochastic domination trying to get rid of the condition $d=2$ or $p\in\{1, 2, \infty\}$, it would be sufficient to find $h>0$ and $c>0$ such that for all $n\geq 0$,
%        \[\P(X^H\cdot e_d\geq h~\vert~\mathcal F_n)\geq c>0.\]
%\end{Remark}

\subsection{Proof of Theorem \ref{thm_good_steps}}

\label{subsection_proof_good_steps}

In this section, we gather both the upper control on the increase of $(M_n)_{n\geq 0}$ and the lower control on the increase of $(m_n)_{n\geq 0}$ to prove Theorem \ref{thm_good_steps}. We start with the following proposition, which provides a strong exponential bound on the increments of $(L_n)_{n\geq0}$.

\begin{Proposition}
    \label{prop_ln_exp_control}
    There exist $q=q(k, d, p)<1$ and $\kappa=\kappa(k, d, p)>0$ such that, for all $n\geq0$, we have
        \[\mathds 1_{L_n>\kappa}\E[e^{L_{n+K_n}-L_n}~\vert~\mathcal F_n]\leq q<1.\]
\end{Proposition}

\begin{proof}
    Denote $A_n\coloneqq L_{n+K_n-1}-L_n+M_{n+K_n}-M_{n+K_n-1}$, so that
        \[L_{n+K_n}-L_n=A_n-(m_{n+K_n}-m_{n+K_n-1}).\]
    From Lemma \ref{lemma_upper_moment} applied with $t=2$ and $j=n+K_n$ which is $\mathcal F_n$-measurable, we have
        \[\E[e^{2A_n}~\vert~\mathcal F_n]\leq\E[e^{2(A_n)_+}~\vert~\mathcal F_n]\leq G_2(L_n)^{K_n}\leq G_2(L_n)^k,\]
    where the last inequality uses that $K_n\leq k$. Therefore, using the Cauchy-Schwarz inequality
    \begin{equation*}
        \begin{split}
            \E[e^{L_{n+K_n}-L_n}~\vert~\mathcal F_n]^2&=\E[e^{A_n} e^{-m_{n+K_n}-m_n}~\vert~\mathcal F_n]^2\\
            &\leq\E[e^{2A_n}~\vert~\mathcal F_n]\E[e^{-2(m_{n+K_n}-m_{n+K_n-1})}~\vert~\mathcal F_n]\\
            &\leq G_2(L_n)^k c_2,
        \end{split}
    \end{equation*}
    where $c_2<1$ is given by Lemma \ref{lemma_lower_control}. Finally, since $G_2$ is non-increasing and $\lim_{L\to\infty}G_2(L)=1$, we deduce that
        \[\mathds 1_{L_n>\kappa}\E[e^{L_{n+K_n}-L_n}~\vert~\mathcal F_n]\leq\sqrt{G_2(\kappa)^kc_2}\to\sqrt{c_2}<1\quad\text{as}\quad \kappa\to\infty.\]
    Hence, we can choose $q\in(\sqrt{c_2}, 1)$ and then $\kappa>0$ big enough to get the sought result. The proof is complete.
\end{proof}

From now on, we fix $\kappa$ and $q$ given by Proposition \ref{prop_ln_exp_control}, and we consider any arbitrary $R>0$. We define $(\nu_n)_{n\geq 0}$ by induction as an intermediate sequence of $(\mathcal F_n)_{n\geq0}$ stopping times with $\nu_0=0$ and for all $n\geq 0$,
    \[\nu_{n+1}\coloneqq\inf\{i>\nu_n~:~L_i\leq\kappa\}.\]
By construction, $(\tau_n(\kappa, R))_{n\geq0}$ is a subsequence of $(\nu_n)_{n\geq0}$. We start by showing that the stopping times $(\nu_n)_{n\geq0}$ are almost surely finite with spacing distributed with an exponentially decaying tail.

\begin{Lemma}
    \label{lemma_tau_n_tail}
    For all $n\geq 0$, $\nu_n$ is almost surely finite. Moreover there exist $\alpha=\alpha(k, d, p)>0$ and $\beta=\beta(k, d, p)>0$ such that, for all $n,\ell\geq 0$ we have
        \[\P\left(\nu_{n+1}-\nu_n>\ell~\vert~\mathcal F_{\nu_n}\right)\leq \alpha e^{-\beta \ell}.\]
\end{Lemma}

\begin{proof}
    We proceed by induction. $\nu_0=0$ is almost surely finite. Now let $n\geq0$. Assume that $\nu_n$ is almost surely finite. Denote $N_0\coloneqq\nu_n$ and for all $i\geq0$, 
        \[N_{i+1}\coloneqq N_i+K_{N_i}.\]
    Since $K_n\leq k$ for all $n\in\N$, then for all $\ell\geq 0$ we have $N_\ell-\nu_n\leq k\ell$, hence
        \begin{equation}
            \label{eq_tau_and_N}
            \P\left(\nu_{n+1}-\nu_n>k\ell~\vert~\mathcal F_{\nu_n}\right)\leq\P\left(N_\ell<\nu_{n+1}~\vert~\mathcal F_{\nu_n}\right).
        \end{equation}
    Now observe that for each $\ell\geq1$
        \begin{equation}
            \begin{split}
                \P(N_\ell<\nu_{n+1}~\vert~\mathcal F_{\nu_n})&\leq\E\left[\mathds 1_{L_{N_\ell}>\kappa}\prod_{i=1}^{\ell-1}\mathds 1_{L_{N_i}>\kappa}~\middle\vert~\mathcal F_{\nu_n}\right]\\
                &\leq\E\left[e^{L_{N_\ell}-\kappa}\prod_{i=1}^{\ell-1}\mathds 1_{L_{N_i}>\kappa}~\middle\vert~\mathcal F_{\nu_n}\right]\\
                &=\E\left[e^{L_{N_1}-\kappa}\prod_{i=1}^{\ell-1}\mathds 1_{L_{N_i}>\kappa}e^{L_{N_i+K_{N_i}}-L_{N_i}}~\middle\vert~\mathcal F_{\nu_n}\right]
            \end{split}
        \end{equation}
    where the second inequality comes from $\mathds 1_{a>b}\leq e^{a-b}$ for any $a, b\in\R$ and the last equality comes from a telescopic product. Now, in the last expectation of the display, we have a product of $\ell$ terms, where the first $i$ terms are $\mathcal F_{N_i}$-measurable. Therefore, using Proposition \ref{prop_ln_exp_control}, we get
        \begin{equation*}
            \P(N_\ell<\nu_{n+1}~\vert~\mathcal F_{\nu_n})\leq\E[e^{L_{N_1}-\kappa}~\vert~\mathcal F_{\nu_n}]q^{\ell-1}\leq\E[e^{L_{\nu_n+K_{\nu_n}}-L_{\nu_n}}~\vert~\mathcal F_{\nu_n}]q^{\ell-1}\leq G_1(0)^kq^{\ell-1},
        \end{equation*}
    where the second inequality comes from $L_{\nu_n}\leq\kappa$ and the last one from Lemma \ref{lemma_upper_moment} since $K_{\nu_n}\leq k$ is $\mathcal F_{\nu_n}$-measurable. Injecting in \eqref{eq_tau_and_N} we get the result.
\end{proof}

To transfer the properties we just established for the sequence $(\nu_n)_{n\geq 0}$ to the sequence $(\tau_n(\kappa, R))_{n\geq 0}$, we define, for each $n\geq 0$, the random variable
    \[T_n\coloneqq\inf\{j>0:m_{n+j}-m_n\geq\kappa+R\}.\]
In words, $T_n$ is the time we have to wait from index $n$ to reach a progress of at least $\kappa+R$. The following lemma ensures that the $(T_n)_{n\geq 0}$ are distributed with exponentially decaying tails.

\begin{Lemma}
    \label{lemma_T_n_tail}
    There exist $\alpha=\alpha(k, d, p)>0$ and $\beta=\beta(k, d, p)>0$ such that for all $n,\ell\geq 0$,
        \[\P(T_n>\ell~\vert~\mathcal F_n)\leq\alpha e^{R-\beta \ell}.\]
\end{Lemma}

\begin{proof}
    Let $n,\ell\geq 0$. Using Markov's inequality, we write
        \begin{equation*}
            \begin{split}
                \P(T_n>k\ell~\vert~\mathcal F_n)=\E[m_{n+k\ell}-m_n<\kappa+R~\vert~\mathcal F_n]
                &\leq \E[e^{\kappa+R-(m_{n+k\ell}-m_n)}~\vert~\mathcal F_n]\\
                &=e^{\kappa+R}\E\left[\prod_{i=0}^{\ell-1}e^{-(m_{n+ik+k}-m_{n+ik})}~\middle|~\mathcal F_n\right]\\
                &\leq e^{\kappa+R}\E[e^{-\Gamma_k}]^\ell,
            \end{split}
        \end{equation*}
    where the last inequality comes from Lemma \ref{lemma_progress}. Since $\E[e^{-\Gamma_k}]<1$ as $\Gamma_k>0$ almost surely, we deduce the sought exponential decay. The proof is complete.
\end{proof}

We now prove the main result of this section by bringing together all the preceding ingredients.

\begin{proof}[Proof of Theorem \ref{thm_good_steps}]
    We proceed by induction to show that $(\tau_n)_{n\geq 0}\coloneqq (\tau_n^{\kappa, R})_{n\geq 0}$ indeed satisfy the sought properties. We have $\tau_0=0$ that is almost surely finite. Now let $n\geq0$ and assume that $\tau_n$ is almost surely finite. First, observe that by definition, $\tau_n=\nu_A$ for some $\mathcal F_{\tau_n}$-measurable $A\in\N$ and $\tau_{n+1}=\nu_B$ where
        \[B=\inf\{j\geq A~:~\nu_j\geq\nu_A+T_{\tau_n}\}.\]
    Since $(\nu_j)_{j\geq 0}$ is increasing, we have $\nu_j-\nu_A\geq j-A$ hence $\nu_{A+T_{\tau_n}}-\nu_A\geq T_{\tau_n}$ and thus we must have
    \begin{equation}
        \label{eq_diffAB}
        1\leq B-A\leq T_{\tau_n}
    \end{equation}
    By definition of $B$, we have $\nu_{B-1}<\tau_n+T_{\tau_n}$, hence
    \begin{equation}
        \label{eq_difftauB1}
        \nu_B-\nu_{B-1}\geq\tau_{n+1}-\tau_n-T_{\tau_n}.
    \end{equation}
    Now let $\ell\geq 0$. Observe that from \eqref{eq_difftauB1}, and using an union bound, we can write
        \begin{equation}
            \label{eq_gamma_proba}
            \begin{split}
                \P\left(\tau_{n+1}-\tau_n>\ell~\vert~\mathcal F_{\tau_n}\right)&\leq\P(\nu_B-\nu_{B-1}> \ell-T_{\tau_n}~\vert~\mathcal F_{\tau_n})\\
                &\leq\P\left(T_{\tau_n}>\frac{\ell}{2}~\middle\vert~\mathcal F_{\tau_n}\right)+\P\left(T_{\tau_n}\leq\frac{\ell}{2},~\nu_B-\nu_{B-1}> \frac{\ell}{2}~\middle\vert~\mathcal F_{\tau_n}\right)
            \end{split}
        \end{equation}
    Next, using \eqref{eq_diffAB} we can write
        \begin{equation*}
            \P\left(T_{\tau_n}\leq\frac{\ell}{2},~\nu_B-\nu_{B-1}\geq \frac{\ell}{2}~\middle\vert~\mathcal F_{\tau_n}\right)=\P\left(\bigcup_{b=0}^{\frac{\ell}{2}-1}\left\lbrace\nu_{A+b+1}-\nu_{A+b}>\frac{\ell}{2}\right\rbrace~\middle\vert~\mathcal F_{\tau_n}\right)
            \leq\frac{\ell}{2}\alpha e^{-\beta\frac{\ell}{2}},
        \end{equation*}
    where the last inequality comes from a union bound and Lemma \ref{lemma_tau_n_tail}, since $A$ is $\mathcal F_{\tau_n}$-measurable. Injecting in \eqref{eq_gamma_proba} we finally get
        \begin{equation*}
            \P\left(\tau_{n+1}-\tau_n>\ell~\vert~\mathcal F_{\tau_n}\right)\leq\P\left(T_{\tau_n}>\frac{\ell}{2}~\middle\vert~\mathcal F_{\tau_n}\right)+\frac{\ell}{2}\alpha e^{-\beta\frac{\ell}{2}}
            \leq\alpha'e^{R-\beta'\frac{\ell}{2}}+\frac{\ell}{2}\alpha e^{-\beta\frac{\ell}{2}},
        \end{equation*}
    where $\alpha', \beta'>0$ are given by Lemma \ref{lemma_T_n_tail}. The result follows.
\end{proof}

\section{Renewal decomposition}

\label{section_renewal}

In this section, we adapt the technique of \cite{Dsf2d} to decompose the exploration process into blocks of random size, thereby making rigorous the intuition that the trajectories behave like random walks in $\R^{d-1}$, that do not interact when they are sufficiently far from each other. To this end, in Subsection \ref{subsection_renewal_event}, we first introduce certain \emph{joint renewal events} for the exploration process, along with associated random times, and show that these times are almost surely finite and that their spacing has an exponentially decaying tail. In Subsection \ref{subsection_block_size}, we then control the fluctuations of the trajectories between consecutive \emph{renewal times}. Finally, in Subsection \ref{subsection_independent_process}, we use the renewal decomposition to describe the behavior of the trajectories via an auxiliary process, namely, the \emph{independent process}.

\subsection{Joint renewal events}

\label{subsection_renewal_event}

Before describing intuitively the idea, let us start with some notations. For all $n\geq 0$ and $i\in\range{1}{k}$, we define
    \[g_n^\uparrow(\mathbf u_i)\coloneqq g_n(\mathbf u_i)+(m_n+\kappa-g_n(\mathbf u_i)\cdot e_d)e_d\quad \text{and}\quad g_n^\downarrow(\mathbf u_i)\coloneqq g_n(\mathbf u_i)+(m_n-g_n(\mathbf u_i)\cdot e_d)e_d.\]
In words, $g_n^\uparrow(\mathbf u_i)$ corresponds to the projection of $g_n(\mathbf u_i)$ on the $e_d$-level $m_n+\kappa$ while $g_n^\downarrow(\mathbf u_i)$ corresponds to the projection of $g_n(\mathbf u_i)$ on the level $m_n$. Let $\kappa>0$ be given by Theorem \ref{thm_good_steps} and $R>0$ be arbitrarily fixed. For any $j\geq 0$, we say that the good step $j$, occurring at time $\tau_j$, is a \emph{joint renewal step}, if the event
    \[A_j\coloneqq\{\forall i\in\range{1}{k},~\#\mathcal N\cap B^+(g_{\tau_j}^\uparrow(\mathbf u_i), R)=\#\mathcal N\cap B(g_{\tau_j}^\downarrow(\mathbf u_i), \kappa+R)\setminus\overline{H}_{\tau_j}=1\}\]
occurs. In words, we say that we have joint renewal of the trajectories at time $\tau_j$ if for each $i\in\range{1}{k}$ we have only one Poisson point in $B(g_{\tau_j}^\downarrow(\mathbf u_i), \kappa+R)\setminus\overline{H}_{\tau_j}$ which is actually in $B^+(g_{\tau_n}^\uparrow(\mathbf u_i), R)$. An illustration of the joint renewal event is provided in Figure \ref{fig_renewal_event}.

\begin{figure}[h]
    \centering
    \includegraphics[width=1\linewidth]{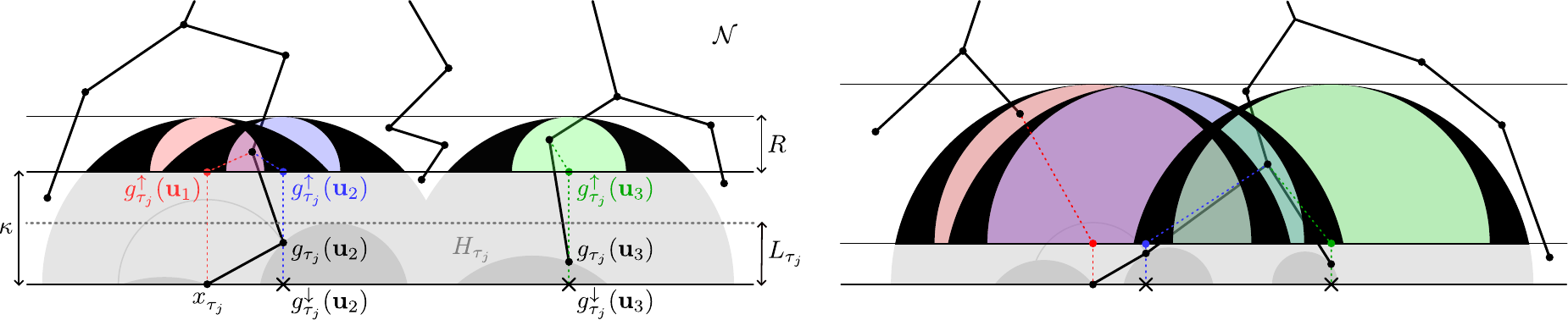}
    \caption{Illustration of the occurrence of the renewal event $A_j$, for $2$ different setups with $k=3$. The light gray and black regions contain no point of $\mathcal N$ while each translucent colorful region (red, blue, and green) contains exactly one point of $\mathcal N$.}
    \label{fig_renewal_event}
\end{figure}

We define $(\gamma_n)_{n\geq 0}$ by induction with $\gamma_0\coloneqq\inf\{i\geq0:A_i~\text{occurs}\}$ and for all $n\geq 0$,
    \[\gamma_{n+1}\coloneqq\inf\{i>\gamma_n:A_i~\text{occurs}\}.\]
In words, for each $n\geq 0$, the renewal step $n$ occurs at the good step index $\gamma_n$, that is at step
    \[\beta_n\coloneqq\tau_{\gamma_n}.\]
We now give some intuition behind the construction of these renewal steps in the case $k=1$, where only a single trajectory is being explored. Fix $n\geq 0$ and suppose that $\beta_n<\infty$. By construction of the renewal event, we have
    \[\Psi(x_{\beta_n})=\Psi\big(g_{\beta_n}(\mathbf u_1)\big)=\Psi\big(g_{\beta_n}^{\uparrow}(\mathbf u_1)\big).\]
In other words, at time $\beta_n$, we can replace the current point $x_{\beta_n}$ by $g_{\beta_n}^{\uparrow}(\mathbf u_1)$ without altering the future evolution of the process. This replacement point lies at level $m_{\beta_n}+\kappa$, which is above the current history set. As a result, we can effectively forget the past in some sense. More precisely, conditional on the information revealed up to time $\beta_n$, the point process
    \[\big(\mathcal N-g_{\beta_n}^{\uparrow}(\mathbf u_1)\big)\cap\mathbb H^+(0)\]
is distributed as $\mathcal N'\cap\mathbb H^+(0)$, where $\mathcal N'$ denotes an independent copy of $\mathcal N$, conditioned on the event
    \[\#\mathcal N'\cap B^+(0,R)=\#\mathcal N'\cap\mathbb H^+(0)\cap B(-\kappa e_d,\kappa+R)=1.\]
Intuitively, this means that the future of the process depends only on the current position and thus behaves like a Markov process. Even more, recalling that $\mathbf p:\R^d\to\R^{d-1}$ denotes the projection parallel to $e_d$, if the renewal times $(\beta_n)_{n\geq 0}$ are almost surely finite, the sequence
    \[\big(\mathbf p(g_{\beta_n}(\mathbf u_1))\big)_{n\geq 0}\]
forms a random walk with i.i.d.\ symmetric increments in $\R^{d-1}$. This formalizes the heuristic that individual trajectories behave as random walks in $\R^{d-1}$. For the cases $k\geq 2$, where we simultaneously explore multiple trajectories, the situation becomes more delicate. Specifically, it is no longer true that for $n\geq 0$ and $\beta_n<\infty$, we have $\Psi(g_{\beta_n}(\mathbf u_i))=\Psi(g_{\beta_n}^{\uparrow}(\mathbf u_i))$ for every $i\in\range{1}{k}$. In fact, as illustrated in Figure \ref{fig_renewal_event}, coalescence can occur after time $\beta_n$ below the level $m_{\beta_n}+\kappa$, when the points $(g_{\beta_n}^{\uparrow}(\mathbf u_i))_{i\in\range{1}{k}}$ are not separated by a distance of at least $\kappa+R$. Due to this possibility of early coalescence, we do not have a Markovian structure a priori. However, we will see that when the trajectories are sufficiently far apart from one another, they evolve approximately as independent random walks. Note that good steps play a crucial role in the construction of the renewal times. The remainder of the section is devoted to making the above ideas into a rigorous argument.\\

As a first step, in the remainder of this subsection, we aim to show that the random times $(\beta_n)_{n\geq0}$ are almost surely finite, and that their spacing is distributed with an exponentially decaying tail. It is worth noting that $(\beta_n)_{n\geq0}$ are not $(\mathcal F_n)_{n\geq0}$ stopping times and that we will need to define enhanced sigma fields. For all $i\geq0$ we define
    \[\mathcal S_i\coloneqq\sigma(\mathcal F_{\tau_i},~A_i).\]
The next lemma ensures that the sequence $(\mathcal S_i)_{i\geq0}$ is indeed a filtration.

\begin{Lemma}
    \label{lemma_sfiltration}
    For all $i\geq 0$, we have $\mathcal S_i\subset\mathcal F_{\tau_{i+1}}$. In particular, $(\mathcal S_i)_{i\geq 0}$ is a filtration.
\end{Lemma}

\begin{proof}
    The event $A_i$ is $\sigma(\mathcal N\cap\mathbb H^{-}(m_{\tau_i}+\kappa+R))$ measurable. Since by construction $m_{\tau_{i+1}}-m_{\tau_i}\geq\kappa+R$, then $A_i$ is $\sigma(\mathcal N\cap\mathbb H^{-}(m_{\tau_{i+1}}))\subset\mathcal F_{\tau_{i+1}}$ measurable. To show that we have a filtration, we just write that $\mathcal S_i\subset\mathcal F_{\tau_{i+1}}\subset\sigma(\mathcal F_{\tau_{i+1}}, A_{i+1})=\mathcal S_{i+1}$.
\end{proof}

We start by showing that, at each good step $j\geq0$, the renewal event $A_j$ occurs with probability uniformly bounded away from zero. To establish this, we rely on the following proposition. More than merely nontrivial, its proof requires a substantial and delicate geometric analysis. This argument is entirely deterministic and independent of the probabilistic aspects of the paper, which is why we defer it to the appendix (section \ref{section_balls}). A brief treatment of the special case $d=p=2$ was given in \cite{Dsf2d}, but in the general setting, the proof becomes significantly more intricate, which is why we provided a detailed presentation.

\begin{Proposition}
    \label{prop_combinatorial}
    There exists two constants $a_0=a_0(d, p, k)>0$ and $R_0=R_0(d, p, k)>0$ such that for all configuration $(c_i)_{i\in\range{1}{k}}\in\mathbb\R^k$ with $c_i\cdot e_d=0$ for each $i\in\range{1}{k}$, there exist a partition $\Pi$ of $\range{1}{k}$ such that for all $\pi\in\Pi$,
        \[|P_\pi^{R_0}|\geq a_0\quad\text{where}\quad P_\pi^{R_0}\coloneqq\bigcap_{i\in\pi} B^+(c_i, R_0)\setminus\bigcup_{j\in\range{1}{k}\setminus\pi} B(c_j-\kappa e_d, \kappa+R_0).\]
\end{Proposition}

We now set $R\coloneqq R_0$, as specified in the preceding proposition. The value of $R$ was left arbitrary until now precisely to allow this choice.

\begin{Lemma}
    \label{lemma_renewalproba}
    There exist $p_0=p_0(k, d, p)>0$ such that for all $n\geq0$, we have
        \[\P(A_n~|~\mathcal F_{\tau_n})\geq p_0.\]
    In particular, for all $m>n$, we have $\P(A_m~|~\mathcal S_n)\geq p_0$.
\end{Lemma}

\begin{proof}
    Let $n\geq0$. For each $i\in\range{1}{k}$, we set
        \[c_i\coloneqq g_{\tau_n}(\mathbf u_i)-(g_{\tau_n}(\mathbf u_i)\cdot e_d)e_d.\]
    Applying Proposition \ref{prop_combinatorial}, there exist $\Pi$ a partition of $\range{1}{k}$ such that for all $\pi\in\Pi$, $|P_\pi^{R_0}|\geq a_0$. By translation, this implies that for all $\pi\in\Pi$,
        \[|\Lambda_{\pi}|\geq a_0\quad\text{where}\quad \Lambda_\pi\coloneqq \bigcap_{i\in\pi} B^+\big(g_{\tau_n}^\uparrow(\mathbf u_i), R\big)\setminus\bigcup_{j\in\range{1}{k}\setminus\pi} B\big(g_{\tau_n}^\downarrow(\mathbf u_i), \kappa+R\big).\]
    Let us denote
        \[E\coloneqq\bigcup_{i\in\range{1}{k}}B\big(g_{\tau_n}^\downarrow(\mathbf u_i), \kappa+R\big)\setminus\left(\overline{H}_{\tau_n}\cup\bigcup_{\pi\in\Pi}\Lambda_\pi\right).\]
    By construction, we have that
        \begin{equation*}
            A_n\supset\{\#\mathcal N\cap E=0\}\cap\{\forall \pi\in\Pi,~\#\mathcal N\cap\Lambda_\pi=1\}.
        \end{equation*}
    Now, observe that thanks to Proposition \ref{prop_resampling}, conditionally on $\mathcal F_{\tau_n}$,
        \[\mathcal N\cap\mathbb H^+(m_{\tau_n})\setminus\overline{H}_{\tau_n}\]
    is distributed as $\mathcal N'\cap\mathbb H^+(m_{\tau_n})\setminus H_{\tau_n}$. Since $E$ and $(\Lambda_{\pi})_{\pi\in\Pi}$ are disjoint subsets of $\mathbb H^+(m_{\tau_n})\setminus\overline{H}_{\tau_n}$, we have then
        \begin{equation*}
            \begin{split}
                \P(A_n~|~\mathcal F_{\tau_n})&\geq\P(\#\mathcal N\cap E\quad\text{and}\quad\forall \pi\in\Pi,~\#\mathcal N\cap\Lambda_\pi=1~|~\mathcal F_{\tau_n})\\
                &=e^{-|E|}\prod_{\pi\in\Pi}\left(|\Lambda_\pi|e^{-|\Lambda_\pi|}\right).
            \end{split}
        \end{equation*}
    Since $|E|\leq k|B^+(0, \kappa+R)|$ by union bound and $a_0\leq |\Lambda_\pi|\leq |B^+(0,R)|$, we deduce that
        \[\P(A_n~|~\mathcal F_{\tau_n})\geq p_0\coloneqq e^{-k|B^+(0, \kappa+R)|}\left(\inf_{a_0\leq t\leq |B^+(0,R)|} te^{-t}\right)^k.\]
    Finally, if $m>n$, from lemma \ref{lemma_sfiltration} we have $\mathcal S_n\subset\mathcal F_{\tau_{n+1}}\subset\mathcal F_{\tau_m}$, and the last part of the statement follows. The proof is complete.
\end{proof}

We now use the previous result to show that the random variables $(\gamma_n)_{n\geq 0}$ are almost surely finite.

\begin{Lemma}
    \label{lemma_eta_stoppingtimes}
    The random variables $(\gamma_n)_{n\geq 0}$ are $(\mathcal S_i)_{i\geq 0}$ stopping times, and for all $n\geq 0$ and $m\geq 0$, we have
        \[\P(\gamma_{n+1}-\gamma_n>\ell~|~\mathcal S_{\gamma_n})\leq(1-p_0)^\ell.\]
\end{Lemma}

\begin{proof}
    First, observe that if $j\geq 0$, then 
        \[\{\gamma_n=j\}=A_j\cap\left\lbrace\sum_{i=0}^{j-1}\mathds 1_{A_i}=n-1\right\rbrace,\]
    which is $\mathcal S_j$ measurable. Hence $\gamma_n$ is a $(\mathcal S_i)_{i\geq 0}$ stopping time. Together with Lemma \ref{lemma_renewalproba} we then obtain that for all $j\geq0$, 
        \[\P(A_{\gamma_n+j+1}~|~\mathcal S_{\gamma_n+j})\geq p_0.\]
    Finally, recalling that $\gamma_{n+1}-\gamma_n$ represents the number of good step we need to wait to reach the next renewal step, it follows that
        \[\P(\gamma_{n+1}-\gamma_n>\ell~|~\mathcal S_{\gamma_n})=\E\left[\prod_{i=1}^\ell (1-\mathds 1_{A_{\gamma_n+i}})~\middle|~\mathcal S_{\gamma_n}\right]\leq(1-p_0)^\ell.\]
    The proof is complete.
\end{proof}

We now define a new filtration that will make the sequence $(\beta_n)_{n\geq 0}$ adapted to it, by setting
    \[\mathcal G_n\coloneqq\mathcal S_{\gamma_n}.\]
for all $n\geq0$. Before showing that the $(\beta_n)_{n\geq 0}$ are almost surely finite, we prove the following lemma, which will be useful several times.

\begin{Lemma}
    \label{lemma_EGn}
    Let $X$ be any $\sigma(\mathcal N)$-measurable non-negative random variable and $n\geq 0$. We have
        \[\E[X~|~\mathcal G_n]=\sum_{i\geq 0}\mathds 1_{\gamma_n=i}\frac{\E[X\mathds 1_{A_i}~|~\mathcal F_{\tau_i}]}{\P(A_i~|~\mathcal F_{\tau_i})}\leq\frac{1}{p_0}\sum_{i\geq 0}\mathds 1_{\gamma_n=i}\E[X~|~\mathcal F_{\tau_i}]\]
\end{Lemma}

\begin{proof}
    Since $\gamma_n$ is an almost surely finite $(\mathcal S_i)_{i\geq 0}$ stopping time by Lemma \ref{lemma_eta_stoppingtimes}, and that $A_{\gamma_n}$ occurs, we can write
        \begin{equation*}
            \begin{split}
                \E[X~|~\mathcal G_n]=\E[X~|~\mathcal S_{\gamma_n}]
                &=\sum_{i\geq 0}\mathds 1_{\gamma_n=i}\mathds 1_{A_i}\E[X~|~\sigma(\mathcal F_{\tau_i},~A_i)]\\
                &=\sum_{i\geq 0}\mathds 1_{\gamma_n=i}\frac{\E[X\mathds 1_{A_i}~|~\mathcal F_{\tau_i}]}{\P(A_i~|~\mathcal F_{\tau_i})}.
            \end{split}
        \end{equation*}
    Then, since for each $i\geq 0$, we have $X\mathds 1_{A_i}\leq X$ since $X\geq 0$ and $\P(A_i~|~\mathcal F_{\tau_i})\geq p_0$ by Lemma \ref{lemma_renewalproba}, we get the result.
\end{proof}

We can now prove that the random variables $(\beta_n)_{n\geq0}$ are almost surely finite and have spacing distributed with an exponentially decaying tail.

\begin{Proposition}
    \label{prop_nu_stoppingtimes}
    The random variables $(\beta_n)_{n\geq 0}$ are $(\mathcal G_n)_{n\geq 0}$ adapted, almost surely finite, and there exist $\alpha=\alpha(k, d, p)>0$ and $\beta=\beta(k, d,p)>0$ such that for all $n,\ell\geq 0$,
        \[\P(\beta_{n+1}-\beta_n>\ell~|~\mathcal G_n)\leq\alpha e^{-\beta \ell}.\]
\end{Proposition}

\begin{proof}
    First, observe that for all $j\geq 0$ we have
        \[\{\beta_n=j\}=\bigcup_{i\geq 0}\{\gamma_n=i,~\tau_i=j\},\]
    where, by definition, each event in the union is $\mathcal S_{\gamma_n}$ measurable. Hence $(\beta_n)_{n\geq 0}$ is $(\mathcal G_n)_{n\geq 0}$ adapted. Now for all $A>0$ and $\ell\geq 0$, observe that we can write
        \begin{equation}
            \label{eq_nudelta_decomp}
            \P(\beta_{n+1}-\beta_n>A\ell~|~\mathcal G_n)\leq\P(\gamma_{n+1}-\gamma_n>\ell~|~\mathcal G_n)+\P(\tau_{\gamma_n+m}-\tau_{\gamma_n}>A\ell~|~\mathcal G_n).
        \end{equation}
    Using Lemma \ref{lemma_EGn}, we have
        \begin{equation}
            \label{eq_gammaprobsum}
            \P(\tau_{\gamma_n+\ell}-\tau_{\gamma_n}>A\ell~|~\mathcal G_n)\leq\frac{1}{p_0}\sum_{i\geq 0}\mathds 1_{\gamma_n=i}\\P(\tau_{i+\ell}-\tau_i>A\ell~|~\mathcal F_{\tau_i}).
        \end{equation}
    From Theorem \ref{thm_good_steps}, there exist two constant $t, C>0$ such that, for all $n\geq 0$,
        \[\E[e^{t(\tau_{n+1}-\tau_n)}~|~\mathcal F_{\tau_n}]\leq C<\infty.\]
    Therefore, for all $i\geq 0$, using Markov's inequality we get
        \begin{equation*}
            \P(\tau_{i+\ell}-\tau_i>Am~|~\mathcal F_{\tau_i})
            \leq e^{-tA\ell}\E\left[\prod_{j=0}^{\ell-1}e^{t(\tau_{i+j+1}-\tau_{i+j})}~\middle|~\mathcal F_{\tau_i}\right]
            \leq e^{-tA\ell}C^\ell.
        \end{equation*}
    Now injecting in \eqref{eq_gammaprobsum} and then in \eqref{eq_nudelta_decomp}, using Lemma \ref{lemma_eta_stoppingtimes} we get that
        \[\P(\beta_{n+1}-\beta_n>A\ell~|~\mathcal G_n)\leq(1-p_0)^\ell+\frac{1}{p_0}e^{-tA\ell}C^\ell.\]
    The result follows then by fixing $A>0$ big enough so that $Ce^{-tA}<1$.
\end{proof}

\subsection{Block size}

\label{subsection_block_size}

In this subsection, we introduce a crucial quantity that will allow us to control the displacement of the trajectories between consecutive renewal times. For all integers $b>a\geq 0$ we define
    \[W(a,b)\coloneqq\sum_{i=a}^{b-1}\|x_i-\Psi(x_i)\|.\]
By construction, for each $i\in\range{1}{k}$, from triangle inequality,
    \[\|g_b(\mathbf u_i)-g_a(\mathbf u_i)\|\leq W(a, b).\]
The goal is to show the following proposition.

\begin{Proposition}
    \label{prop_block_size}
    For all $n\geq 0$, define 
        \[W_{n+1}\coloneqq W(\beta_n,\beta_{n+1}).\]
    Then there exist $\alpha=\alpha(k, d, p)>0$ and $\beta=\beta(k, d, p)>0$ such that for all $n\geq0$ and $t>0$,
        \[\P(W_{n+1}\geq t~|~\mathcal G_n)\leq\alpha e^{-\beta \sqrt t}.\]
\end{Proposition}

The random variables $(W_{n+1})_{n\geq 0}$ are called \emph{block sizes}, and they will later be used to control the range of interaction as well as the displacement of each trajectory between consecutive renewal steps. For each $n\geq 0$ and $i\in\range{1}{k}$ we have by construction that
    \[\|g_{\beta_{n+1}}(\mathbf u_i)-g_{\beta_n}(\mathbf u_i)\|\leq W_{n+1},\]
and furthermore,
    \[\bigcup_{j=\beta_n}^{\beta_{n+1}-1}B^+\big(g_j(\mathbf u_i), \|g_{j+1}(\mathbf u_i)-g_j(\mathbf u_i)\|\big)\subset B^+\big(g_{\beta_n}(\mathbf u_i), W_{n+1}\big).\]
In words, between step $\beta_n$ and $\beta_{n+1}$, each trajectory explores a region whose size is controlled by $W_{n+1}$. We start with the following lemma.

\begin{Lemma}
    \label{lemma_Wab}
    For all $b>a\geq 0$, we have
        \[\E\left[\exp\left(\frac{W(a, b)}{b-a}\right)~\middle\vert~\mathcal F_a\right]\leq e^{L_a}G_2(0)^{b-a}.\]
\end{Lemma}

\begin{proof}
    First, recall that for all $i\in\range{a}{b-1}$ we have $(\|x_i-\Psi(x_i)\|-L_i)_+=M_{i+1}-M_i$, giving
        \[\|x_i-\Psi(x_i)\|\leq L_i+M_{i+1}-M_i.\]
    Summing for $i$ between $a$ and $b-1$ we get
        \begin{equation}
            \label{eq_boundWab}
            W(a,b)\leq M_b-M_a+\sum_{i=a}^{b-1}L_i.
        \end{equation}
    Now observe that
        \begin{equation}
            \label{eq_sumLi}
            \begin{split}
                \sum_{i=a}^{b-1}L_i=\sum_{i=a}^{b-1}\left(L_a+\sum_{j=a}^{i-1}(L_{j+1}-L_j)\right)
                &=(b-a)L_a+\sum_{i=a}^{b-1}(b-1-i)(L_{i+1}-L_i)\\
                &\leq (b-a)L_a+(b-a)(M_b-M_a),
            \end{split}
        \end{equation}
    where in the last line we used that for all $i\in\range{a}{b-1}$,
        \[(b-1-i)(L_{i+1}-L_i)\leq(b-a)(M_{j+1}-M_j).\]
    Putting together \eqref{eq_boundWab} and \eqref{eq_sumLi}, we get that
        \[\frac{W(a, b)}{b-a}\leq \frac{M_b-M_a}{b-a}+L_a+M_b-M_a\leq L_a+2(M_b-M_a).\]
    Hence, taking the exponential and the expectation,
        \begin{equation*}
            \E\left[\exp\left(\frac{W(a, b)}{b-a}\right)~\middle\vert~\mathcal F_a\right]\leq e^{L_a}\E[e^{2(M_b-M_a)}~\vert~\mathcal F_a]\leq e^{L_a}G_2(0)^{b-a},
        \end{equation*}
    where for the last inequality we used Proposition \ref{lemma_GtL} to get that for all $i\in\range{a}{b-1}$, 
        \[\E[e^{2(M_{i+1}-M_i)}~|~\mathcal F_i]\leq G_2(L_i)\leq G_2(0).\]
    The proof is complete.
\end{proof}

We can now prove Proposition \ref{prop_block_size}.

\begin{proof}[Proof of Proposition \ref{prop_block_size}]
    Let $\ell\in\N$ and $A>0$. First, observe that we can write
        \begin{equation}
            \label{eq_probaWn}
            \P(W_{n+1}\geq A\ell^2~|~\mathcal G_n)\leq\P(\beta_{n+1}-\beta_n>\ell~|~\mathcal G_n)+\P(W(\beta_n,\beta_n+\ell)\geq A\ell^2~|~\mathcal G_n).
        \end{equation}
    Using Lemma \ref{lemma_EGn}, we can then write
        \begin{equation*}
            \begin{split}
                \P(W(\beta_n,\beta_n+\ell)\geq A\ell^2~|~\mathcal G_n)&\leq\frac{1}{p_0}\sum_{i\geq0}\mathds 1_{\gamma_n=i}\P(W(\tau_i,\tau_i+\ell)>A\ell^2~|~\mathcal F_{\tau_i})\\
                &\leq\frac{e^{-A\ell}}{p_0}\sum_{i\geq0}\mathds 1_{\gamma_n=i}\E\left[\exp\left(\frac{W(\tau_i,\tau_i+\ell)}{\ell}\right)~\middle|~\mathcal F_{\tau_i}\right]\\
                &\leq\frac{e^{-A\ell}G_2(0)^\ell}{p_0}\sum_{i\geq0}\mathds 1_{\gamma_n=i}e^{L_{\tau_i}}\\
                &\leq\frac{e^\kappa}{p_0}e^{-A\ell}G_2(0)^\ell,
            \end{split}
        \end{equation*}
    where the second inequality follows from Markov's inequality, the third one from Lemma \ref{lemma_Wab}, and the last inequality uses that $L_{\tau_i}\leq\kappa$ for all $i\geq0$ by definition of $\tau_i$. Now injecting in \eqref{eq_probaWn} and using Proposition \ref{prop_nu_stoppingtimes}, we get
        \[\P(W_{n+1}\geq A\ell^2~|~\mathcal G_n)\leq\alpha e^{-\beta \ell}+\frac{e^\kappa}{p_0}e^{-A\ell}G_2(0)^\ell.\]
    Finally the result follows by choosing $A>0$ big enough so that $e^{-A}G_2(0)<1$.
\end{proof}

\subsection{The independent process}

\label{subsection_independent_process}

In this subsection, we introduce a new process, referred to as the \emph{independent process}, which intuitively captures the behavior of trajectories when they evolve independently of one another. The rough idea behind this construction is inspired by \cite{DiscreteDsf}. Its definition closely parallels that of the original exploration process, with the key difference being that the $k$ trajectories evolve in \emph{independent environments} rather than a shared one. This process, together with renewal times and block sizes, will allow us to rigorously formalize the heuristic that trajectories behave approximately independently when they are sufficiently far apart.\\

Here is the construction of the independent process. Let $(\mathcal N_i)_{i\in\range{1}{k}}$ be i.i.d.\ random variables distributed as $\mathcal N'$ conditioned on the event
    \[\#\mathcal N'\cap B^+(0,R)=\#\mathcal N'\cap B(-\kappa e_d,\kappa+R)\cap\mathbb H^+(0)=1,\]
which is a centered version of the renewal event for a single trajectory. These random variables represent the independent environments in which each trajectory will evolve. We define the independent process $((g'_n(i))_{i\in\range{1}{k}})_{n\geq 0}$ by induction as follows. For each $i\in\range{1}{k}$, we set  $g_0'(i)\coloneqq 0$. Note that we have re-centered everything for simplicity: since the trajectories will not interact, their absolute positions play no role in the analysis, and placing them all at the origin serves to streamline the construction. Let $n\geq0$ and assume that $(g_n'(i))_{i\in\range{1}{k}}$ is constructed. We set
    \[m'_n\coloneqq\min_{1\leq i\leq k}\big(g'_n(i)\cdot e_d\big),\]
and we fix $i_n\in\range{1}{k}$ such that $g'_n(i_n)\cdot e_d=m'_n$. In case of multiple possibilities, we chose one deterministically. Then, for each $i\in\range{1}{k}$ and $n\geq0$ we define
    \[g'_{n+1}(i)=\begin{cases}
        \Psi_{\mathcal N_i}\big(g_n'(i)\big) & \text{if $i=i_n$}\\
        g'_n(i) & \text{otherwise,}
    \end{cases}\]
where $\Psi_{\mathcal N_i'}(g_n'(i))$ denotes the closest point to $g'_n(i)$ in $\mathbb H^+(m_n')\cap\mathcal N_i$ with respect to the $\ell^p$ distance. In words, analogously to the original exploration process, at each step we advance the trajectory whose current position has the lowest $e_d$-coordinate. Now that the independent exploration process is constructed, we define some quantities of interest. For all $n\geq 0$ and each $i\in\range{1}{k}$ we define an \emph{individual history set}
    \[H_n^i\coloneqq\mathbb H^+(m'_n)\cap\bigcup_{j=0}^{n-1}B^+\big(g_n'(i), \|g_{n+1}'(i)-g_n'(i)\|\big),\]
as well as an \emph{individual history size}
    \[L_n^i\coloneqq\max(\{0\}\cup\{x\cdot e_d-m'_n:x\in H_n^i\}).\]
For each $n\geq 0$ and $i\in\range{1}{k}$ we set
    \[g_n^+(i)=g_n(i)+(m_n+\kappa-g_n(i)\cdot e_d)e_d\quad\text{and}\quad g_n^-(i)=g_n(i)+(m_n-g_n(i)\cdot e_d)e_d,\]
and we define
    \begin{equation*}
        \begin{split}
            \zeta\coloneqq\inf\{n\geq 0:\quad &m'_n\geq R,~\forall i\in\range{1}{k},~L_n^i\leq\kappa,\\
            &\text{and}\quad\forall i\in\range{1}{k},~\#\mathcal N_i\cap B^+(g_n^-(i), \kappa+R)\setminus \overline {H}_n^i=\#B^+(g_n^+(i), R)=1\}.
        \end{split}
    \end{equation*}
This random index is the analog of the first renewal time in the exploration process. Finally for all $1\leq i<j\leq k$, we define
    \[\Delta^{i,j}\coloneqq\mathds 1_{\zeta<\infty} \mathbf p\big(g'_\zeta(i)-g'_\zeta(j)\big),\]
In words, $\Delta^{i,j}$ represent the drift from time $0$ to $\zeta$ between trajectory $i$ and $j$. We also set
    \[W\coloneqq\sum_{\substack{i\in\range{1}{k}\\0\leq n<\zeta}}\|g'_{n+1}(i)-g'_n(i)\|,\]
so that by construction, for all $1\leq i<j\leq k$,
    \[\|\Delta^{i, j}\|\leq W.\]
Before relating the independent process to the exploration process, we establish some properties on the distribution of $(\Delta^{i, j})_{1\leq i<j\leq k}$ in the following lemma.

\begin{Lemma}
    \label{lemma_Delta}
    The $\R^{d-1}$-valued random variables $(\Delta^{i, j})_{1\leq i<j\leq k}$ are identically distributed, and their common distribution is sign-symmetric and exchangeable. More precisely, for any permutation $\sigma$ of $\range{1}{d-1}$ and sign function $\varepsilon:\range{1}{d-1}\to\{-1,1\}$, the random variable
        \[\big((\varepsilon(s)\Delta^{1,2}\cdot e_{\sigma(s)})_{1\leq s<d}, W\big)\]
    is distributed as $(\Delta^{1,2}, W)$.
\end{Lemma}

\begin{proof}
    To begin, observe that the only element in the construction that could prevent the variables to be identically distributed is the dissymmetry introduced when choosing $i_n$ for $n\geq0$ when there is a tie, i.e.\ more than one index $i\in\range{1}{k}$ such that $g'_n(i)\cdot e_d=m'_n$. Then, observe that almost surely, ties occur only in the $k$ first steps. In fact, for each $i\in\range{1}{k}$, $g_0(i)\cdot e_d=0$, thus $\zeta\geq k$ and 
        \[g_k'(i)=\Psi_{\mathcal N_i}(0)\]
    for each $i\in\range{1}{k}$. The variables $(g_k(i)\cdot e_d)_{i\in\range{1}{k}}$ are thus i.i.d.\ and almost surely distinct. By induction, we then get that, almost surely, for each $n\geq k$, the random variables $(g_k'(i)\cdot e_d)_{i\in\range{1}{k}}$ are distinct. Hence, we deduce that the random variables $((g'_n(i))_{n\geq k})_{i\in\range{1}{k}}$ are identically distributed and the first statement follows. It remains to check that $\Delta^{1,2}$ is exchangeable and sign-symmetric. Let $\sigma$ be a permutation of $\range{1}{d-1}$ and $\varepsilon:\range{1}{d-1}\to\{-1,1\}$ be a sign function. Let us define
        \[S=S_{\sigma,\epsilon}:\R^d\to\R^d,~x\mapsto (x\cdot e_d)e_d+\sum_{s=1}^{d-1} \varepsilon(s)(x\cdot e_{\sigma(s)})e_s.\]
    Since $S$ preserves the Lebesgue measure, we have that $S(\mathcal N')$ is distributed as $\mathcal N'$. Hence, we deduce that
        \[\big(S(\mathcal N_i)\big)_{i\in\range{1}{k}}.\]
    is distributed as $(\mathcal N_i)_{i\in\range{1}{k}}$. Finally, we conclude the proof by observing that we obtain almost surely
        \[\big((\varepsilon(s)\Delta^{1,2}\cdot e_{\sigma(s)})_{1\leq s<d}, W\big)\]
    by constructing the independent process using the transformed environments $(S(\mathcal N_i))_{i\in\range{1}{k}}$.
\end{proof}

Using the independent process as well as renewal times and block sizes, we conclude the section by showing the following theorem.

\begin{Theorem}
    \label{thm_renewal_decomposition}
    Let $n\geq 0$. For all $1\leq i<j\leq k$, we define
        \[Z_n^{i, j}\coloneqq \mathbf p\big(g_{\beta_n}(\mathbf u_i)-g_{\beta_n}(\mathbf u_j)\big),\quad\text{and we set}\quad r_n\coloneqq\frac{1}{2}\min_{1\leq i<j\leq k}\|Z_n^{i,j}\|.\]
    Then, conditionally on $\mathcal G_n$, when $r_n\geq \kappa+R$,
        \[\big(W_{n+1}\wedge r_n,~\big((Z^{i,j}_{n+1}-Z^{i,j}_n)\mathds 1_{W_{n+1}<r_n}\big)_{1\leq i<j\leq k}\big)\quad\text{is distributed as}\quad\big(W\wedge r_n,~(\Delta^{i,j}\mathds 1_{W<r_n})_{1\leq i<j\leq k}\big).\]
    Furthermore, there exists $\alpha=\alpha(k, d, p)>0$ and $\beta=\beta(k, d, p)>0$ such that for all $t\geq0$,
        \[\P(W\geq t)\leq\alpha e^{-\beta \sqrt t}.\]
\end{Theorem}

\begin{proof}
    Let $(\mathcal N_i')_{i\in\range{1}{k}}$ be i.i.d.\ copies of $\mathcal N$. From now on, we work conditionally on $\mathcal G_n$. Assume $r_n\geq\kappa+R$. For each $i\in\range{1}{k}$ we set
        \[\mathcal N_\mathrm c^i\coloneqq(\mathcal N_i'\setminus A)\cup \big((\mathcal N-g_{\beta_n}^\uparrow(\mathbf u_i))\cap A\big)\quad\text{where}\quad A\coloneqq (-r_n, r_n]^{d-1}\times\R_+.\]
    We refer to Figure \ref{fig_delta} for an illustration. Let us check that the random variables $(\mathcal N_\mathrm c^i)_{i\in\range{1}{k}}$ are independent and identically distributed as $\mathcal N'$ conditioned on the event
        \[\#\mathcal N'\cap B^+(0,R)=\#\mathcal N'\cap B(-\kappa e_d,\kappa+R)\cap\mathbb H^+(0)=1.\]
    Thanks to Lemma \ref{lemma_EGn} and Proposition \ref{prop_resampling},
        \[\mathcal N\cap\mathbb H^+(m_n+\kappa)\]
    is distributed as $\mathcal N'\cap\mathbb H^+(m_n+\kappa)$ conditioned on 
        \[\{\forall i\in\range{1}{k},~\#\mathcal N'\cap B^+(g_{\beta_n}^\uparrow(\mathbf u_i),R)=\#\mathcal N'\cap B(g_{\beta_n}^\downarrow(\mathbf u_i),\kappa+R)\cap\mathbb H^+(m_n+\kappa)=1\}.\]
    Since $r_n\geq\kappa+R$ and the sets $(g_{\beta_n}^\uparrow(\mathbf u_i)+A)_{i\in\range{1}{k}}$ are disjoints by construction, this condition transfers to the constructed variables $(\mathcal N_c^i)_{i\in\range{1}{k}}$, yielding the claimed distribution. Now, let $((\Delta^{i, j}_\mathrm c)_{1\leq i<j\leq k}, W_\mathrm c)$ be defined by constructing the independent process using the environments $(\mathcal N_\mathrm c^i)_{i\in\range{1}{k}}$. By construction, we have
        \[W_{n+1}\wedge r_n=W_\mathrm c\wedge r_n,\]
    and if $W_{n+1}<r_n$ then no trajectory will explore the resampled regions, hence for all $1\leq i<j\leq k$,
        \[Z_{n+1}^{i, j}-Z_n^{i, j}=\Delta^{i, j}_\mathrm c,\]
    implying the first part of the theorem. To get the last statement, let $t>0$,  $M\geq\kappa+R$ and consider $k=2$, $n=0$, $\mathbf u_1=-Me_1$ and $\mathbf u_2=Me_1$. The event $A_0$, on which $\beta_0=0$ and thus $r_0=M\geq\kappa+R$, occurs with positive probability thanks to Lemma \ref{lemma_renewalproba}. On that event, we then have
       \begin{equation*}
           \P(W\wedge M>t)=\P(W_0\wedge M>t~|~\mathcal G_0)\leq\P(W_0>t~|~\mathcal G_0)\leq\alpha e^{-\beta\sqrt t}.
       \end{equation*}
   Finally, since $M$ was arbitrary, we get the sought result. The proof is complete.
\end{proof}

\begin{figure}[h]
    \centering
    \includegraphics[width=\linewidth]{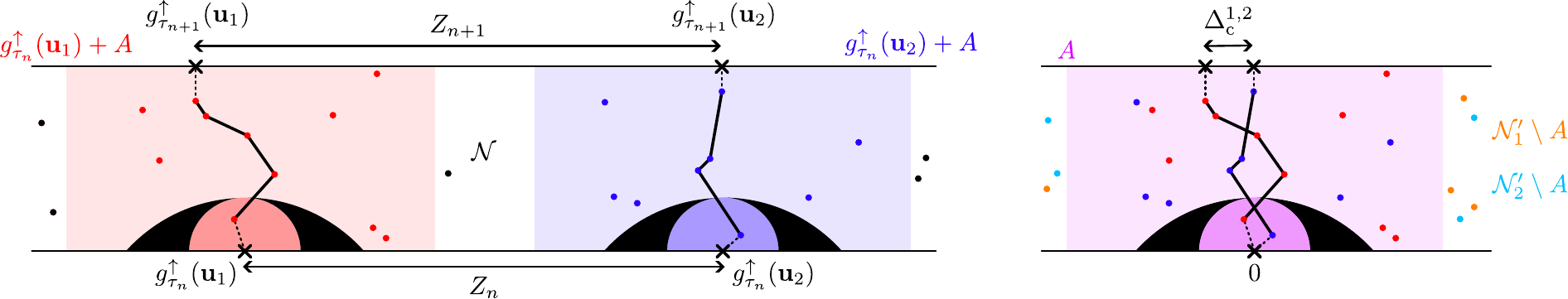}
    \caption{Illustration of the construction behind the proof of Theorem \ref{thm_renewal_decomposition}.}
    \label{fig_delta}
\end{figure}

\section{Coalescence in low dimensions}

\label{section_coalescence}

In this section, we establish the coalescence statement of our main result, stated precisely in the following theorem.

\begin{Theorem}
    \label{thm_coalescence}
    Assume that $d=2$ with $p\in[1, \infty]$ or $d=3$ with $p\in\{1, 2, \infty\}$. Then the DSF is almost surely a tree.
\end{Theorem}

In the remaining of this section, we assume that $d=2$ with $p\in[1, \infty]$ or $d=3$ with $p\in\{1, 2, \infty\}$.

To prove the theorem, we analyze the case $k=2$ of the exploration process, starting from two fixed deterministic points $\mathbf u_1, \mathbf u_2\in\R^d$ satisfying $\mathbf u_1\cdot e_d=\mathbf u_2\cdot e_d$. Our goal is to show that the two trajectories coalesce, which amounts to proving that
    \[(Z_n)_{n\geq 0}\coloneqq(Z_n^{1, 2})_{n\geq 0}\]
reaches eventually $0$, almost surely. We start by presenting a sketch of the proof. The behavior of the $(\mathcal G_n)_{n\geq 0}$-adapted process $(Z_n)_{n\geq 0}$ far from the origin can be approximated by a process with i.i.d.\ increments. More precisely, thanks to Theorem \ref{thm_renewal_decomposition}, for all $n\geq0$, conditionally on $\mathcal G_n$, if $\|Z_n\|\geq2(\kappa+R)$,
    \[(Z_{n+1}-Z_n)\mathds 1_{2W_{n+1}<\|Z_n\|}\quad\text{is distributed as}\quad\Delta\mathds 1_{2W<\|Z_n\|},\quad\text{where}\quad \Delta\coloneqq\Delta^{1,2}.\]
The distribution of the $\R^{d-1}$-valued random variable $\Delta$ is sign-symmetric and exchangeable by Lemma \ref{lemma_Delta} and $\|\Delta\|$ admits moments of all orders from Theorem \ref{thm_renewal_decomposition} as $\|\Delta\|\leq W$.

Then, the overall approach is similar to the strategy developed in \cite[Section 4]{DiscreteDsf}. We want to use the increment approximation far from $0$ in order to find a function $V:\R^{d-1}\to\R_+$ satisfying $V(z)\to\infty$ as $z\to\infty$, and a constant $M>0$ such that, for all $n\geq0$,
    \[\mathds 1_{\|Z_n\|>M}\E[V(Z_{n+1})-V(Z_n)~\vert~\mathcal G_n]\leq 0.\]
Such a function $V$ is called a \emph{Lyapunov function} (see e.g.\ \cite{meyn2012markov}). We will see that its existence ensures the process does not escape to infinity. Then, since for any fixed $K>0$, whenever $\|Z_n\|_p\leq K$ for some $n\geq0$, there is a probability uniformly bounded away from $0$ of coalescing in one step, we will deduce that coalescence occurs almost surely. In the case of simple random walks on $\Z$ or $\Z^2$, functions of the form $z\mapsto\log^\alpha(1+\|z\|_2)$ for $\alpha\in(0,1)$ can be used as Lyapounov function. In our case, for computational convenience, we choose instead
    \begin{equation}
        \label{eq_def_V}
        V:\R^{d-1}\to\R_+,~z\mapsto\log\log(e+\|z\|_2^2).
    \end{equation}
We emphasize that we use the Euclidean norm $\|\cdot\|_2$ in the definition of $V$, even when $p\neq 2$, as it simplifies later computations. In order to avoid ambiguity, we will often denote the $\ell^p$ norm by $\|\cdot\|_p$ rather than $\|\cdot\|$ in what follows. Note that for each $n\geq 0$, $V(Z_n)$ is integrable as $V(z)=o_{z\to\infty}(\|z\|_p)$ and
    \[\|Z_n\|_p\leq Z_0+\sum_{i=0}^{n-1}W_i.\]
To prepare for the Taylor expansion used in the analysis of the increments of the process $(V(Z_n))_{n\geq0}$, we begin by computing some expectations involving the asymptotic increment variable $\Delta$.

\begin{Lemma}
    \label{lemma_moment_Xy}
    Let $y\in\R^{d-1}$. Define 
        \[X(y)\coloneqq\|y+\Delta\|_2^2-\|y\|_2^2.\]
    Then we have 
        \begin{itemize}
            \item $E[X(y)]=(d-1)\E[(\Delta\cdot e_1)^2]$,
            \item $\E[X(y)^2]\geq 4\|y\|_2^2\E[(\Delta\cdot e_1)^2]$,
            \item and $\E[X(y)^3]=\mathcal O_{\|y\|_2\to\infty}(\|y\|_2^2)$.
        \end{itemize}
\end{Lemma}

\begin{proof}
    First, observe that we can write
        \begin{equation}
            \label{eq_Xy}
            X(y)=(y+\Delta)\cdot(y+\Delta)-y\cdot y=\|\Delta\|_2^2+2y\cdot\Delta.
        \end{equation}
    Taking the expectation, since $\Delta$ is symmetric and exchangeable, we get
        \[\E[X(y)]=\E[\|\Delta\|_2^2]=\sum_{s=1}^{d-1}\E[(\Delta\cdot e_s)^2]=(d-1)\E[(\Delta\cdot e_1)^2].\]
    Now, squaring \eqref{eq_Xy} and using symmetry, we get
        \begin{equation*}
            \begin{split}
                \E[X(y)^2]=\E[\|\Delta\|_2^4]+4\E[(y\cdot\Delta)^2]\geq4\E[(y\cdot\Delta)^2]=4\|y\|_2^2\E[(\Delta\cdot e_1)^2],
            \end{split}
        \end{equation*}
    where for the last inequality we used that $(\Delta\cdot e_{s_1})(\Delta\cdot e_{s_2})$ is symmetric for all $s_1,s_2\in\range{1}{d-1}$ with $s_1\neq s_2$ from Lemma \ref{lemma_Delta}. Similarly, from symmetry of $\Delta$ and using \eqref{eq_Xy}, we can write
        \begin{equation}
            \label{eq_Xy3}
            \begin{split}
                \E[X(y)^3]&=\E[(\|\Delta\|_2^2+2y\cdot\Delta)^3]\\
                &=\E[\|\Delta\|_2^6]+3\E[(2y\cdot\Delta)^2\|\Delta\|_2^2]\\
                &=\E[\|\Delta\|_2^6]+12\|y\|_2^2\E[(\Delta\cdot e_1)^2\|\Delta\|_2^2],
            \end{split}
        \end{equation}
    where for the last inequality we used that $(\Delta\cdot e_{s_1})(\Delta\cdot e_{s_2})\|\Delta\|_2^2$ is symmetric for all $s_1,s_2\in\range{1}{d-1}$ with $s_1\neq s_2$. The proof is complete.
\end{proof}

We now define the function
    \[\delta:\R^{d-1}\to\R,~y\mapsto\E[V(y+\Delta)-V(y)],\]
where $V$ is defined in \eqref{eq_def_V}. This function will be central to approximate $\E[V(Z_{n+1})-V(Z_n)~\vert~\mathcal G_n]$ when $\|Z_n\|_p$ is large. The following lemma gives an asymptotic estimate of $\delta$ using a Taylor expansion. It gives in particular that $\delta(y)$ is negative when $\|y\|_2$ is big enough. Note that, as in the case of simple walks in $\R^{d-1}$, this is where the fact that $d\leq 3$, or equivalently $d-1\leq2$, plays a role.

\begin{Lemma}
    \label{lemma_RW}
    Assume $d\leq 3$. Then we have that 
        \[\delta(y)\leq-\frac{2\E[(\Delta\cdot e_1)^2]+o_{y\to\infty}(1)}{\|y\|_2^2(\log\|y\|_2^2)^2}\]
\end{Lemma}

\begin{proof}
    Let $f:(1,\infty)\to\R,~ z\mapsto\log\log z$. From direct computations, for all $z>1$ we have
    \begin{equation}
        \label{eq_fderivatives}
        f^{(1)}(z)=\frac{1}{z\log z},\quad f^{(2)}(z)=-\frac{\log z+1}{(z\log z)^2}, \quad f^{(3)}(z)=\frac{2+o_{z\to\infty}(1)}{z^3\log z}\quad\text{and}\quad f^{(4)}(z)\leq 0.
    \end{equation}
    Let $y\in\R^{d-1}$. We set
        \[z\coloneqq e+\|y\|_2^2.\]
    Since $f^{(4)}$ is non-positive, using a Taylor expansion at order $3$ we get that
        \begin{equation}
            \label{eq_taylor}
            \begin{split}
                V(y+\Delta)-V(y)&=f[z+X(y)]-f(z)\\
                &\leq X(y) f^{(1)}(z)+\frac{X(y)^2}{2}f^{(2)}(z)+\frac{X(y)^3}{6}f^{(3)}(z).
            \end{split}
        \end{equation}
    Now taking the expectation on both sides and using Lemma \ref{lemma_moment_Xy} as well as $f^{(2)}\leq 0$, we get
        \begin{equation*}
            \begin{split}
                \delta(y)&=\E[V(y+\Delta)-V(y)]\\
                &\leq (d-1)\E[(\Delta\cdot e_1)^2]f^{(1)}(z)+2\|y\|_2^2\E[(\Delta\cdot e_1)^2]f^{(2)}(z)+f^{(3)}(z)\mathcal O_{z\to\infty}(z).
            \end{split}
        \end{equation*}
    Now since $\|y\|_2^2=z-e$ we can rewrite the inequality as
        \begin{equation}
            \label{eq_Eyineq}
            \delta(y)\leq 2\E[(\Delta\cdot e_1)^2]\left(\frac{d-1}{2}f^{(1)}(z)+zf^{(2)}(z)-ef^{(2)}(z)\right)+f^{(3)}(z)\mathcal O_{z\to\infty}(z).
        \end{equation}
    Now, observe that since $d-1\leq 2$ and $f^{(1)}\geq 0$, then 
        \[\frac{d-1}{2}f^{(1)}(z)+zf^{(2)}(z)\leq f^{(1)}(z)+zf^{(2)}(z)=-\frac{1}{z\log^2z}.\] 
    Moreover,
        \[f^{(2)}(z)=o_{z\to\infty}\left(\frac{1}{z\log^2z}\right)\quad\text{and}\quad zf^{(3)}(z)=o_{z\to\infty}\left(\frac{1}{z\log^2z}\right),\]
    Hence injecting in \eqref{eq_Eyineq} we get
        \[\delta(y)\leq-\frac{2\E[(\Delta\cdot e_1)^2]+o_{z\to\infty}(1)}{z\log^2z},\]
    and the sought result follows by definition of $z$. The proof is complete.
\end{proof}

Using the previous lemma, we now prove that $V$ is indeed a Lyapunov function.

\begin{Proposition}
    \label{prop_lyapounov}
    There exist a constant $M>0$ such that for all $n\geq0$ 
        \[\mathds 1_{\|Z_n\|_p>M}\E[V(Z_{n+1})-V(Z_n)~\vert~Z_n]\leq0.\]
\end{Proposition}

\begin{proof}
    Let $n\geq 0$ and assume $\|Z_n\|_p\geq 2(\kappa+R)$. Using Theorem \ref{thm_renewal_decomposition}, we have that
        \[\E[V(Z_{n+1})\mathds 1_{2W_{n+1}<\|Z_n\|_p}~|~\mathcal G_n]=h(Z_n),\]
    where
        \[h:\R^{d-1}\to\R,~y\mapsto\E[V(y+\Delta)\mathds 1_{2W<r_n}]\leq\E[V(y+\Delta)],\]
    as $V\geq0$. Therefore, we deduce that
        \begin{equation}
            \label{eq_diffgoodevent}
            \E[V(Z_{n+1})\mathds 1_{2W_{n+1}<\|Z_n\|_p}-V(Z_n)~|~\mathcal G_n]\leq\delta(Z_n).
        \end{equation}
    By equivalence of norms on $\R^{d-1}$, there exist $C>0$ such that $\|\cdot\|_2\leq C\|\cdot\|_p$. By construction, we have $\|Z_{n+1}-Z_n\|_p\leq W_{n+1}$, hence using triangle inequality we get
        \[\|Z_{n+1}\|_2\leq C\|Z_{n+1}\|_p\leq C(\|Z_n\|_p+W_{n+1}).\]
    This imply that on the event $\{2W_{n+1}\geq\|Z_n\|_p\}$, $\|Z_{n+1}\|_2\leq 3CW_{n+1}$ hence using the definition of $V$, we get that
        \[V(Z_{n+1})\mathds 1_{2W_{n+1}\geq\|Z_n\|_p}\leq g(W_{n+1})\mathds 1_{2W_{n+1}\geq\|Z_n\|_p}.\]
    where
        \[g:\R_+\to\R_+,~z\mapsto\log\log(e+9C^2z^2).\]
    Using Proposition \ref{prop_block_size}, since $g(z)=o_{z\to\infty}(\log z)$, there exists $C'<\infty$ such that for all $m\geq 0$, $\E[g(W_m)^2~|~\mathcal G_m]\leq C'$. Then, using a Cauchy-Schwarz inequality, we get
        \begin{equation}
            \begin{split}
                \label{eq_diffbadevent}
                \E[g(W_{n+1})\mathds 1_{2W_{n+1}\geq \|Z_n\|_p}~|~\mathcal G_n]^2&\leq\P[2W_{n+1}\geq\|Z_n\|_p~|~\mathcal G_n]\E[g(W_{n+1})^2\mathds 1_{2W_{n+1}\geq\|Z_n\|_p}~|~\mathcal G_n]\\
                &\leq\alpha C'e^{-\beta\sqrt{\|Z_n\|_p/2}}=\tilde\delta(Z_n)^2,
            \end{split}
        \end{equation}
    where $\alpha$ and $\beta$ are given by Proposition \ref{prop_block_size} and 
        \[\tilde\delta:\R^{d-1}\to\R,~y\mapsto\sqrt{\alpha C'}e^{-\frac{\beta}{2}\sqrt{\|Z_n\|_p/2}}.\]
    Combining \eqref{eq_diffbadevent} with \eqref{eq_diffgoodevent} we get
        \[\E[V(Z_{n+1})-V(Z_n)~|~\mathcal G_n]\leq \delta(Z_n)+\tilde\delta(Z_n).\]
    Finally, since $\tilde\delta(y)=o_{y\to\infty}(\delta(y))$ and $\delta(y)\to 0^-$ as $y\to\infty$, we deduce that we can find $M\geq2(\kappa+R)$ such that for all $n\geq 0$, 
        \[\big(\delta(Z_n)+\tilde\delta(Z_n)\big)\mathds 1_{\|Z_n\|_p>M}\leq 0.\]
    The proof is complete.

\end{proof}

We now prove from the existence of the Lyapunov function that $(Z_n)_{n\geq 0}$ can not escape to infinity.

\begin{Lemma}
    \label{lemma_notescape}
    Almost surely, $Z_n\not\to\infty$ as $n\to\infty$.
\end{Lemma}

\begin{proof}
    For all $n_0\geq0$, we define
        \[\zeta_{n_0}\coloneqq\inf\{n\geq n_0:\|Z_n\|\leq M\},\]
    Where $M$ is given by Proposition \ref{prop_lyapounov}. Now observe that if $Z_n\to\infty$ as $n\to\infty$, there exist $n_0\in\N$ such that for all $n\geq n_0$, $\|Z_n\|>M$. In that case $\zeta_{n_0}=\infty$ and since $V(z)\to\infty$ as $z\to\infty$, 
        \[\liminf_{n\to\infty} V(Z_{n\wedge\zeta_{n_0}})=\liminf_{n\to\infty} V(Z_n)=\infty.\]
    To show that $Z_n\not\to\infty$ as $n\to\infty$, it is therefore enough to check that for all $n_0\in\N$, we have
        \begin{equation}
            \label{eq_liminf_finite}
            \liminf_{n\to\infty} V(Z_{n\wedge\zeta_{n_0}})<\infty\quad\text{almost surely}.
        \end{equation}
    To conclude the proof, let us show that it is indeed the case. Fix $n_0\in\N$. Let $n\geq n_0$. The event $\zeta_{n_0}>n$ is $\mathcal G_n$ measurable, hence we can write
        \begin{equation*}
            \E[V(Z_{(n+1)\wedge\zeta_{n_0}})-V(Z_{n\wedge\zeta_{n_0}})]=\E(\mathds 1_{\zeta_{n_0}>n}\E[V(Z_{n+1})-V(Z_n)~\vert~\mathcal G_n])\leq0
        \end{equation*}
    where the last inequality follows from Proposition \ref{prop_lyapounov} since $\|Z_n\|_p>M$ on the event $\zeta_{n_0}>n$ by definition. We deduce then by induction that for all $n\geq n_0$,
        \[\E[V(Z_{n\wedge\zeta_{n_0}})]\leq\E[V(Z_{n_0})].\]
    Finally using Fatou's Lemma, since $V\geq 0$, we get
        \[\E\left[\liminf_{n\to\infty}V(Z_{n\wedge\zeta_{n_0}})\right]\leq\liminf_{n\to\infty}\E[V(Z_{n\wedge\zeta_{n_0}})]\leq\E[V(Z_{n_0})]<\infty.\]
    This proves \eqref{eq_liminf_finite}. The proof is complete.
\end{proof}

We now show that for all $K>0$, coalescence in one step occurs with probability bounded away from zero at any step $n\geq 0$, provided $\|Z_n\|\leq K$.

\begin{Lemma}
    \label{lemma_coalescence_proba}
    Let $K>0$. There exist $q_K<1$ such that for all for all $n\geq 0$,
        \[\mathds 1_{\|Z_n\|\leq K}\P(Z_{n+1}\neq0~\vert~\mathcal G_n)\leq q_K.\]
\end{Lemma}

\begin{proof}
    Let $n\geq 0$. We set
        \[U_n=\bigcup_{i\in\{1,2\}}B^+\big(g_{\beta_n}^\uparrow(\mathbf u_i),R\big),\]
    and
        \[V_n\coloneqq\bigcup_{i\in\{1,2\}}B^+\big(g_{\beta_n}^\uparrow(\mathbf u_i),K+R\big)\setminus U_n.\]
    Thanks to Lemma \ref{lemma_EGn} and Proposition \ref{prop_resampling}, conditionally on $\mathcal G_n$,
        \[\mathcal N\cap\mathbb H^+(m_n+\kappa)\setminus U_n\]
    is distributed as $\mathcal N'\cap\mathbb H^+(m_n+\kappa)\setminus U_n$. Therefore, we have that
        \[\P(\#\mathcal N\cap V_n=0~|~\mathcal G_n)=e^{-|V_n|}\geq e^{-2|B^+(0,K+R)|}.\]
    Now observe that if $Z_n\leq K$, then 
        \[B^+\big(g_{\beta_n}^\uparrow(\mathbf u_2),R\big)\subset B^+\big(g_{\beta_n}^\uparrow(\mathbf u_1),K+R\big)\quad\text{and}\quad B^+\big(g_{\beta_n}^\uparrow(\mathbf u_1),R\big)\subset B^+\big(g_{\beta_n}^\uparrow(\mathbf u_2),K+R\big).\]
    Therefore, if $Z_n\leq K$ and $\#\mathcal N\cap V_n=0$ and if $x$ and $y$ respectively denotes the only points of $\mathcal N$ in $B^+(g_{\beta_n}^\uparrow(\mathbf u_1),R)$ and $B^+(g_{\beta_n}^\uparrow(\mathbf u_2),R)$, then
        \[\Psi(x)=y\quad\text{or}\quad\Psi(y)=x\quad\text{or}\quad x=y.\]
    This ensures that when $Z_n\leq K$ and $\#\mathcal N\cap V_n=0$, either coalescence has occurred before time $\beta_n+2$, either $x\neq y$ and coalescence will occur at time $\beta_n+3\leq\beta_{n+1}$. It follows that,
        \begin{equation*}
            \begin{split}
                \mathds 1_{\|Z_n\|\leq K}\P(Z_{n+1}\neq0~\vert~\mathcal G_n)&\leq\mathds 1_{\|Z_n\|\leq K}\P(\#\mathcal N\cap V_n\neq0~\vert~\mathcal G_n)\\
                &\leq(1-e^{-2|B^+(0,K+R)|})\eqqcolon q_K<1.
            \end{split}
        \end{equation*}
    The proof is complete.
\end{proof}

We now derive that almost surely, the trajectories eventually coalesce.

\begin{Proposition}
    \label{prop_coalescence}
    Almost surely, the process $(Z_n)_{n\geq 0}$ eventually hits $0$, meaning that the trajectories from $\mathbf u_1$ and $\mathbf u_2$ coalesce.
\end{Proposition}

\begin{proof}
    Since $\liminf_{n\to\infty}\|Z_n\|$ is almost surely finite from Lemma \ref{lemma_notescape}, we have
        \begin{equation}
            \label{eq_coal_liminfK}
            \P(\forall n\geq 0,~Z_n\neq0)=\lim_{K\to\infty}\P\left(\liminf_{n\to\infty}\|Z_n\|<K,~\forall n\geq 0,~Z_n\neq0\right).
        \end{equation}
    Fix $K>0$. Denote $\theta_0=\inf\{n\geq 0~:~\|Z_n\|\leq K\}$ and for all $\ell\geq 0$,
        \[\theta_{\ell+1}=\inf\{n>\theta_\ell~:~\|Z_n\|\leq K\}.\]
    Observe that for all $\ell\geq 1$ we can write
        \begin{equation}
            \begin{split}
                \P(\theta_\ell<\infty,~\forall i\in\range{0}{\ell},~Z_{\theta_i+1}\neq0)&=\E[\mathds 1_{\{\theta_\ell<\infty,~\forall i\in\range{0}{\ell-1},~Z_{\theta_i+1}\neq 0\}}\P(Z_{\theta_l+1}\neq0~\vert~\mathcal G_{\theta_\ell})]\\
                &\leq\E[\mathds 1_{\{\theta_\ell<\infty,~\forall i\in\range{0}{\ell-1},~Z_{\theta_i+1}\neq0\}}q_K]\\
                &\leq q_K\P(\theta_{l-1}<\infty,~\forall i\in\range{0}{l-1},~Z_{\theta_i+1}\neq0),
            \end{split}
        \end{equation}
    where the first inequality comes from Lemma \ref{lemma_coalescence_proba} as $\|Z_{\theta_\ell}\|\leq K$ by construction, and the last line uses that $\theta_l<\infty$ implies $\theta_{l-1}<\infty$. Hence, by immediate induction, we deduce that for all $\ell\geq 0$
        \[\P(\theta_\ell<\infty,~\forall i\in\range{0}{\ell}~Z_{\theta_i+1}\neq0)\leq {q_K}^\ell.\]
    Now since $\liminf_{n\to\infty}\|Z_n\|<K$ implies that $\theta_l<\infty$ for all $\ell\geq 0$, we get that
        \[\P\left(\liminf_{n\to\infty}\|Z_n\|<K,~\forall n\geq 0,~Z_n\neq0\right)\leq \P(\theta_\ell<\infty,~\forall i\in\range{0}{\ell},~Z_{\theta_i+1}\neq0)\leq {q_K}^\ell,\]
    for each $\ell\geq 0$, and thus letting $l\to\infty$,
        \[\P\left(\liminf_{n\to\infty}\|Z_n\|<K,~\forall n\geq 0,~Z_n\neq0\right)=0.\]
    Finally, injecting in \eqref{eq_coal_liminfK} we deduce that $\P(\forall n\geq 0,~\|Z_n\|>0)=0$, i.e.\ the process will almost surely hit $0$. The proof is complete.
\end{proof}

We are now ready to prove the main result of the section.

\begin{proof}[Proof of Theorem \ref{thm_coalescence}]
    From Proposition \ref{prop_coalescence}, we have that almost surely, for all $\mathbf u_1,\mathbf u_2\in\Q^d$ with $\mathbf u_1\cdot e_d=\mathbf u_2\cdot e_d$, the two trajectories explored from $\mathbf u_1$ and $\mathbf u_2$ eventually coalesce. We now conclude using a density argument. Let $x,y\in\mathcal N$. Fix $h\in\Q$ be such that $h>\max\{x\cdot e_d,~y\cdot e_d\}$. Since $h>x\cdot e_d$ and $\Psi^{\circ n}(x)\cdot e_d\to\infty$, there exist $n_0\in\N$ such that
        \[\Psi^{n_0}(x)\cdot e_d\leq h\quad\text{and}\quad\Psi^{n_0+1}(x)\cdot e_d>h.\]
    Then, we can almost surely choose $\mathbf u_1\in\Q^d$ with $\mathbf u_1\cdot e_d=h$ close enough to the segment $[\Psi^{n_0}(x), \Psi^{n_0+1}(x)]$ so that $\Psi(\mathbf u_1)=\Psi^{n_0+1}(x)$. Similarly, we can choose $\mathbf u_2\in\Q^d$ with $\mathbf u_2\cdot e_d=h$ so that $\Psi(\mathbf u_2)=\Psi^{n_1}(y)$ for some $n_1\in\N$. Since the trajectories explored from $\mathbf u_1$, $\mathbf u_2$ coalesce, we have that the trajectories from $x$ and $y$ coalesce. The proof is complete.
\end{proof}

\section{Non-coalescence in high dimensions}

\label{section_non_coalescence}

In this section, we prove the non-coalescence part of our main result, stated precisely in the following theorem.

\begin{Theorem}
    \label{thm_non_coalescence}
    Assume $p\in\{1, 2, \infty\}$ and $d\geq 4$. Then the DSF consists almost surely of infinitely many disjoint trees.
\end{Theorem}

For the remainder of this section, we assume that $p\in\{1, 2, \infty\}$ and $d\geq 4$. The strategy is the following. Fix an arbitrary $k\geq 0$. Our goal is to find deterministic starting points $(\mathbf u_i)_{i\in\range{1}{k}}\in\R^k$ with $\mathbf u_i\cdot e_d=0$ for each $i\in\range{1}{k}$, such that the probability that the corresponding explored trajectories never coalesce is positive. That is
    \[\P(\forall n\geq0,~r_n>0)>0.\]
To this end, consider a sequence $((\Delta^{i, j}_n)_{1\leq i<j\leq k},D_{n+1})_{n\geq 0}$ of i.i.d.\ copies of $((\Delta^{i, j})_{1\leq i<j\leq k},W)$, independent of the Poisson point process $\mathcal N$. For each $n\geq 0$ and pair $1\leq i<j\leq k$, define
    \[S_n^{i, j}\coloneqq\sum_{l=1}^n \Delta_l^{i, j}\quad\text{and we set}\quad\tilde r_n\coloneqq\frac{1}{2}\max_{1\leq i<j\leq k}\|Z_0^{i, j}+S_n^{i, j}\|.\]
The processes $((Z_n^{i, j})_{1\leq i<j\leq k})_{n\geq 0}$ and $((Z_0^{i, j}+S_n^{i, j})_{1\leq i<j\leq k})_{n\geq 0}$ are linked via Theorem~\ref{thm_renewal_decomposition}. The following lemma leverages this connection to prove a key intermediate result used in the next step of the argument.

\begin{Lemma}
    \label{lemma_indep}
    For all $n\geq 0$, let us define
        \begin{itemize}
            \item $G_n\coloneqq\{\forall k\leq n-1,~r_n>W_{n+1}\}\cap\{\forall k\leq n,~r_n\geq\kappa+R\}$,
            \item $\tilde G_n\coloneqq\{\forall k\leq n-1,~\tilde r_n>D_{n+1}\}\cap\{\forall k\leq n,~\tilde r_n\geq\kappa+R\}$.
        \end{itemize}
    Then,
        \[(\mathds 1_{G_n}~,(Z_n^{i, j}\mathds 1_{G_n})_{1\leq i<j\leq k})\quad\text{and}\quad(\mathds 1_{\tilde G_n},~((Z_0^{i, j}+S_n^{i, j})\mathds 1_{\tilde G_n})_{1\leq i<j\leq k})\]
    are identically distributed. In particular, $\P(G_n)=\P(\tilde G_n)$.
\end{Lemma}

\begin{proof}
    Obtained by induction over $n\geq 0$. The initialization is clear since $r_0=\tilde r_0$ and the heredity is ensured by Theorem \ref{thm_renewal_decomposition}.
\end{proof}

Using Lemma \ref{lemma_indep}, we can write that
    \begin{equation}
        \label{eq_probadrift}
        \begin{split}
            \P(\forall n\geq0,~r_n>0)&\geq\P\left(\bigcap_{n\geq 0}G_n\right)
            =\lim_{n\to\infty}\P(G_n)\\
            &=\lim_{n\to\infty}\P(\tilde G_n)
            =\P\left(\bigcap_{n\geq 0}\tilde G_n\right)
            \geq\P(\forall n\geq 0,~\tilde r_n\geq D_{n+1}+\kappa+R).
        \end{split}
    \end{equation}
Then, observing that on the renewal event $A_0$, we have $\beta_0=\tau_0=0$ and thus $Z_0^{i, j}=\mathbf p(\mathbf u_i-\mathbf u_j)$ for all $1\leq i<j\leq k$, we derive from \eqref{eq_probadrift} that
    \begin{equation*}
        \begin{split}
            &\P(\forall n\geq0,~r_n>0)\\
            &\quad\geq\P(A_0,~\forall n\geq 0,~\tilde r_n\geq D_{n+1}+\kappa+R)\\
            &\quad=\P(A_0,~\forall n\geq0,~\forall 1\leq i<j\leq k,~\|\mathbf p(\mathbf u_i-\mathbf u_j)+S_n^{i, j}\|\geq2(D_{n+1}+\kappa+R)).
        \end{split}
    \end{equation*}
It therefore suffices to verify that we can choose $(\mathbf u_i)_{i\in\range{1}{k}}$ so that
    \begin{equation}
        \label{eq_non-coal_proba}
        \P(A_0,~\forall n\geq0,~\forall 1\leq i<j\leq k,~\|\mathbf p(\mathbf u_i-\mathbf u_j)+S_n^{i, j}\|_p\geq2(D_{n+1}+\kappa+R))>0.
    \end{equation}
The advantage of this reformulation is that we are now working with a system built on i.i.d.\ random variables. Before proceeding with the proof of Theorem~\ref{thm_non_coalescence}, we establish the following lemma, which is based on a general result of Kesten~\cite{KestenErickson} concerning the rate of escape to infinity for transient random walks in dimension larger than three. Note that this is where we use that $d-1\geq 3$.

\begin{Lemma}
    \label{lemma_kesten}
    For all $A>0$ and $\epsilon>0$, there exists $(s_i)_{i\in\range{1}{k}}\in(\R^{d-1})^k$ such that
        \[\P(\forall n\geq0,~\forall 1\leq i<j\leq k,~\|s_i-s_j+S_n^{i, j}\|\geq A+n^{1/3})\geq 1-\epsilon.\]

    %Let $f:\R_+\to\R_+^*$ be an increasing function such that $f(t)t^{-1/2}$ decreases to $0$ as $t\to\infty$ and such that
    %    \[\int_1^\infty f(t)^{d-3}t^{-\frac{d-1}{2}}\mathrm dt<\infty.\]
    %Then for all $\epsilon>0$, there exists $(s_i)_{i\in\range{1}{k}}\in(\R^{d-1})^k$ such that
    %    \[\P[\forall n\geq0~\forall 1\leq i<j\leq k~\|s_i-s_j+S_n^{i, j}\|\geq f(n)]\geq 1-\epsilon.\]
\end{Lemma}

\begin{proof}
    For each $1\leq i<j\leq k$, $(S_n^{i,j})_{n\geq 0}$ is a random walk with i.i.d.\ increments distributed as $\Delta$, whose support spans $\R^{d-1}$ as it is exchangeable and non-trivial. Let $f:t\mapsto A+t^{1/3}$. The function $f$ is increasing and the function $t\mapsto f(t)t^{-1/2}$ decreases to $0$ as $t\to\infty$. Moreover as $t\to\infty$, since $d\geq 4$, we have
        \[f(t)^{d-3}t^{-\frac{d-1}{2}}\sim t^{\frac{d-3}{3}-\frac{d-1}{2}}=t^{-\frac{1}{2}-\frac{d}{6}}=o\left(\frac{1}{t^{1+1/6}}\right),\]
    implying that
        \[\int_{[1,\infty]}f(t)^{(d-1)-2}t^{-\frac{d-1}{2}}\mathrm dt<\infty.\]
    These are precisely the conditions required to apply the main theorem of \cite{KestenErickson}, which implies that, almost surely, for each $1\leq i<j\leq k$,
        \[\frac{\|S_n^{i,j}\|}{f(n)}\to\infty.\]
    Therefore,
        \[\P(\forall n\geq n_0~\forall 1\leq i<j\leq k,~\|S_n^{i,j}\|\geq f(n))\to 1\quad\text{as}\quad n_0\to\infty.\]
    Now, fixing $n_0\geq 0$ big enough and using that the process $((S_n^{i,j})_{1\leq i<j\leq k})_{n\geq0}$ is a random walk with i.i.d.\ increments in $(\R^{d-1})^{k(k+1)/2}$, we can write
        \begin{equation}
            \label{eq_markov_n0}
            1-\epsilon\leq\P(\forall n\geq n_0~\forall 1\leq i<j\leq k~\|S_n^{i,j}\|\geq f(n))
            =\E[P((S_{n_0}^{i,j})_{1\leq i<j\leq k})],
        \end{equation}
    where
        \[P:(\R^{d-1})^{k(k+1)/2}\to[0, 1],~(s_{i,j})_{1\leq i<j\leq k}\mapsto\P(\forall n\geq0~\forall i\in\range{1}{k}~\|s_{i,j}+S_n^{i,j}\|\geq f(n_0+n)).\]
    Since for each $1\leq i<j<\ell\leq k$ we have $S_{n_0}^{i,j}+S_{n_0}^{j,\ell}=S_{n_0}^{i,\ell}$, then \eqref{eq_markov_n0} implies that there exists $(s_{i,j})_{1\leq i<j\leq k}\in(\R^{d-1})^{k(k+1)/2}$ such that
        \[1-\epsilon\leq P((s_{i,j})_{1\leq i<j\leq k})\]
    with $s_{i,j}+s_{j, \ell}=s_{i, \ell}$ for all $1\leq i<j<\ell\leq k$. Then, setting $s_1\coloneqq 0$ and $s_i\coloneqq s_{1, i}$ for each $1< i\leq k$, we have $s_{i,j}=s_i-s_j$ for all $1\leq i<j\leq k$ and finally,
        \[\P(\forall n\geq0~\forall 1\leq i<j\leq k~\|s_i-s_j+S_n^{i, j}\|\geq f(n+n_0)\geq f(n)=A+n^{1/3})\geq 1-\epsilon.\]
    The proof is complete.
\end{proof}

We now have all the ingredients to prove the main result of the section.

\begin{proof}[Proof of Theorem \ref{thm_non_coalescence}]
    For any function $f:\R_+\to\R_+^*$, using a union bound, we can write,
        \begin{equation}
            \label{eq_proba_non_coal}
            \begin{split}
                &\P(A_0,~\forall n\geq 0~\forall 1\leq i<j\leq k,~\|p(\mathbf u_i-\mathbf u_j)+S_n^{i,j}\|\geq 2(D_{n+1}+\kappa+R))\\
                &\quad\geq\P(A_0,~\forall n\geq 0~\forall 1\leq i<j\leq k,~\|p(\mathbf u_i-\mathbf u_j)+S_n^{i,j}\|\geq f(n)\geq 2(D_{n+1}+\kappa+R))\\
                &\quad\geq\P(\forall n\geq 0~\forall 1\leq i<j\leq k,~\|p(\mathbf u_i-\mathbf u_j)+S_n^{i,j}\|\leq f(n))\\
                &\qquad\quad\quad+\P(A_0)-1-\P(\exists n\geq 0,~f(n)<2(D_{n+1}+\kappa+R)).
            \end{split}
        \end{equation}
    By Lemma \ref{lemma_renewalproba}, we have
        \[\P(A_0)=\P(A_0~|~\mathcal F_0)\geq p_0>0,\]
    where $p_0$ does not depend on the configuration $(\mathbf u_i)_{i\in\range{1}{k}}$. By an union bound and using proposition \ref{prop_block_size}, we have
        \[\P(\exists n\geq 0,~f(n)<2(W_{n+1}+\kappa+R))\leq\sum_{n\geq 0}\alpha e^{-\beta \sqrt{\left(\frac{f(n)}{2}-\kappa-R\right)_+}}.\]
    We can thus choose $f:t\mapsto A+t^{1/3}$ with $A>0$ big enough so that by the previous display,
        \begin{equation}
            \label{eq_pwn}
            \P(\exists n\geq 0,~f(n)<2(W_{n+1}+\kappa+R))\leq\frac{p_0}{3}.
        \end{equation}
    Applying Lemma \ref{lemma_kesten}, we get $(s_i)_{i\in\range{1}{k}}\in(\R^{d-1})^k$ such that
        \begin{equation}
            \label{eq_proba_escape}
            \P(\forall n\geq0,~\forall 1\leq i<j\leq k,~\|s_i-s_j+S_n^{i, j}\|\geq f(n))\geq 1-\frac{p_0}{3}.
        \end{equation}
    Then, taking $(\mathbf u_i)_{i\in\range{1}{k}}$ such that $\mathbf p(\mathbf u_i)=s_i$ and $\mathbf u_i\cdot e_d=0$ for each $i\in\range{1}{k}$, injecting \eqref{eq_proba_escape} and \eqref{eq_pwn} in \eqref{eq_proba_non_coal}, we get
        \[\P(A_0,~\forall n\geq 0~\forall 1\leq i<j\leq k,~\|p(\mathbf u_i-\mathbf u_j)+S_n^{i,j}\|>2(W_{n+1}+\kappa))\geq\frac{p_0}{3}>0,\]
    which implies that $\P(\forall n\geq 0,~r_n>0)>0$. Therefore, the probability that the DSF contains more than $k$ non-coalescing trajectories is positive. Since this event is asymptotic and the asymptotic sigma field of a Poisson point process is trivial, we deduce that almost surely, we can find $k$ non-coalescing trajectories. Since $k$ was arbitrary, we have that the DSF consists almost surely of infinitely many disjoint trees. The proof is complete.
\end{proof}

\section{Convergence toward the Brownian web}

\label{section_BW}

In this section, we complete the proof of our main result by precisely formulating the convergence of the DSF under diffusive scaling toward the Brownian web in dimension $d=2$, and explaining how this follows from the results of \cite{Dsf2d}. Unlike the previous section, this part involves a direct adaptation of their argument, and we therefore limit ourselves to recalling the main steps.\\

Introduced by Arratia \cite{arratia1979coalescing}, the Brownian web is an object that appears as a scale limit of a certain class of random geometric graphs (e.g.\ \cite{newmanconvergence}). We recall the framework of \cite{BrownianWeb} in minimal detail to define this object and state properly our result. We call a \emph{path} any continuous application $\pi:[\sigma_\pi,\infty)\to\R$ with $\sigma_\pi\in[-\infty,\infty]$ and we denote by $\Pi$ the collection of such paths. For all $\pi,\pi'\in\Pi$, we define
    \[d_\Pi(\pi,\pi')\coloneqq|\tanh\sigma_\pi-\tanh\sigma_{\pi'}|\vee \sup_{t>\sigma_\pi\wedge\sigma_{\pi'}}\left|\frac{\tanh\pi(t\vee\sigma_\pi)-\tanh\pi'(t\vee\sigma_{\pi'})}{1+|t|}\right|,\]
with the conventions $\sup\emptyset=-\infty$ and $\tanh(\infty)=-\tanh(-\infty)=1$. Then, the space $(\Pi,d_\Pi)$ is complete and separable. Let $\mathcal H$ be the collection of all non-empty compact subset of $\Pi$. We endow $\mathcal H$ with the Hausdorff distance $d_\mathcal H$, defined for all $K,K'\in\mathcal H$ by
    \[d_{\mathcal H}(K, K')\coloneqq\sup_{x\in K}d_\Pi(x, K')\vee\sup_{x'\in K'}d_\Pi(x', K).\]
Since $(\Pi, d_\Pi)$ is a complete and separable, then so does $(\mathcal H,d_\mathcal H)$. The next theorem of \cite{BrownianWeb} defines the Brownian web as a random variable taking value in the Polish space $(\mathcal H,d_\mathcal H)$.

\begin{Theorem}[Theorem 2.1 of \cite{BrownianWeb}]
    There exists an $(\mathcal H,d_\mathcal H)$-valued random variable $\mathcal W$ whose distribution is uniquely determined by the following properties:
        \begin{itemize}
            \item For any deterministic $(x,t)\in\R^2$, there is almost surely a unique path $\pi^{x,t}$ in $\mathcal W$ starting at the space time point $(x,t)$, i.e.\ such that $\sigma_{\pi^{x,t}}=t$ and $\pi^{x,t}(t)=x$.
            \item For any deterministic $n\geq1$ and finite sequence $[(x_i, t_i)]_{1\leq i\leq n}\in\mathcal(\R^2)^n$, $(\pi^{x_i, t_i})_{1\leq i\leq n}$ is distributed as coalescing Brownian motion started from space time points $[(x_i, t_i)]_{1\leq i\leq n}$.
            \item For any deterministic dense subset $\mathcal D\subset\R^2$, $\mathcal W$ is almost surely the closure of 
                \[\{\pi^{x, t}:(x, t)\in\mathcal D\}\]
            in $(\Pi,d_\Pi)$.
        \end{itemize}
    This variable $\mathcal W$ is called the Brownian web.
\end{Theorem}

Now that the Brownian web is defined as a random variable taking value in $\mathcal H$, we define scaling of the DSF in $\mathcal H$ in order to set up the convergence. For any $u\in\mathcal N$, we define $\pi^u:[u\cdot e_2,\infty)\to\R$ as the path that satisfies
    \[\pi^u(\Psi^{ i}(u)\cdot e_2)=\Psi^{ i}(u)\cdot e_1\]
for all $i\geq 0$ and that is linear on every segment $[\Psi^{ i}(u)\cdot e_2, \Psi^{ {i+1}}(u)\cdot e_2]$ for all $i\geq0$. For all $u\in\mathcal N$, $n\geq 0$, $\gamma>0$ and $\sigma>0$, we define
    \[\pi^u_{n,\gamma,\sigma}:\left[\frac{u\cdot e_2}{\gamma n^2},\infty\right)\to\R^d,~t\mapsto \frac{\pi^u(\gamma n^2t)}{\sigma n},\]
which is also a path. For each $n\geq 0$, $\gamma>0$ and $\sigma>0$ we set
    \[\mathcal X_n(\gamma,\sigma)\coloneqq\{\pi^u_{n,\gamma,\sigma}:u\in\mathcal N\}.\]
We can now state formally the convergence result of our main theorem.

\begin{Theorem}
    \label{thm_convergence_BW}
    If $d=2$ and $p\in[1, \infty]$, there exists $\gamma=\gamma(p)$ and $\sigma=\sigma(p)$ such that the sequence $[\mathcal X_n(\gamma, \sigma)]_{n\geq 0}$ of $(\mathcal H, d_{\mathcal H})$-valued random variables converge in distribution toward the Brownian web $\mathcal W$.
\end{Theorem}

To prove Theorem \ref{thm_convergence_BW}, we apply \cite{Dsf2d} arguments without modifications. First, we consider $k=2$ in the exploration process with arbitrary $\mathbf u_1,\mathbf u_2$ with $\mathbf u_1\cdot e_2=\mathbf u_2\cdot e_2=0$ and $\mathbf u_1\cdot e_1>\mathbf u_2\cdot e_1$. We denote
    \[(Z_n)_{n\geq 0}\coloneqq (Z^{1, 2}_n)_{n\geq 0}.\]
By planarity, it is a non-negative process. In order to apply \cite[Corollary 5.6]{Dsf2d} to this process, we need to verify some assumptions, that we gather in the following proposition.

\begin{Proposition}
    \label{prop_assumption}
    There exists constants $M_0, C_0, C_1, C_2, C_3>0$ such that for all $n\geq 0$
    \begin{itemize}
        \item[(i)] For all $n\geq 0$, there exists an event $F_n$, such that, on the event $\{Z_n>M_0\}$,
            \[\P(F_n~|~\mathcal G_n)\geq 1-\frac{C_0}{(Z_n)^3}\quad\text{and}\quad\E[(Z_{n+1}-Z_n)\mathds 1_{F_n}~|~\mathcal G_n]=0.\]
        \item [(ii)] For all $n\geq 0$, on the event $\{Z_n\leq M_0\}$,
            \[\E[Z_{n+1}-Z_n~|~\mathcal G_n]\leq C_1.\]
        \item [(iii)] For all $m>0$, there exists $c_m>0$ such that for all $n\geq 0$, on the event $\{Z_n<m\}$,
            \[\P(Z_{n+1}=0~|~\mathcal G_n)\geq c_m.\]
        \item[(iv)] For all $n\geq 0$, on the event $\{Z_n>M_0\}$,
            \[\E[(Z_{n+1}-Z_n)^2~|~\mathcal G_n]\geq C_2\quad\text{and}\quad\E[|Z_{n+1}-Z_n|^3~|~\mathcal G_n]\leq C_3\]
    \end{itemize}
\end{Proposition}

\begin{proof}
    We start with $(iv)$. For all $n\geq0$ we set
        \[F_n\coloneqq\{2(W_{n+1}\vee(\kappa+R))<Z_n\}.\]
    If $M>0$ and $n\geq 0$, using Theorem \ref{thm_renewal_decomposition}, on the event $\{Z_n>M\}$, we have
        \begin{equation*}
            \begin{split}
                \E[(Z_{n+1}-Z_n)^2~|~\mathcal G_n]&\geq\E[(Z_{n+1}-Z_n)^2\mathds 1_{F_n}~|~\mathcal G_n]\\
                &=\E[(\Delta^{1, 2})^2\mathds 1_{2(W\vee(\kappa+R))<Z_n}~|~\mathcal G_n]\\
                &\geq\E[(\Delta^{1, 2})^2\mathds 1_{2(W\vee(\kappa+R))<M}],
            \end{split}
        \end{equation*}
    hence we can choose $M_0>0$ big enough so that on the event $\{Z_n>M_0\}$,
        \[\E[(Z_{n+1}-Z_n)^2~|~\mathcal G_n]\geq\E[(\Delta^{1, 2})^2\mathds 1_{2(W\vee(\kappa+R))<M_0}]\eqqcolon C_2>0.\]
    Furthermore, from Proposition \ref{prop_block_size} there exists $C_3>0$ such that for all $n\geq 0$,
        \[\E[|Z_{n+1}-Z_n|^3~|~\mathcal G_n]\leq\E[(W_{n+1})^3~|~\mathcal G_n]\leq C_3<\infty.\]
    We now turn to $(i)$. Thanks to Theorem \ref{thm_renewal_decomposition}, we have that
        \[\E[(Z_{n+1}-Z_n)\mathds 1_{F_n}~|~\mathcal G_n]=\E[\Delta^{1, 2}\mathds 1_{W\vee(\kappa+R)<r_n}~|~\mathcal G_n]=0,\]
    where we used that $\Delta^{1, 2}\mathds 1_{W\vee(\kappa+R)<r_n}$ is symmetric from Lemma \ref{lemma_Delta}. From Proposition \ref{prop_block_size},
        \[\P(F_n~|~\mathcal G_n)\geq 1-\P(2W_{n+1}>Z_n~|~\mathcal G_n)\geq 1-\alpha e^{-\beta\sqrt{\frac{Z_n}{2}}}.\]
    We can thus fix $C_0>0$ such that $(i)$ is satisfied. For $(ii)$, observe that using Proposition \ref{prop_block_size}, there exists $C_1>0$ such that for all $n\geq 0$,
        \[\E[Z_{n+1}-Z_n~|~\mathcal G_n]\leq\E[|Z_{n+1}-Z_n|~|~\mathcal G_n]\leq \E[W_{n+1}~|~\mathcal G_n]\leq C_1<\infty.\]
    Finally, point (iii) is given by Lemma \ref{lemma_coalescence_proba}. The proof is complete.
\end{proof}

Proposition \ref{prop_assumption} ensure that the assumptions of \cite[Corollary 5.6]{Dsf2d} are satisfied for the process $(Z_n)_{n\geq 0}$, and thus, denoting
    \[\nu\coloneqq\inf\{n\geq 0:Z_n=0\},\]
we get that there exists a constant $ K>0$ independent of $\mathbf u_1,\mathbf u_1$, such that for all $n\geq 0$,
    \[\P(\nu\geq n)\leq\frac{K}{\sqrt n}\max\{1, (\mathbf u_1-\mathbf u_2)\cdot e_1\}.\]
Then, following the proof of \cite[Theorem 5.1]{Dsf2d}, we get the following theorem, which gives a key asymptotic estimate on the coalescing time of two trajectories.

\begin{Theorem}
    For any $\mathbf u_1,\mathbf u_2\in\R^d$ such that $\mathbf u_1\cdot e_2=\mathbf u_2\cdot e_2=0$, we define
        \[T(\mathbf u_1,\mathbf u_2)\coloneqq\inf\{t\geq 0:\pi^{\mathbf u_1}(t)=\pi^{\mathbf u_2}(t)\}.\]
    Then, there exists a constant $K'>0$ such that for all $\mathbf u_1,\mathbf u_2\in\mathbb H(0)$ and $t>0$,
        \[\P(T(\mathbf u_1,\mathbf u_2)>t)\leq\frac{K'}{\sqrt t}\max\{1, |(\mathbf u_1-\mathbf u_2)\cdot e_1|\}.\]
\end{Theorem}

Then, \cite[Section 6]{Dsf2d} gives all the details to prove Theorem \ref{thm_convergence_BW}. No modifications are needed, as it relies only on the coalescing time estimate of two trajectories given by the previous theorem, the renewal decomposition of the exploration process, the block size tail decay estimate and planarity. More precisely, \cite[Theorem 6.3]{Dsf2d} provides a set of conditions whose verification ensures convergence to the Brownian web under diffusive scaling. We refer the reader to \cite{Dsf2d} for further details.

\appendix

\section{Obstructions to generalization of Theorem \ref{thm_sto_dom_poisson}}

\label{section_counter_example}

The validity of the stochastic domination given by Theorem \ref{thm_sto_dom_poisson} for any $p\in[1, \infty]$ when $d=2$, and for any $d\geq 2$ when $p\in\{1, 2, \infty\}$, initially seem somewhat arbitrary. This naturally raises the question of whether the result remains valid beyond these specific assumptions.

A key component of the proof is Lemma \ref{lemma_section}, which plays a central role in establishing Proposition \ref{prop_sto_dom}, the core intermediate stochastic domination result. However, Lemma \ref{lemma_section} does not hold in full generality outside the specified parameter ranges. We illustrate this with a counterexample.

Consider the case $p=d=3$, and let $H$ the half-ball $B^+(c, \|c\|)$ centered at $c=(3, 3, 0)$. Since $0\in\partial B(c,\|c\|_3)$, $H$ is a valid re-centered history set. Recall that for any $h\in[0, 1]$ we have $\rho(h)=(1-h^3)^{\frac{1}{3}}$. Define the point
    \[x\coloneqq\left(\frac{3}{4},~ -\frac{1}{2},~ 0\right).\]
We claim that $x\in\Phi(S_{\rho(2/3)}\cap H)$ but $x\notin\Phi(S_0\cap H)$, violating the inclusion $\Phi(S_{\rho(2/3)}\cap H)\subset \Phi(S_0\cap H)$ even though $\rho(2/3)>0$. Indeed, using the fact that $\rho\circ\rho=\mathrm{Id}$ and $e_d\cdot x=e_d\cdot c=0$, we compute
\begin{equation*}
    \begin{split}
        \left\|\rho\circ\rho\left(\frac{2}{3}\right)x+\rho\left(\frac{2}{3}\right)e_d-c\right\|^3&=\left\|\frac{2}{3}x-c\right\|^3+\rho\left(\frac{2}{3}\right)^3\\&=\left(3-\frac{2}{3}\cdot\frac{3}{4}\right)^3+\left(\frac{2}{3}\cdot\frac{1}{2}+3\right)^3+1-\left(\frac{2}{3}\right)^3\\
    &=\frac{11527}{216}<2\cdot3^3=\|c\|^3,
    \end{split}
\end{equation*}
which shows $x\in\Phi(S_{\rho(2/3)}\cap H)$. On the other hand,
\begin{align*}
    \|x - c||^3 &= \left(3 - \frac{3}{4}\right)^3 + \left(3 + \frac{1}{2}\right)^3 = \frac{3473}{64} > 2 \cdot 3^3 = \|c\|^3,
\end{align*}
hence $x\notin\Phi(S_0\cap H)$. This counterexample demonstrates that Lemma \ref{lemma_section} does not extend in general when both $d\neq 2$ and $p\notin\{1, 2,\infty\}$. Interestingly, we have not identified any counterexample to the generalization in the case $p\in(1, 2)$, which may suggest that, while the stated assumptions cannot be entirely removed, they might be further sharpened.

Moreover, in the setting $d=3$, $p=4$ and  $H=B^+(c, \|c\|)$ with $c=(3, 3, 0)$,  numerical evidence indicates that the function $h\mapsto\alpha_h$ is strictly decreasing. This would imply that $U^H \cdot e_d$ and $U^\emptyset \cdot e_d$ are not identically distributed, yet the reverse stochastic domination to that of Proposition \ref{prop_sto_dom}, namely $U^H \cdot e_d \preceq_{\mathrm{sto}}U^\emptyset \cdot e_d$, holds. This is somewhat surprising and points to a genuine delicacy in the structure of the general case. It strongly suggests that Proposition \ref{prop_sto_dom} may not hold once all assumptions on $p$ and $d$ are removed. Whether Theorem \ref{thm_sto_dom_poisson} itself remains valid in full generality is an open question. At present, there is no compelling evidence either confirming or ruling out such an extension. Whether Theorem \ref{thm_sto_dom_poisson} itself extends remains an open and subtle question. There is, at this stage, no clear reason why it should or should not.

\section{Proof of Proposition \ref{prop_combinatorial}}

\label{section_balls}

This section aims to prove Proposition \ref{prop_combinatorial}. From now on, we fix $\kappa>0$, $k\geq 1$, $d\geq 1$ and $p\in[1,\infty]$. We want to show that there exists two constants $a_0, R_0>0$ such that for all configuration $(c_i)_{i\in\range{1}{k}}\in\R^k$ with $c_i\cdot e_d=0$ for each $i\in\range{1}{k}$, there exist a partition $\Pi$ of $\range{1}{k}$ such that for all $\pi\in\Pi$,
    \[|P_\pi^{R_0}|\geq a_0\quad\text{where}\quad P_\pi^{R_0}\coloneqq\bigcap_{i\in\pi} B^+(c_i, R_0)\setminus\bigcup_{j\in\range{1}{k}\setminus\pi} B(c_j-\kappa e_d, \kappa+R_0).\]
As an intermediate problem, we will show the following proposition. The cases $p<\infty$ and $p=\infty$ will be proved separately in two distinct subsections. We will then derive Proposition \ref{prop_combinatorial} from it.

\begin{Proposition}
    \label{prop_ball}
    There exist $\eta=\eta(d, p)>0$ such that for all $(c_i)_{i\in\range{1}{k}}\in\R^k$ with $c_i\cdot e_d=0$ for each $i\in\range{1}{k}$, there exist a partition $\Pi$ of $\range{1}{k}$ and $(\alpha_\pi)_{\pi\in\Pi}\in(\R^d)^\Pi$ such that for all $\pi\in\Pi$,
        \[B(\alpha_\pi,\eta)\subset P_\pi\quad\text{where}\quad P_\pi\coloneqq\bigcap_{i\in\pi}B^+(c_i, 1)\setminus\bigcup_{j\in\range{1}{k}\setminus\pi}B(c_j, 1).\]
    In particular, $|P_\pi|\geq|B(0,\eta)|$ for all $\pi\in\Pi$.
\end{Proposition}

\subsection{The case $p<\infty$}

In this section, we prove Proposition \ref{prop_ball} when $p<\infty$.

We start with some notation. Let $(c_i)_{1\leq i\leq k}\in\R^k$ with $c_i\cdot e_d=0$ for each $i\in\range{1}{k}$. For every non-empty subset $I$ of $\range{1}{k}$, we define the barycenter of $(c_i)_{i\in I}$ by
    \[c_I\coloneqq\frac{1}{|I|}\sum_{i\in I}c_i,\]
and the radius of $(c_i)_{i\in I}$ with respect to $c_I$ by
    \[r_I\coloneqq\max_{i\in I}\|c_i-c_I\|.\]
As a first step toward Proposition \ref{prop_ball}, we want to group the centers into clusters, represented by a partition $\Pi$ of $\range{1}{k}$ such that every cluster is small in radius and all the clusters are far from each other. We start with the following lemma.

\begin{Lemma}
    \label{lemma_clustering}
    There exist $C_k\geq 1$ such that for all $(c_i)_{1\leq i\leq k}\in\R^d$, $i_0\in\range{1}{k}$ and $\xi>0$ there exist $\pi\subset\range{1}{k}$ such that $i_0\in\pi$ and
        \begin{itemize}
            \item $r_\pi\leq C_k\xi$,
            \item for all $j\in\range{1}{k}\setminus\pi$, $\|c_j-c_\pi\|>r_\pi+\xi$.
        \end{itemize}
\end{Lemma}

%The rough idea is the following: start with $\pi=\{i_0\}$, which has radius $0$. If there is a point that contradicts the second point of the statement, we add it to $\pi$ and the radius can not grow too much as such a point should be close to the barycenter. Doing so until there is no more of such points, we obtain a cluster of point whose radius is controlled by $\xi$ and that satisfies the second point of the statement by construction.

\begin{proof}
     We define a non-decreasing sequence of sets $(\pi_{n})_{n\geq 1}\in\mathcal P(\range{1}{k})^{\N^*}$ with
        \begin{itemize}
            \item $\pi_0\coloneqq\{i_0\}$,
            \item for $n\geq1$, let $I_n\coloneqq\{i\in \range{1}{k}\setminus\pi_n~:~\|c_i-c_{\pi_n}\|\leq r_{\pi_n}+\xi\}$ and set
                $$\pi_{n+1}\coloneqq
                    \begin{cases}
                        \pi_n  & \text{if } I_n=\emptyset \\
                        \pi_n\cup\{\min I_n\} & \text{otherwise}
                    \end{cases}.$$
        \end{itemize}
    By construction, $(\pi_n)_{n\geq 1}$ increases and then stays stationary. Since the size of the sets is bounded by $k$, we have that $\pi_k=\pi_{k+1}$, that is $I_k=\emptyset$. We set $\pi\coloneqq \pi_n$. By definition of $I_k=\emptyset$ we have that for all $j\in\range{1}{k}\setminus\pi_k$,
        $$\|c_j-c_\pi\|>r_\pi+\epsilon.$$
    Now it remains to check that we can control $r_\pi$. 
    
    Let $n\geq 1$. If $I_n=\emptyset$, then $r_{\pi_n}=r_{\pi_{n+1}}$. Otherwise we can set $j_n\coloneqq\min I_n$ and write
        $$c_{\pi_{n+1}}=\frac{1}{n+1}\left(c_{j_n}+\sum_{i\in\pi_n}c_i\right)=\frac{c_{j_n}}{n+1}+\frac{n}{n+1}c_{\pi_n},$$
    which implies that
    \begin{equation}
        \label{eq_barycenter_delta}
        \|c_{\pi_{n+1}}-c_{\pi_n}\|=\frac{1}{n+1}\|c_{j_n}-c_{\pi_n}\|\leq \frac{1}{n+1}(r_{\pi_n}+\xi).
    \end{equation}
    We deduce then that when $I_n\neq\emptyset$, we have
    \begin{equation}
        \label{eq_radius_growth}
        \begin{split}
            r_{\pi_{n+1}}&=\max_{j\in\pi_n\cup\{j_n\}}\|c_j-c_{\pi_{n+1}}\|\\
            &\leq\|c_{\pi_n}-c_{\pi_{n+1}}\|+\max\left(\|c_{j_n}-c_\pi\|,r_{\pi_n}\right)\\
            &\leq\frac{1}{n+1}(r_{\pi_n}+\xi)+\max(r_{\pi_n}+\xi,r_{\pi_n})\\
            &=\left(1+\frac{1}{n+1}\right)(r_{\pi_n}+\xi),
        \end{split}
    \end{equation}
    where the first inequality comes from the triangle inequality and definition of $r_{\pi_n}$ and the second one from (\ref{eq_barycenter_delta}) and definition of $j_n\in I_n$. Now if by induction we set $C_1\coloneqq 1$ and for all $n\geq1$ we define 
        $$C_{n+1}\coloneqq\left(1+\frac{1}{n+1}\right)(C_n+1),$$
    then since $r_{\pi_0}=0\leq C_0\xi$, the inequality (\ref{eq_radius_growth}) allows us to obtain by induction that for all $n\geq 1$, $r_{\pi_n}\leq C_n\xi$. In particular
        $$r_{\pi_k}\leq C_k\xi,$$
    the proof is complete.
\end{proof}

Next, applying Lemma \ref{lemma_clustering} inductively, we can show the following clustering property.

\begin{Lemma}
    \label{lemma_eps_delta_grouping}
    Let $\delta>0$. There exists $\epsilon>0$ such that for any $(c_i)_{1\leq i\leq k}\in\R^d$ with $c_i\cdot e_d=0$ for each $i\in\range{1}{k}$, we can construct a partition $\Pi$ of $\range{1}{k}$ such that for all $\pi\in\Pi$
        \begin{itemize}
            \item $r_\pi\leq\delta$,
            \item and for all $j\in\range{1}{k}\setminus\pi$ we have $\|c_j-c_\pi\|>r_\pi+\epsilon$.
        \end{itemize}
\end{Lemma}

%The idea is to use Lemma \ref{lemma_clustering} multiple times with different well chosen $\xi$ and $i_0$ to get clusters that will form the elements of the partition.

\begin{proof}
    Let us construct by induction some sequences $(X_n)_{n\geq 0}$, $(\Pi_n)_{n\geq 0}$ and $(\xi_n)_{n\geq 0}$ such that $(X_n)_{n\geq 0}$ is a sequence of subset of $\range{1}{k}$ increasing to $\range{1}{k}$ and for all $n\geq0$, $\xi_n\leq\delta/C_k$ and $\Pi_n$ is a partition of $X_n$ satisfying for all $\pi\in\Pi_n$
        \begin{itemize}
            \item[(A)] $r_\pi\leq\delta$,
            \item[(B)] for all $j\in\range{1}{k}\setminus\pi$ we have $\|c_j-c_\pi\|_p>r_\pi+\xi_n$.
        \end{itemize}
    We start by setting $\Pi_0\coloneqq\emptyset$ and $\xi_0\coloneqq\delta/C_k$. Since $\Pi_0$ is empty, the two points are verified, and the initialization is verified. 
    
    Now let $n\geq0$ and assume we have constructed $X_n$, $\Pi_n$ and $\xi_n$. If $X_n=\range{1}{k}$, we set $\Pi_{n+1}=\Pi_n$, $\xi_{n+1}=\xi_n$ and the heredity holds. Otherwise, we need to enlarge $X_n$. To do so we fix $i_n\coloneqq\min X_n\setminus\range{1}{k}$ and we set
        \[\xi_{n+1}\coloneqq\frac{\xi_n}{2C_k}.\]
    We have $\xi_{n+1}\leq\xi_n\leq\delta/C_k$. Applying Lemma \ref{lemma_clustering} to get a subset $\pi_n\subset\range{1}{k}$ such that $i_n\in\pi_n$, $r_\pi\leq C_k\xi_{n+1}\leq\delta$ and for all $j\in\range{1}{k}\setminus\pi_n$, $\|c_j-c_{\pi_n}\|>r_\pi+\xi_{n+1}$. We set $\Pi_{n+1}\coloneqq\Pi_n\cup\{\pi_n\}$ and $X_{n+1}\coloneqq X_n\cup\pi_n$. Since $\xi_{n+1}\leq\xi_n$, then (A) and (B) are verified. 
    
    Now it remains to check that $\Pi_{n+1}$ is a partition, i.e.\ that $\pi_n\cap X_n=\emptyset$. Let $j\in X_n$. By definition, there exists $\pi\in\Pi_n$ such that $j\in\pi$. Since $i_n\notin\pi$, we can write
        \begin{equation*}
            \begin{split}
                r_{\pi'}+\xi_n&<\|c_{i_n}-c_{\pi'}\|\\
                &\leq\|c_{i_n}-c_{\pi_n}\|+\|c_{\pi_n}-c_j\|+\|c_j-c_{\pi}\|\\
                &\leq r_{\pi_n}+\|c_\pi-c_{j}\|+r_\pi,
            \end{split}
        \end{equation*}
    which implies that
        \begin{equation*}
                \|c_{\pi_n}-c_j\|-r_{\pi_n}>\xi_n-2r_{\pi_n}\geq\xi_n-2C_k\xi_{n+1}=0.
        \end{equation*}
    By definition of $r_{\pi_n}$, this implies $j\notin\pi_n$. We then have $\pi_n\cap X_n=\emptyset$, and thus $\Pi_{n+1}$ is indeed a partition; the construction is complete.

    Now since $(X_n)_{n\geq0}$ increases to $\range{1}{k}$, we always have $X_k=\range{1}{k}$. Therefore, setting $\Pi=\Pi_k$, from (A) and (B), the proposition holds with 
        \[\epsilon\coloneqq\xi_k=\frac{\delta}{C_k(2C_k)^k}>0.\]
\end{proof}

We now prove the sought proposition.

\begin{proof}[Proof of Proposition \ref{prop_ball} in the case $p<\infty$]
    Fix $\delta<\frac{1}{2}$. Let $\epsilon>0$ and $\Pi$ be given by Lemma \ref{lemma_eps_delta_grouping}. Up to choosing a smaller, $\epsilon$, we can assume $\epsilon<\delta$. Fix $\pi\in\Pi$. Let $x\in\R^d$ be such that $x\cdot e_d=0$ and $h\in\R$. Define $\alpha(\pi, x, h)\coloneqq c_\pi+x+he_d$ and observe that for $i\in\pi$, since $\|c_i-c_\pi\|\leq r_\pi$ we have
        \begin{equation*}
            \begin{split}
                \|\alpha(\pi, x, h)-c_i\|^p&=\|c_\pi-c_i+x+he_d\|^p\\
                &=|h|^p+\|c_\pi-c_i+x\|^p\\
                &=|h|^p+(\|c_\pi-c_i\|+\|x\|)^p\\
                &\leq |h|^p+(r_\pi+\|x\|)^p.
            \end{split}
        \end{equation*}
    We deduce then that
        \begin{equation}
            \label{eq_inside_cond}
            |h|^p+(r_\pi+\|x\|)^p<1\implies\alpha(\pi, x, h)\in \bigcap_{i\in\pi} B(c_i, 1).
        \end{equation}
    Similarly, if $j\in\range{1}{k}\setminus\pi$, since $\|c_j-c_\pi\|>r_\pi+\epsilon$, we have 
        \begin{equation*}
            \begin{split}
                \|\alpha(\pi, x, h)-c_i\|^p&=|h|^p+\|c_\pi-c_i+x\|^p\\
                &\geq |h|^p+(\|c_\pi-c_i\|-\|x\|)_+^p\\
                &> |h|^p+(r_\pi+\epsilon-\|x\|)_+^p,
          \end{split}
        \end{equation*}
    and thus we deduce that
        \begin{equation}
            \label{eq_outside_cond}
            h^p+(r_\pi+\epsilon-\|x\|)_+^p\geq1\implies\alpha(\pi, x, h)\notin \bigcup_{j\in\range{1}{k}\setminus\pi} B(c_j, 1).
        \end{equation}
    Now set
        \[h_\pi\coloneqq\left(1-\left(r_\pi+\frac{\epsilon}{2}\right)^p\right)^{1/p}\quad\text{and}\quad\alpha_\pi\coloneqq \alpha(\pi, 0,h_\pi).\]
    Note that $h_\pi$ is well defined as $r_\pi+\frac{\epsilon}{2}\leq\delta+\epsilon\leq 1$. Let us check that we can choose $\eta>0$ sufficiently small independently on the configuration so that $B(\alpha_\pi,\eta)\subset P_\pi$ using (\ref{eq_inside_cond}) and (\ref{eq_outside_cond}). Let $z\in\R^d$. There exists $x\in\R^d$ with $x\cdot e_d=0$ and $h\in\R$ such that $z=\alpha(\pi, x, h)$. Denote $\ell\coloneqq \|z-\alpha_h\|$. When $\ell\leq\frac{\epsilon}{2}$, we have
        \begin{equation}
            \label{eq_inside_cv}
            \begin{split}
                |h|^p+(r_\pi+\|x\|)^p&=1+|{h_\pi}|^p-|h|^p+(r_\pi+\|x\|)^p-(r_\pi+\epsilon/2)^p\\
                &\leq 1+|{h_\pi}|^p-|h|^p+\|x\|^p-(\epsilon/2)^p\\
                &\leq 1+p\ell(1+\ell)^{p-1}+\ell^p-(\epsilon/2)^p\\
                &\rightarrow 1-(\epsilon/2)^p<1 \quad\text{as}\quad \ell\rightarrow0,
            \end{split}
        \end{equation}
    where the first inequality comes from the convexity of $y\mapsto y^p$, uses that $\|x\|\leq\ell\leq\frac{\epsilon}{2}$, and the second from an application of the mean value Theorem using $|h_\pi|\leq 1$. Note that convergence only depends on $\epsilon$, which is fixed and not on the configuration of centers. Similarly, we have
    have
        \begin{equation}
            \label{eq_outside_cv}
            \begin{split}
                |h|^p+(r_\pi+\epsilon-\|x\|)_+^p &=1+|{h_\pi}|^p-|h|^p+(r_\pi+\epsilon-\|x\|)_+^p-(r_\pi+\epsilon/2)^p \\
                &\geq 1-\ell^p+(r_\pi+\epsilon-\|x\|)_+^p-(r_\pi+\epsilon/2)_+^p\\
                &\geq 1-\ell^p+(\epsilon-\ell)_+^p-(\epsilon/2)^p\\
                &\rightarrow 1+\epsilon^p-(\epsilon/2)^p>1\quad\text{as}\quad \ell\rightarrow0,
            \end{split}
        \end{equation}
    where the first and second inequalities come from the convexity of $y\mapsto y^p_+$. Here, the last convergence only depends on $\epsilon$. From (\ref{eq_inside_cv}) and (\ref{eq_outside_cv}) we deduce that we can choose $\eta=\eta(\epsilon)>0$ small enough so that $B(\alpha_\pi,\eta)\subset P_\pi$. Since $\eta$ only depends on $\epsilon$, the proof is complete.
\end{proof}

\subsection{The case $p=\infty$}

Throughout this subsection, we assume $p=\infty$. The goal is to show Proposition \ref{prop_ball}. We start with the following lemma, which addresses the case $d=1$.

\begin{Lemma}
    \label{lemma_dim_1} 
    Assume $d=1$. Let $(c_i)_{i\in\range{1}{k}}\in\R^k$. There exists a partition $\Pi$ of $\range{1}{k}$ such that for all $\pi\in\Pi$,
        $$|P_\pi|\geq\frac{2}{k!}.$$
    In particular, since all $(P_\pi)_{\pi\in\Pi}$ are segments, there exist real points $(\alpha_\pi)_{\pi\in\Pi}$ such that for all $\pi\in\Pi$, we have
        $$B\left(\alpha_\pi,\frac{1}{k!}\right)\subset P_\pi.$$
\end{Lemma}

\begin{proof}
    Without loss of generality, we can assume that the $(c_i)_{i\in\range{1}{k}}$ are ordered in a non-decreasing manner. First observe that if $I\subset\range{1}{k}$, then $I\neq\range{\min I}{\max I}$, then $P_I=\emptyset$. Indeed, if $i\in\range{\min I}{\max I}$, from the ordering if $x\in B(c_{\min I},1)\cap B(c_{\max I}, 1)$ then $x<c_{\min I}+1\leq c_i+1$ and $x>c_{\max I}-1\geq c_i+1$ and thus $x\in B(c_i,1)$, giving us $P_I\subset B(c_{\min I},1)\cap B(c_{\max I}, 1)\subset B(c_i,1)$ and finally $i\in I$ if $P_I\neq\emptyset$.

    We now construct by induction a sequence $(k_n)_{1\leq n\leq k}\in{\N^*}^k$ increasing to $k$ and a sequence $(\Pi_n)_{1\leq n\leq k}$ such that for all $n\in\range{1}{k}$, $\Pi_n$ is a partition of $\range{1}{k_n}$ and
        \begin{equation}
            \label{eq_piece_size}
            \min_{\pi\in\Pi_n}|P_\pi|\geq2\prod_{i=0}^{n-1}\frac{1}{k-i}.
        \end{equation}
    For the initialization, since $B(c_1, 1)$ has size $2$ and is composed of at most $k$ non-empty pieces, that are the $P_{\range{1}{a}}$ for $a\in\range{1}{k}$, at least one of them has size bigger than $\frac{2}{k}$. We set then $k_1$ to be the smallest index in $\range{1}{k}$ such that
        $$|P_{\range{1}{k_1}}|\geq\frac{2}{k}.$$
    By construction, the initialization is then verified, letting $\Pi_1=\{\range{1}{k_1}\}$.

    Now let $n\in\range{0}{k-1}$. If $k_n=k$, we set $\Pi_{n+1}\coloneqq\Pi_n$, $k_{n+1}=k_n$ and the heredity is clearly verified. If not, then $n\leq k_n<k$ and we want to select a range of indices containing $k_n+1$ to make the heredity work. To do that, let $\pi_n$ be the element of $\Pi_n$ containing $k_n$ and write
        \begin{equation}
            \label{eq_sum_pieces}
            \begin{split}
                \sum_{i=k_n+1}^k|P_{\range{i_0+1}{i}}|&=\sum_{\substack{1\leq a\leq b\leq k\\k_n+1\in\range{a}{b}}}|P_{\range{a}{b}}|-\sum_{\substack{1\leq a\leq b\leq k\\\{k_n,k_n+1\}\subset\range{a}{b}}}|P_{\range{a}{b}}|\\
                &=|B(c_{k_n+1}, 1)|-|B(c_{k_n}, 1)\cap B(c_{k_n+1}, 1)|\\
                &=|B(c_{k_n}, 1)|-|B(c_{k_n}, 1)\cap B(c_{k_n+1}, 1)|\\
                &=|B(c_{k_n+1}, 1)|-|B(c_{k_n}, 1)\cap B(c_{k_n+1}, 1)|\\
                &=|B(c_{k_n}, 1)\setminus B(c_{k_n+1}, 1)|\\
                &\geq |P_{\pi_n}|\\
                &\geq \min_{\pi\in\Pi_n}|P_\pi|\\
                &\geq2\prod_{i=0}^{n-1}\frac{1}{k-i}.
            \end{split}
        \end{equation}
    No observe that the sum in the LHS of (\ref{eq_sum_pieces}) has $k-k_n\leq k-n$ terms, so one of them is bigger than $\frac{1}{k-n}$ of the total sum. Using the inequality, we can thus choose $k_{n+1}$ to be the minimal index in $\range{k_n+1}{k}$ such that 
        $$|P_{\range{k_n+1}{k_{n+1}}}|\geq\frac{2}{k-n}\prod_{i=0}^{n-1}\frac{1}{k-i}=2\prod_{i=0}^{n}\frac{1}{k-i}.$$
    Then letting $\Pi_{n+1}\coloneqq\Pi_n\cup\{\range{k_n+1}{k_{n+1}}\}$, the heredity follows.
    
    To conclude, we set $\Pi=\Pi_k$. By construction $\Pi$ is a partition of $\range{1}{k}$ and 
        $$\min_{\pi\in\Pi}|P_\pi|\geq2\prod_{i=0}^{k-1}\frac{1}{k-i}=\frac{2}{k!}.$$
\end{proof}

\begin{Remark}
    In practice, it seems that the optimal constant in Lemma \ref{lemma_dim_1} is $\frac{2}{k}$ but is not reached by the algorithm we propose in the proof.
\end{Remark}

We now prove the main proposition in the case $p=\infty$.

\begin{proof}[Proof of Proposition \ref{prop_ball} in the case $p=\infty$]
    The idea is to decompose into the different dimensions. For each $s\in\range{1}{d-1}$, we apply Lemma \ref{lemma_dim_1} to get a partition $\Pi_s$ of $\range{1}{k}$ and real numbers $(\alpha_{s, \pi})_{\pi\in\Pi_s}$ such that for all $\pi\in\Pi_s$, and $x\in\R$,
    \begin{equation}
        \label{eq_one_dim_res}
        |x-\alpha_{s, \pi}|<\frac{1}{k!}\implies
            \begin{cases}
                \forall i\in\pi,~|x\cdot e_s-c_i\cdot e_s|<1,~\text{and}\\
                \forall j\in\range{1}{k}\setminus\pi,~|x\cdot e_s-c_j\cdot e_s|\geq1.
            \end{cases}
    \end{equation}
    We set
        \[\Pi\coloneqq\left\lbrace\bigcap_{s=1}^{d-1}\pi_s,~(\pi_s)_{1\leq s\leq d}\in\prod_{s=1}^{d-1} \Pi_s\right\rbrace.\]
    Let us check that for each element $\pi$ of this partition we can find a center $\alpha_\pi\in\R^d$ such that $B\left(\alpha_\pi,\frac{1}{2k!}\right)\subset P_\pi$. Let $\pi\in\Pi$. By definition there exists $(\pi_s)_{1\leq s\leq d}\in\prod_{s=1}^{d-1} \Pi_s$ such that $\pi=\bigcap_{s=1}^{d-1}\pi_s$. Set
        \[\alpha_\pi\coloneqq\frac{1}{2}e_d+\sum_{s=1}^{d-1} \alpha_{s,\pi_s}e_s.\]
    Now let $x\in B\left(\alpha_\pi, \frac{1}{2k!}\right)$ and $i\in\pi$. By construction, for each $s\in\range{1}{d-1}$ we have $i\in\pi_s$ and $|x\cdot e_s-\alpha_\pi\cdot e_s|<\frac{1}{k!}$ and thus from (\ref{eq_one_dim_res}), $|c_i\cdot e_s-x\cdot e_s|<1$. We deduce than $\|x-c_i\|<1$, that is $x\in B(c_i,1)$.
    
    Now let $j\in\range{1}{k}\setminus\pi$. By definition, there exists $s\in\range{1}{d-1}$ such that $j\notin\pi_s$. Since $|x\cdot e_s-\alpha_\pi\cdot e_s|<\frac{1}{k!}$ we have then from (\ref{eq_one_dim_res}) that $|x\cdot e_s-c_ij\cdot e_s|\geq1$. This implies $\|x-c_j\|\geq 1$, i.e.\ $x\notin B(c_j,1)$. 
    
    Finally, we have by construction that $\frac{1}{2}-\frac{1}{2k!}\geq0$, hence $x\in\mathbb H^+(0)$. We just proved that $x\in P_\pi$, the proof is complete.
\end{proof}

\subsection{Conclusion}

In this subsection, we use Proposition \ref{prop_ball} to derive Proposition \ref{prop_combinatorial}. We start with the following lemma, which will be useful.

\begin{Lemma}
    \label{lemma_shrink}
    Let $z, c\in\R^d$ and $r>0$, then
        \[B(z, r)\setminus B(c, 1)=\emptyset\implies B\left(z, \frac{r}{2}\right)\setminus B\left(c, 1+\frac{r}{2}\right)=\emptyset.\]
\end{Lemma}

\begin{proof}
    Let $c\neq x\in B(z, \frac{r}{2})$. Define
        \[x'\coloneqq x-\frac{r}{2}\frac{x-c}{\|x-c\|}.\]
    Observe that by the triangle inequality,
        \[\|x'-z\|\leq \|x-z\|+\frac{r}{2}<r,\]
    hence $x'\in B(z, r)$ and thus
        \[1\leq\|x'-c\|=\left(1-\frac{r}{2}\right)\|x-c\|=\|x-c\|-\frac{r}{2}\|x-c\|\leq\|x-c\|-\frac{r}{2}.\]
    Where the last inequality comes from $x\notin B(c,1)$. We deduce then $\|x-c\|\geq1+\frac{r}{2}$, that is $x\notin B(c,1+\frac{r}{2})$, the proof is complete.
\end{proof}

Finally, we prove the proposition.

\begin{proof}[Proof of Proposition \ref{prop_combinatorial}]
    From Proposition \ref{prop_ball}, there exists $\eta>0$ such that for every configuration, fixing 
        \[R_0\coloneqq\frac{4\kappa}{\eta},\]
    we can find a partition $\Pi$ of $\range{1}{k}$ and $(\alpha_\pi)_{\pi\in\Pi}\in(\R^d)^k$ such that for each $\pi\in\Pi$, 
        \[B(\alpha_\pi,\eta)\subset\bigcap_{i\in\pi} B^+\left(\frac{c_i}{R_0}, 1\right)\setminus\bigcup_{j\in\range{1}{k}\setminus\pi} B\left(\frac{c_j}{R_0}, 1\right).\]
    From Lemma \ref{lemma_shrink} since for all $\pi\in\Pi$ $j\in\range{1}{k}\setminus\pi$ we have $B(\alpha_\pi, \eta)\subset \R^d\setminus B(\frac{c_j}{R_0},1)$, then we have $B(\alpha_\pi, \frac{\eta}{2})\subset\R^d \setminus B(\frac{c_j}{R_0},1+\frac{\eta}{2})$. We then deduce that for all $\pi\in\Pi$,
        \[B\left(\alpha_\pi,\frac{\eta}{2}\right)\subset\bigcap_{i\in\pi} B^+\left(\frac{c_i}{R_0}, 1\right)\setminus\bigcup_{j\in\range{1}{k}\setminus\pi} B\left(\frac{c_j}{R_0}, 1+\frac{\eta}{2}\right).\]
    Now scaling everything by $R_0$ we deduce for all $\pi\in\Pi$,
        \begin{equation}
            \label{eq_ball_inclusion}
            B\left(R_0\alpha_\pi,\frac{R_0\eta}{2}\right)\subset\bigcap_{i\in\pi} B^+\left(c_i, R_0\right)\setminus\bigcup_{j\in\range{1}{k}\setminus\pi} B\left(c_j, R_0+\frac{R_0\eta}{2}\right).
        \end{equation}
    Finally observe that from triangular inequality for all $j\in\range{1}{k}$,
        \[B(c_j-\kappa e_d,\kappa+R_0)\subset B(c_j, R_0+2\kappa)=B\left(c_j, R_0+\frac{R_0\eta}{2}\right),\]
    where the last equality comes from the definition of $R_0$. Finally injecting in (\ref{eq_ball_inclusion}) we get
        \[B\left(R_0\alpha_\pi,\frac{R_0\eta}{2}\right)\subset\bigcap_{i\in\pi} B^+(c_i, R_0)\setminus\bigcup_{j\in\range{1}{k}\setminus\pi} B(c_j-\kappa e_d, \kappa+R_0).\]
    This shows that for all $\pi\in\Pi$, we have
        \[|P^{\kappa, R_0}_\pi|\geq\left|B\left(R_0\alpha_\pi,\frac{R_0\eta}{2}\right)\right|\eqqcolon a_0>0.\]
    The proof is complete.
\end{proof}

\bibliographystyle{amsalpha}
\bibliography{bibliography}

\end{document}